\documentclass[11pt,leqno]{article}
\usepackage{amsthm,amsfonts,amssymb,amsmath,eufrak,oldgerm} 
\usepackage{graphicx}
\usepackage{epsfig}
\renewcommand\d{\partial}
\renewcommand\a{\alpha}
\renewcommand\b{\beta}
\renewcommand\o{\omega}
\newcommand\s{\sigma}
\renewcommand\t{\tau}
\newcommand\R{\mathbb R}
\newcommand\C{\mathbb C}\newcommand\N{\mathbb N}
\newcommand\un{\underline}

\def\g{\gamma}
\def\t{\tau}
\def\O{\Omega}

\def\l{\lambda}
\def\L{\Lambda}
\def\eps{\varepsilon }
\def\e{\varepsilon}

\newcommand\kernel{\hbox{\rm Ker}}

\newcommand\br{\begin{rem}}
\newcommand\er{\end{rem}}
\newcommand\bp{\begin{pmatrix}}
\newcommand\ep{\end{pmatrix}}
\newcommand\be{\begin{equation}}
\newcommand\ee{\end{equation}}
\newcommand\ba{\begin{equation}\begin{aligned}}
\newcommand\ea{\end{aligned}\end{equation}}

\newcommand{\CalB}{\mathcal{B}}

\newcommand{\CalC}{\mathcal{C}}

\newcommand{\CalO}{\mathcal{O}}
\newcommand{\CalT}{\mathcal{T}}
\newcommand{\mA}{{\mathbb A}}
\newcommand{\RR}{{\mathbb R}}

\newcommand{\ZZ}{{\mathbb Z}}

\newcommand{\Id}{{\rm Id }}
\newcommand{\Range}{{\rm Range }}

\newcommand{\sgn}{\text{\rm sgn}}
\newtheorem{theo}{Theorem}[section]
\newtheorem{prop}[theo]{Proposition}
\newtheorem{cor}[theo]{Corollary}
\newtheorem{lem}[theo]{Lemma}
\newtheorem{defi}[theo]{Definition}
\newtheorem{ass}[theo]{Assumption}

\newtheorem{rem}[theo]{Remark}

\numberwithin{equation}{section}
\title{Transition to longitudinal instability of detonation waves is generically associated with Hopf bifurcation to time-periodic galloping solutions}

\author{\sc \small Benjamin Texier\thanks{
Universit\'e Paris Diderot (Paris 7), Institut de Math\'ematiques de Jussieu, UMR CNRS 7586;
texier@math.jussieu.fr:
Research of B.T.  was partially supported
under NSF grant number DMS-0505780.
} and
Kevin Zumbrun\thanks{Indiana University, Bloomington, IN 47405;
kzumbrun@indiana.edu:
Research of K.Z. was partially supported
under NSF grants no. DMS-0300487 and DMS-0801745.
 }}

\begin{document}

\maketitle

\begin{abstract}

We show that transition to longitudinal
instability of strong detonation solutions of reactive
compressible Navier--Stokes equations is generically
associated with Hopf bifurcation to nearby time-periodic  
``galloping'', or ``pulsating'', solutions, 
in agreement with physical and numerical observation.
In the process, we 
determine readily numerically verifiable stability and bifurcation conditions
in terms of an associated Evans function,
and obtain the first complete nonlinear stability
result for strong detonations of the reacting Navier--Stokes equations,
in the limit as amplitude (hence also heat release) goes to zero.
The analysis is by pointwise semigroup techniques
introduced by the authors and collaborators in previous works.
\end{abstract}

\tableofcontents
\bigbreak
\section{Introduction}

Motivated by physical and numerical observations of 
time-oscillatory ``galloping''
or ``pulsating'' instabilities of detonation waves
\cite{MT, BMR, FW, MT, ALT, AT, F1, F2, KS},
we study
stability and
Hopf bifurcation of viscous detonation waves,
or traveling-wave solutions of the reactive 
compressible Navier--Stokes equations.
This extends a larger program begun in 
\cite{ZKochel, LyZ1, LyZ2, JLW, LRTZ} 
toward the dynamical study of viscous combustion waves
using Evans function/inverse Laplace transform techniques
introduced in the context of viscous shock waves 
\cite{GZ, ZH, ZS, ZKochel, MaZ3}, 
continuing the line of investigation
initiated in \cite{TZ1, TZ2, SS, TZ3}
on bifurcation/transition to instability.

It has long been observed that transition to 
instability of detonation waves occurs in certain predictable
ways, with the archetypal behavior in the case of
longitudinal, or one-dimensional instability being transition
from a steady planar progressing wave $U(x,t)=\bar U(x_1-st)$
to a galloping, or time-periodic planar progressing wave
$\tilde U(x_1-st, t)$, where $\tilde U$ is periodic in the
second coordinate, and in the case of transverse, or multi-dimensional
instability, transition to more complicated ``spinning''
 or ``cellular behavior''; see \cite{KS, TZ1, TZ2}, and
references therein.

The purpose of this paper is, restricting to the one-dimensional
case, to establish this principle {rigorously}, 
arguing from first principles from the physical equations
that 
{\it transition to longitudinal instability of
detonation waves is {generically
associated with} Hopf bifurcation to time-periodic galloping solutions},
not only at the spectral but also at the full nonlinear level.
In the process, we establish the first full nonlinear stability results
for strong detonations of the reacting Navier--Stokes equations,
extending the sole previous result obtained
by Tan--Tesei \cite{TT} 
for the special class of 
initial perturbations with zero 
integral.

\subsection{The reacting Navier-Stokes equations}\label{equations}

The single-species reactive compressible Navier--Stokes equations, 
in Lagrangian coordinates, appear as \cite{Ch}
\begin{equation}\label{rNS}
\left\{ \begin{aligned}
 \d_t \tau - \d_x u & = 0,\\
 \d_t u + \d_x p  & = \d_x (\nu\tau^{-1} \d_x u),\\
 \d_t E + \d_x(pu) & = 
\d_x \big(qd \t^{-2} \d_x z+
\kappa \tau^{-1} \d_x T + \nu\tau^{-1} u \d_x u\big),\\
 \d_t z + k \phi(T) z & = \d_x (d \t^{-2} \d_x z),\\
\end{aligned}\right.
\end{equation}
where $\tau>0$ denotes specific volume, $u$ velocity, $E > 0$ total specific energy, and $0 \leq z \leq 1$ mass fraction of the reactant.

 The variable
 $$ U := (\tau,u,E,z) \in \R^4 $$
 depend on time $t \in \R_+,$ position $x \in \R,$ and 
parameters $\nu,\kappa, d,k,q,$ 
where $\nu >0$ is a viscosity coefficient, $\kappa> 0$ and $d>0$
are respectively coefficients of heat conduction and species diffusion, $k > 0$ represents the rate of the reaction, and $q$ is the heat release parameter,
with $q>0$ corresponding to an exothermic reaction and $q<0$ to
an endothermic reaction.  

In \eqref{rNS}, $T=T(\tau,e,z)>0$ represents temperature, 
$p=p(\tau,e,z)$ pressure, where the internal energy $ e > 0$ is defined through the relation
$$
E = e+ \frac{1}{2} u^2 + qz.
$$

In \eqref{rNS}, we assume a simple one-step, one-reactant, one-product reaction
 $$
  A \stackrel{k \phi(T)}{\longrightarrow} B, \qquad z := [A\,], \qquad [A\,] + [B\,] = 1.
 $$ 
 where $\phi$ is an ignition function. More realistic reaction models are described in \cite{GS2}. 
 
 In the variable $U,$  after the shift 
  $$x \to x - s t, \qquad s \in \R,$$
   the system \eqref{rNS} takes the form of a system of differential equations
  \begin{equation} \label{rNS-symb}
  \d_t U + \d_x (F(U)) = \d_x (B(U) \d_x U) + G(U),
  \end{equation}
 where
 $$ F := \left(\begin{array}{c} -u \\ p \\ p u \\ 0\end{array}\right) - s(\e) U, \qquad G := \left(\begin{array}{c} 0 \\ 0 \\ 0 \\ - k \phi(T) z \end{array}\right),$$ 
 and
 $$ B := \left(\begin{array}{cccc} 0 & 0 & 0 & 0 \\ 0 & \nu \tau^{-1} & 0 & 0 \\ \kappa \tau^{-1} \d_\tau T & - \kappa u \tau^{-1} \d_{e} T + \nu \t^{-1} u & \kappa \tau^{-1} \d_{e} T & \kappa \tau^{-1} (\d_z T - q \d_{e} T) + q d \t^{-2}\\ 0 &  0 &0 & d \t^{-2}\end{array}\right).$$ 
   
 The  characteristic speeds of the first-order part of  \eqref{rNS}, i.e.,
the eigenvalues of $\d_U F(U)$,  are
 \begin{equation} \label{char}
  \{ \quad \underbrace{- s - \s, \quad -s, \quad - s + \s}_{\mbox{fluid eigenvalues}},\underbrace{-s}_{\mbox{reactive eigenvalue}}\},
  \end{equation}
  where $\s,$ the sound speed of the gas, is
   $$
  \s := (p \d_{e} p - \d_\tau p)^{\frac{1}{2}} = \tau^{-1} (\Gamma(\Gamma + 1) e)^{\frac{1}{2}}.
  $$

 \subsection{Assumptions}

 We make the following assumptions: 
 
 \begin{ass} \label{idealgas} We assume a reaction-independent ideal gas equation of state,
\be\label{idealeos}
p=\Gamma \tau^{-1} e,
\qquad
T=c^{-1} e, 
\ee
where $c>0$ is the specific heat constant and
$\Gamma$ is the Gruneisen constant. 
\end{ass}

\begin{ass} \label{ignition} 
 The ignition function $\phi$ is smooth; it vanishes identically for $T \leq T_i,$ and is strictly positive for $T > T_i.$ 
\end{ass} 

\br\label{Ar}
A typical ignition function is given by the modified Arrhenius law
\be\label{Arlaw}
\phi(T)=Ce^{\frac{\mathcal{E}}{T-T_i}},
\ee
where $\mathcal{E}$ is activation energy.
\er

\br\label{eosrmk}
The specific choice \eqref{idealeos} is made for concreteness/clarity of exposition.
Our results 
remain valid
for any reaction--independent equation of state with
$p_\tau<0$, $p_e>0$, and $T_e>0$.\footnote{
An obvious exception is Lemma \ref{ex}, which depends on specific structure.}
With further effort, reaction-dependence should be treatable 
as well.
\er

\subsection{Coordinatizations} \label{sec:coord0}
 We let
 $$
  w := (u,E,z) \in \R^3, \qquad v := (\tau,u,E) \in \R^3.
 $$
 Then we have the coordinatizations
 $$ U = (v,z) = (\tau ,w).$$

In particular, Assumption \ref{idealgas} implies that in the $(\tau,w)$ coordinatization, $B$ takes the block-diagonal form 
 $$
 B = \left(\begin{array}{cc} 0 & 0 \\ 0 & b \end{array}\right),
 $$
where $b$ is full rank for all values of the parameters and $U;$ the system \eqref{rNS-symb} in $(\t,w)$ coordinates is 
 $$\left\{ \begin{aligned}
  \d_t \t  - s \d_x \t - J \d_x w & =  0, \\
  \d_t w + \d_x f(\t,w) & =  \d_x (  b(\t,w) \d_x w) + g(w),
  \end{aligned} \right. $$
 with the notation
 \begin{equation} \label{J}
 J := \left(\begin{array}{ccc} 1 & 0 & 0 \end{array}\right), \quad f := \left(\begin{array}{c} p \\ p u \\ 0 \end{array}\right) - s w, \quad g := \left(\begin{array}{c} 0 \\ 0 \\ - k \phi(T) z \end{array}\right).
 \end{equation}
 In the $(v,z)$ coordinatization, the system \eqref{rNS-symb} takes the form
$$ \left\{ \begin{aligned}
  \d_t v + \d_x f^\sharp(v,z) & = \d_x (b_1^\sharp(v) \d_x v + b^\sharp_2(v) \d_x z) \\
  \d_t z - s \d_x z + k\phi(T) z & = \d_x (d \t^{-2}\d_x z).
  \end{aligned} \right.
 $$
 where the flux is
$ f^\sharp = (- u - s \t, p - s u , pu - s E),$ and, under Assumption \ref{idealgas}, the diffusion matrices are
$$b_1^\sharp = \left(\begin{array}{ccc} 0 & 0 & 0 \\ 0 & \nu \t^{-1} & 0 \\ 0 & \t^{-1}(\nu - \kappa c^{-1}) u & \kappa \t^{-1} c^{-1} \end{array}\right), \qquad b_2^\sharp = \left(\begin{array}{c} 0 \\ 0 \\ q (d \t^{-2} - \kappa \t^{-1})\end{array}\right).$$

 Note that, in the $(v,z)$ coordinatization, the first component is a conservative variable, in the sense that $\d_t v$ is a perfect derivative, hence
\be\label{cons}
\int_\R  ( v(x,t)-
v(x,0)) \, dx\equiv 0,
\ee
for $v(t)-v(0)\in W^{2,1}(\R)$.
\subsection{Strong detonations} \label{strong}
We prove in this article stability and bifurcation results for viscous strong detonations of \eqref{rNS}, defined as follows:

\begin{defi} \label{def:strong} A one-parameter, right-going family of 
{\it viscous strong detonations} is a family $\{\bar U^\e\}_{\e \in \R}$ 
of smooth stationary solutions of \eqref{rNS-symb}, 
associated with speeds $s(\e),$ 
$s(\e) > 0,$ model parameters $(\nu,\kappa, d,k,q)(\eps)$ and 
ignition function $\phi^\eps$, 
with $\bar U^\eps, \phi^\eps, (s, \nu,\kappa, d,k,q)(\eps)$ depending
smoothly on $\eps$ in $L^\infty \times L^\infty \times \RR^6$, 
satisfying 
\be\label{profile}
\bar U^\e(x,t)=\bar U^\e(x),
\qquad \lim_{x\to \pm \infty} \bar U^\e(x)=U^\e_\pm,
\ee
connecting a burned state on the left to an unburned state on the right,
\be\label{burnedstates}
z^\e_- \equiv 0, \quad z^\e_+ \equiv 1,
\ee
with a temperature on the burned side above ignition temperature
\begin{equation} \label{burned-temp}
 T_-^\e > T_i,
 \end{equation}
and satisfying the Lax characteristic conditions
\be\label{Lax}
\s_- := \s(U^\e_-) > s > \s_+ := \s(U^\e_+),
\ee
uniformly in $\e.$
\end{defi}

\begin{figure}[t]
\begin{center}
\scalebox{0.5}{\input{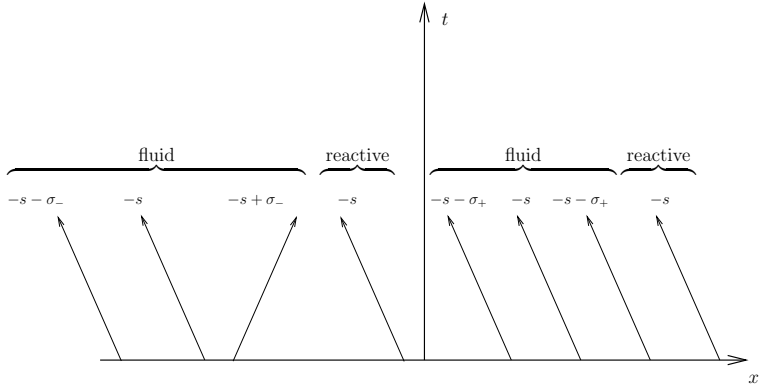}}
\caption{Characteristic speeds for strong detonations.}
\label{fig1}
\end{center}
\end{figure}

Consider a standing wave \eqref{profile}, $U = (\t,u,E,z),$ solution of \eqref{rNS-symb}, with endstates $U_\pm = (\t_\pm,u_\pm,E_\pm,z_\pm).$ It satisfies the linear constraint
$$
- s (\t - \t_-) = u - u_-,
$$
the system of ordinary differential equations 
 \begin{equation}\label{tw-ode} \left\{\begin{aligned}
 \nu \t^{-1}u' & = p - s u - (p -su)_-, \\
 \kappa \t^{-1} c^{-1} E' + \t^{-1}(\nu - \kappa c^{-1})u u'  & = pu - s E - (pu - s E)_-  \\ & \quad + (\kappa \t^{-1} c^{-1} - d \t^{-2}) q y, \\ z' & = y, \\
 d \t^{-2} y' & = - s y + k \phi(T) z,
   \end{aligned}\right.
   \end{equation}
   and the Rankine-Hugoniot relations
  \begin{equation}
  \label{RH} 
   \left\{ 
   \begin{aligned} 
   - s (\t_+ - \t_-) & = u_+ - u_-, \\ (p - s u)_+ & = (p - s u)_-, \\ (p u - s E)_+ & = (p u - s E)_-, \\  
    y_\pm   &= 0, \\ \phi(T_\pm) z_\pm & = 0,
    \end{aligned}
    \right.
    \end{equation}
   expressing the fact that $(u_\pm, E_\pm, 0, z_\pm)$ are rest points of \eqref{tw-ode}. 
 
  From \eqref{Lax} and \eqref{RH}, we note that the right endstate of a strong detonation satisfies 
 \begin{equation} \label{ignition-right}
  \phi(T^\e_+)=0,
  \end{equation}
  which, by Assumption \ref{ignition}, implies also 
 \begin{equation} \label{ignition-right-prime}
 \phi'(T^\e_+) = 0.
 \end{equation}

\begin{lem} \label{ex} Under Assumptions {\rm \ref{idealgas}}, {\rm \ref{ignition}}, if $q >0$ and $s$ is large enough with respect to $q,$ then for any $z_+ \in (0,1],$ there exists an open subset ${\cal O}_-$ in $\R^3,$ such that any left endstate $U_- = (v_-,0)$ with $v_- \in {\cal O}_-$ satisfies \eqref{burned-temp} and \eqref{Lax}, and is associated with a right endstate $U_+ = (v_+,z_+)$ satisfying $T_+ < T_i,$ \eqref{Lax} and \eqref{RH}. 
\end{lem}

 The existence of strong detonations was proved by Gasser and Szmolyan \cite{GS1} for small dissipation coefficients $\nu, \kappa$ and $d.$ We restrict throughout the article to strong detonations with left endstates as in the above Lemma. 
   
 \begin{rem}\label{smallq}
 In the small-heat-release limit $q\to 0,$ the equations in $(y,z)$ (in system \eqref{tw-ode}) are decoupled from the fluid equations; in particular, strong detonations converge to ordinary
nonreacting gas-dynamical shocks of standard Lax type, the existence of which 
has been established by Gilbarg {\rm \cite{G}.}    
 \end{rem}

 A consequence of Lemma \ref{ex} is that strong detonations converge exponentially to their endstates, a key fact of the subsequent stability and bifurcation analysis. 
 
\begin{cor}\label{profdecay} Under Assumptions {\rm \ref{idealgas}}, {\rm \ref{ignition}}, 
let $\{\bar U^\e\}_\e$
be a family of viscous strong detonations. There exist $C, \eta_0 >0,$ such that, for $k \geq 0$ and $j \in \{0,1\},$ 
\be\label{profdecayeq} \begin{aligned}
|\d_\e^j \partial_x^k (\bar U^\e-U^\e_-)(x)| & \le Ce^{-\eta_0|x|}, \qquad x < 0, \\
|\d_\e^j \partial_x^k (\bar U^\e-U^\e_+)(x)| & \le Ce^{-\eta_0|x|}, \qquad x > 0.
  \end{aligned}
\ee 
\end{cor}

In particular, $|(\bar U^\e)'(x)| \le Ce^{-\eta_0|x|},$ for all $x.$  
\begin{rem} \label{trans} In the ZND limit, strong detonations are transverse orbits of \eqref{tw-ode}, a result proved in Section 3.6 of {\rm \cite{LyZ2}}, 
following {\rm \cite{GS1}.}
\end{rem} 

 Lemma \ref{ex} and Corollary \ref{profdecay} are proved in Section \ref{sec:strong}.

\subsection{Structure of the equations and the profiles} \label{remarkson}
  
 System \eqref{rNS-symb}, seen as a system in $\t,w,$ satisfies
\begin{itemize}
\item[{\rm (A1)}] the convection terms in the equation in $\t$ are linear in $(\t,w);$ 
\item[{\rm (A2)}] the diffusion matrix $b$ is positive definite. 
\end{itemize}

 For strong detonation waves, the convection terms in \eqref{rNS} satisfy
  \begin{itemize}
  \item[{\rm (H1)}] The convection coefficient $s(\e)$ in the evolution equation in $\t$ is nonzero, uniformly in $\e.$
  \item[{\rm (H2)}] The spectrum of $\d_U F,$ given in \eqref{char}, is real, simple, and nonzero, uniformly in $\e.$
  \end{itemize}

  System \eqref{rNS-symb} satisfies the Kawashima dissipativity condition
 \begin{itemize}
 \item[{\rm (H3)}] For all $\e,$ for all $\xi \in \R,$
  $$ \qquad \Re \s\big(i \xi \d_U F(U^\e_\pm) - \xi^2 B(U^\e_\pm) + \d_U G(U^\e_\pm)\big) \leq - \frac{\theta\xi^2}{1+  \xi^2},$$
 \end{itemize}
  at the endstates $U^\e_\pm$ of a family of strong detonations. In (H3), $\s$ denotes spectrum of a matrix, and $\theta >0$ is independent of $\xi$ and $\e.$  To verify {\rm (H3)}, it suffices, by a classical result of \cite{ShK}, to check that \eqref{rNS-symb} has a symmetrizable hyperbolic-parabolic structure, and that the genuine coupling condition holds. These conditions are coordinates-independent, and easily checked in $(\t,u,e)$ coordinates.

 Finally, the assumption
 
 \begin{itemize}
 \item[{\rm (H4)}] Considered as connecting orbits of \eqref{tw-ode}, $\bar U^\e$ lie in a smooth one-dimensional manifold of solutions of \eqref{tw-ode}, obtained as a transversal intersection of the unstable manifold at $U^\e_-$ and the stable manifold at $U^\e_+,$ 
 \end{itemize}
 
\noindent  holds in the ZND limit, as stated in Remark \ref{trans}. Under (H4), in a vicinity of $\bar U^\e,$ the set of stationary solutions of \eqref{rNS-symb} with limits $U^\e_\pm$ at $\pm \infty$ is a smooth one-dimensional manifold, given by $\{ \bar U^\e(\cdot - c), \quad c \in \R\},$ and the associated speed $\e \to s(\e)$ is smooth.
 
 Conditions (A1)-(A2),(H0)-(H4) are the assumptions of \cite{TZ3} (where $G \equiv 0$), themselves a strengthened version of the assumptions of \cite{MaZ3}.
\subsection{The Evans function}\label{evans}

A central object in the study of stability of traveling
waves is the {\it Evans function} $D(\e,\cdot)$ (precisely defined in Section \ref{sec:Evans}), 
 a Wronskian
of solutions of the eigenvalue equation $(L(\e)-\lambda)U=0$
decaying at plus or minus spatial infinity \cite{AGJ}\footnote{
For applications of the Evans function to stability of viscous shock and detonation
waves, see, e.g., \cite{AGJ, GZ, ZS, ZKochel, LyZ1, LyZ2, LRTZ}.},
where the linearized operator $L$ is defined as
 \begin{equation} \label{def-l}
 L(\e)  := - \d_x(A \, \cdot) +  \d_x (B(\bar U^\e) \d_x \,\cdot) + \d_U G(\bar U^\e),
 \end{equation} 
 with the notation
 \begin{equation} \label{A}
 A := - \d_U F(\bar U^\e) + (\d_U B(\bar U^\e) \,\cdot) (\bar U^\e)'.
 \end{equation}

Recall the important result of \cite{LyZ2}:

\begin{prop} [\cite{LyZ2}, Theorem 4] \label{Lyprop}
Under Assumptions {\rm \ref{idealgas}} and {\rm \ref{ignition}}, let $\{\bar U^\e\}_\e$ be a one-parameter family of viscous strong detonation waves satisfying {\rm (H4).}

For all $\e,$ the associated Evans function has a zero of multiplicity
one at $\lambda=0:$ 
$$
 D(\e,0)=0, \qquad \mbox{and} \qquad D'(\e,0)\ne 0.
$$
\end{prop}

\begin{proof} By translational invariance, $D(\e,0) = 0,$ for all $\e.$ Generalizing similar results known for shock waves \cite{GZ, ZS},
there was established in \cite{ZKochel, LyZ1, LyZ2} the fundamental relation
\begin{equation}\label{der-Evans}
  D'(\e,0)=\gamma \delta.
  \end{equation}
 In \eqref{der-Evans}, $\gamma$ is a coefficient given as
a Wronskian of solutions of the linearized traveling-wave ODE
about $\bar U;$ transversality corresponds to $\g \neq 0.$ In \eqref{der-Evans}, $\delta$ is the Lopatinski determinant
$$
\delta:=\det\big(\begin{array}{cccc} r_1^- & r_2^- & r_4^- & (\begin{array}{ccc}\tau_+ -\t_- & u_+ - u_- & 
E_+ - E_- \end{array})^{\tt{tr}}\end{array}\big),
$$
(where $r_j^-$ denote the eigenvectors of
$\d_U F(U^\e_-)$ associated with outgoing
eigenvalues, $F$ as in \eqref{rNS-symb}, and ${\tt tr}$ denotes transverse matrix\footnote{This notation will be used throughout the article.}) determining hyperbolic stability of the Chapman--Jouget (square wave)
approximation modeling the detonation as a shock discontinuity. Hyperbolic stability corresponds to $\delta\ne 0.$
See \cite{ZKochel, LyZ1, LyZ2, JLW} for further discussion.
By (H4), $\g \neq 0,$ while $\delta\ne 0$ by direct
calculation comparing to the nonreactive (shock-wave) case.
\end{proof}

\begin{rem} The vectors $r_1^-,$ $r_2^-$ and $r_4^-$ correspond to outgoing modes to the left of $x =0,$ see Section {\rm \ref{sec:low-f}} and Figure {\rm \ref{fig4}}. (The fluid modes $r_j^-,$ $1 \leq j \leq 3,$ are ordered as usual by increasing characteristic speeds: $- s -\s_- < -s < 0 < - s + \s_-,$ so that $r_3^-$ is incoming.) 
\end{rem}

\subsection{Results}\label{results}

Let $X$ and $Y$ be two Banach spaces, and consider a traveling wave $\bar U$ solution of a general evolution equation.

\begin{defi} A traveling wave $\bar U$
is said to be $X \to Y$ \emph{linearly orbitally stable}
if, for any solution $\tilde U$ of the linearized equations
about $\bar U$ with initial data in $X,$ there exists a phase shift $\delta,$ 
such that $|\tilde U(\cdot,t)- \delta(t)\bar U'(\cdot)|_Y$ 
is bounded for $0\le t \le \infty.$ 

It is said to be $X \to Y$ \emph{linearly asymptotically orbitally stable} if it is $X \to Y$ linearly orbitally stable and if moreover
$|\tilde U(\cdot,t)- \delta(t)\bar U'(\cdot)|_Y\to 0$ as $t\to \infty$.
\end{defi}

\begin{defi} A traveling wave $\bar U$ is said to be $X \to Y$ \emph{nonlinearly orbitally stable } if,
for each $\delta>0$,  
for any solution $\tilde U$ of the nonlinear equations with
$|\tilde U(\cdot, 0)-\bar U|_X$ sufficiently small, there exists a phase shift $\delta,$ 
such that
$|\tilde U(\cdot,t)- \bar U(\cdot-\delta(t), t)|_Y\le \delta$ 
for $0\le t\le \infty.$

It is said to be $X \to Y$ \emph{nonlinearly asymptotically orbitally stable}
if it is $X \to Y$ nonlinearly orbitally stable and if moreover
$|\tilde U(\cdot,t) - \bar U(\cdot-\delta(t), t)|_Y\to 0$ as $t\to \infty$.
\end{defi}

\subsubsection{Stability}
Our first result, generalizing that of 
\cite{LRTZ} in the artificial viscosity case, is
a characterization of
linearized 
stability 
and a sufficient condition for nonlinear stability,
in terms of an Evans function
 condition.

\begin{theo}\label{stab}
Under Assumptions {\rm \ref{idealgas},} {\rm \ref{ignition},} let $\{\bar U^\e\}_\e$ be a one-parameter family of viscous strong detonation waves.

For all $\e,$ $\bar U^\e$ 
is $L^1\cap L^p\to L^p$ linearly orbitally stable if and only if, for all $\e,$
 \begin{equation} \label{D}
\mbox{the only zero of $D(\e,\cdot)$ in $\Re \l \ge 0$ is a simple zero at the origin.}
\end{equation}
If \eqref{D} holds, $\bar U^\e$ is $L^1 \cap H^3 \to L^1 \cap H^3$ linearly and nonlinearly orbitally stable, 
and $L^1\cap H^3\to L^p\cap H^3$ asymptotically orbitally stable, for $p > 1,$  
with 
\be\label{nonest}
|\tilde U^\e(\cdot, t)-\bar U^\e(\cdot -\delta(t))|_{L^p}
\le C|\tilde U^\e_0-\bar U^\e|_{L^1\cap H^3}(1+t)^{- \frac{1}{2}(1-\frac{1}{p})},\\
\ee
where $\tilde U^\e$ is the solution of \eqref{rNS-symb} issued from $\tilde U_0^\e,$ for some $\delta(\cdot)$ satisfying
\ba \nonumber
|\delta(t)| &\le C|\tilde U^\e_0-\bar U^\e|_{L^1\cap H^3},\\
|\dot \delta(t)| &\le C|\tilde U^\e_0-\bar U^\e|_{L^1\cap H^3}(1+t)^{-\frac{1}{2}}.
\nonumber \ea
\end{theo}

\br\label{smallII}
It is shown in \cite{LyZ2} that in the small heat-release
limit $q\to 0$, strong detonations are Evans stable if and only if
the limiting gas-dynamical profile (see Remark \ref{smallq}) 
is Evans stable: in particular, for shock (or equivalently detonation) 
amplitude sufficiently small \cite{HuZ2}.
\er

\begin{cor}\label{rigstab}
Under Assumptions {\rm \ref{idealgas},} {\rm \ref{ignition},} 
strong detonation profiles are linearly and nonlinearly orbitally stable 
(in the strong sense of \eqref{nonest}) in the limit as amplitude $|U_+-U_-|$
(hence also heat release $q$) goes to zero, with $U_-$ (or $U_+$) held fixed.
\end{cor}

Corollary \ref{rigstab} is notable as the first complete nonlinear stability
result for strong detonations of the reacting Navier--Stokes equations.
The only previous result on this topic, a partial stability result
applying to zero mass (i.e., total integral) perturbations, was obtained
by Tan and Tesei under similar, but 
more restrictive assumptions (in particular, 
for nonphysical Heaviside-type ignition function) in 1997. 

\subsubsection{Transition from stability to instability} \label{sec:transition}

\begin{theo}\label{spectralPH} Under Assumptions {\rm \ref{idealgas},} {\rm \ref{ignition},} let $\{\bar U^\e\}_\e$ be a one-parameter family of viscous strong detonation waves  satisfying {\rm (H4)}. 

Assume that the family of equations \eqref{rNS-symb} and profiles $\bar U^\e$ undergoes transition to instability
at $\eps=0$ in the sense that 
$\bar U^\eps$ is linearly stable for $\eps<0$ and linearly unstable for $\eps>0$.

Then, one or more pair of nonzero complex conjugate eigenvalues of
$L(\e)$ move from the stable (negative real part)
to the neutral or unstable (nonnegative real part) half-plane as $\eps$ passes from
negative to positive through $\eps=0$, while $\lambda=0$ remains
a simple root of $D(\eps,\cdot)$ for all $\e.$
\end{theo}

That is, transition to instability is associated with a Hopf-type
bifurcation in the spectral configuration of the linearized operator about
the wave.

\begin{proof}[Proof of Theorem {\rm \ref{spectralPH}.}]
By Theorem \ref{stab}, transition from stability to
instability must occur through the passage of a root
of the Evans function from the stable half-plane to the neutral or unstable half-plane.
However, 
Proposition \ref{Lyprop} implies that $D$ has a zero of multiplicity one at the origin,
for all 
$\e,$
and so no root can
pass through the origin.
It follows that transition to instability, if it occurs, must
occur through the passage of one or more nonzero complex conjugate
pairs $\lambda=\g \pm i\tau$, $\tau\ne 0$, from the stable half-plane ($\g < 0$ for $\e < 0$) to the neutral or unstable half-plane ($\g \geq 0$ for $\e \geq 0$).
\end{proof}

Our third result and the main object of this paper
is to establish, under appropriate nondegeneracy
conditions, that the spectral Hopf bifurcation configuration described
in Theorem \ref{spectralPH} is realized at the nonlinear level as a 
genuine bifurcation to time-periodic solutions.

 Given $k \in \N$ and a weight function $\o > 0,$ define the Sobolev space and associated norm
 \begin{equation} \label{wS} H^k_\o := \{ f \in {\cal S}'(\R), \, \o^{\frac{1}{2}} f \in H^k(\R)\}, \qquad \| f \|_{H^k_\o} := \| \o^{\frac{1}{2}} f \|_{H^k}. \end{equation}
 
 Let $\o \in C^2$ be a growing weight function such that, for some $\theta_0 > 0,$ $C > 0,$ for all $x,y,$
 \begin{equation} \left\{ \label{ass-o} \begin{aligned} 
  1 \leq \o(x) & \leq e^{\theta_0 (1 + |x|^2)^\frac{1}{2}}, \\  |\o'(x)| + |\o''(x)| & \leq C \o(x), \\
   \o(x) & \leq C \o(x-y) \o(y). \end{aligned} \right.
  \end{equation}

\begin{theo}\label{PH}
Under Assumptions {\rm \ref{idealgas},} {\rm \ref{ignition},} let $\{\bar U^\e\}_\e$ be a family of viscous strong detonation waves  satisfying {\rm (H4)}. 

Assume that the family of equations \eqref{rNS-symb} and profiles $\bar U^\e$ undergoes transition from linear stability to linear instability at $\e = 0.$

Moreover, assume that this transition is associated with
passage of a single complex
conjugate pair of eigenvalues of $L(\e),$ $\lambda_\pm (\eps)=\gamma(\eps)+i\tau(\eps)$
through the imaginary axis, satisfying
\be\label{nondeg}
\gamma(0)=0, \quad \tau(0)\ne0, \quad d\gamma/d\eps(0)\ne 0.
\ee

Then, given a growing weight $\o$ satisfying \eqref{ass-o} with $\theta_0$ sufficiently small, for $r \ge 0$ sufficiently small and $C>0$ sufficiently large,
there are $C^1$ functions
$r \to \e(r)$, $r \to T(r),$ with $\e(0)=0$, $T(0)=2\pi/\tau(0),$
and a $C^1$ family of time-periodic solutions  $\tilde U^{r}(x,t)$
of \eqref{rNS-symb} with $\e=\e(r)$, of period $T(r)$,
with
\be\label{expdecaybd}
C^{-1}r \le \|\tilde U^r-\bar U^\eps\|_{H^2_\o}\le Cr,
\ee
Up to translation in $x$, $t$,
these are the only time-periodic solutions nearby in $\|\cdot\|_{H^2_\o}$
with period $T\in [T_0, T_1]$ for any fixed $0<T_0<T_1<+\infty$.
\end{theo}

That is, transition to linear instability of viscous strong detonation
waves is ``generically'' (in the sense of \eqref{nondeg})
associated with Hopf bifurcation to time-periodic galloping solutions,
as asserted in the title of this paper.

The choices $\o \equiv 1$ and $\o = e^{\theta_0 (1 + |x|^2)^\frac{1}{2}}$ are allowed in \eqref{ass-o}, as well as $\o = (1 + |x|^2)^p,$ for any real $p > 0.$ In Theorem \ref{PH}, we need, in particular, $\theta_0 < \eta_0,$ where $\eta_0$ is as in Corollary \ref{profdecay}, so that the spatial localization given by \eqref{expdecaybd} is less precise than the spatial localization of the background profile $\bar U^\e.$ The smallness condition on $\theta_0$ is described in Remark \ref{rem-o}.

\subsubsection{Nonlinear instability}
We complete our discussion with the following straightforward
result verifying that the exchange of linear stability
described in Theorem \ref{PH}, as expected, 
corresponds to an exchange of nonlinear stability as well, the
new assertion being nonlinear {\it instability} for $\eps>0$.

\begin{theo}\label{instab}
Under the assumptions of Theorem {\rm \ref{PH},} the viscous strong
detonation waves $\bar U^\eps$ undergo a transition at $\eps=0$
from nonlinear orbital stability to instability; that is,
$\bar U^\eps$ is nonlinearly orbitally stable for $\eps<0$ 
and unstable for $\eps>0$.
\end{theo}

\subsection{Verification of stability/bifurcation conditions}\label{verification}
The above theory not only describes the nature of possible
bifurcation/exchange of stability
but characterizes its occurrence
in terms of corresponding spectral conditions involving zeros of the
Evans function of the linearized operator about the wave.
These may readily and efficiently be computed numerically \cite{HuZ1,BHRZ},
answering in a practical sense the question of whether or not such
transitions actually occur as parameters are varied
in any given compact region.

Much more can be said in certain interesting limiting cases.
It is shown in \cite{LyZ2} that in the small heat-release
limit $q\to 0$, strong detonations are Evans stable if and only if
the limiting gas-dynamical profile (see Remark \ref{smallq}) 
is Evans stable.
As 
noted in Corollary \ref{rigstab},
this implies in particular that strong
detonations are stable in the small-amplitude limit 
as the distance between endstates goes to zero with 
one endstate held fixed (forcing $q\to 0$ as well). 
For an ideal gas law \eqref{idealeos},
stability of large-amplitude detonations in the small heat-release
limit is strongly suggested by the recent asymptotic and numerical
studies of \cite{HLZ,HLyZ} indicating that viscous ideal gas shocks
are stable for arbitrary amplitudes.

A 
more interesting limit from the viewpoint of stability transitions
is the small-viscosity, or ZND limit as $\nu$, $\kappa$, $d$ 
go to zero.
Recall, \cite{GS1,GS2}, that in this limit, the viscous detonation
profile approaches an invscid profile composed of a smooth reaction
zone preceded by a shock discontinuity.
In \cite{Z4}, it has recently been shown that strong detonations
are stable in the ZND limit if and only if both the limiting
ZND profile and the viscous shock profile associated with its
component shock discontinuity satisfy spectral Evans stability
conditions like those developed here for viscous detonations.
Since viscous shocks for ideal gas law \eqref{idealeos} as just
mentioned are uniformly stable, this means that Evans stability of
rNS profiles reduces in the small viscosity limit
to Evans stability of the limiting ZND profile.

For ZND profiles, there is a wealth of numerical 
\cite{Er1,Er2,FW,S2,KS,BMR,BM,KS} and asymptotic \cite{F1,FD,B,BN,S1,Er4}
literature indicating that stability transitions do, and do often,
occur.
Indeed, a classic benchmark problem of Fickett and Woods \cite{FW}
tests numerical code for parameters $\Gamma=1.2$, $\mathcal{E}=50$, $q=50$ for
which transition to stability is known to occur as overdrive is varied
as a bifurcation parameter \cite{BMR}.
In multidimensions, a theorem of Erpenbeck \cite{Er3} 
gives a rigorous proof of instability for certain detonation types, 
occurring through high-frequency transverse modes (the only such
proof to our knowledge).
{\it In short, the evidence is overwhelming that spectral bifurcation occurs
in the ZND context, whence (by the results of \cite{Z4}) also
for \eqref{rNS}
for $\nu$, $\kappa$, $d$ sufficiently small.}

Together with these observations, the results of this paper answer 
definitively and positively
the fundamental question whether the reacting Navier--Stokes
equations are adequate to capture the bifurcation phenomena observed
for more than half a century in physical experiments \cite{FD,Er1}.
A very interesting problem would be to establish in one
dimension a rigorous spectral instability result for ZND 
analogous to that of Erpenbeck for multi-d, thus completing an
entirely mathematical proof;
in this regard, we mention that the analyses of \cite{BN,S1} 
appear to come very close.

\subsection{Discussion and open problems}\label{discussion}

This analysis in large part concludes the one-dimensional program
set out in \cite{TZ2}.  However, a very interesting
remaining open problem is to determine
linearized and nonlinear stability of the bifurcating 
time-periodic solutions, in the spirit of Section \ref{stabproof}. For a treatment in the shock
wave case with semilinear viscosity, see \cite{BeSZ}.
Likewise, it would be very interesting to carry out a
numerical investigation of the spectrum of the linearized
operator about detonation waves with varying physical parameters, 
as done in \cite{LS,KS} in the inviscid ZND setting,
but using the viscous methods of \cite{Br1, Br2, BrZ, BDG, HuZ1}
to treat the full reacting Navier--Stokes equations, in order
to determine the physical bifurcation boundaries.

Other interesting open problems are the extension to
multi-dimensional (spinning or cellular) bifurcations, as
carried out for artificial viscosity systems in \cite{TZ2},
and to the case of weak detonations (analogous to the case
of undercompressive viscous shocks; see \cite{HZ,RZ,LRTZ}).

The strong detonation structure considerably simplifies both stability and bifurcation arguments 
over what was done in \cite{LRTZ}. We remark that, at the expense of further complication, 
nonlinear stability of general (time-independent) combustion waves, including
also weak detonations and strong or weak deflagrations, 
may be treated by a combination
of the pointwise arguments of \cite{LRTZ} and \cite{RZ}.

We remark finally that the restriction to a scalar reaction variable is for simplicity only. Indeed, the results of this article (as well as the results of the article by Lyng and Zumbrun \cite{LyZ2} from which it draws) are independent 
of the dimension of the reactive equation, so long as the reaction satisfies an assumption of exponential decay of space-independent states (with temperature at $-\infty$ above the ignition temperature).

\medskip

{\bf Plan of the paper.} Lemma \ref{ex} and Corollary \ref{profdecay} are proved in Section \ref{sec:strong}. We give a detailed description of the low-frequency behavior of the resolvent kernel for the linearized equations in Section \ref{sec:resker}, following \cite{MaZ3}. In Section \ref{stabsection}, we prove Theorem \ref{stab}, while Section \ref{bifsection} is devoted to the proof of Theorem \ref{PH}. Finally, in Section \ref{instabproof}, we prove Theorem \ref{instab}.

\medbreak
{\bf Acknowledgement.}
Thanks to Bj\"orn Sandstede
and Arnd Scheel for their interest in this work
and for stimulating discussions on spatial dynamics 
and bifurcation in the absence of a spectral gap. Thanks to Gregory Lyng for pointing out reference \cite{Ch}. 
B.T. thanks Indiana University
for their hospitality during the collaborative visit in
which the analysis was carried out.
B.T. and K.Z. separately thank the
Ecole Polytechnique F\'ed\'erale de Lausanne for their hospitality
during two visits in which a substantial part of the
analysis was carried out.

\section{Strong detonations} \label{sec:strong}

\begin{proof}[Proof of Lemma {\rm \ref{ex}.}] Let $U_-$ be a given left endstate, with $z_- =0,$ satisfying \eqref{burned-temp} and \eqref{Lax}. We look for a right endstate $U_+,$ with $z_+ \in (0,1],$ that satisfies \eqref{RH}, \eqref{Lax}, and $T_+ < T_i.$ We note that \eqref{RH}(i) determines $u_+$ and that $T_+ < T_i$ entails \eqref{RH}(v). 

  The Rankine-Hugoniot relations in the $(\t_+,p_+)$ plane are     
 $$ \left\{\begin{aligned}
  p & = -s^2 \t + c_1 & \qquad \mbox{(R)}, \\
  p & = (c_0 - s \t (1 + \Gamma^{-1}))^{-1}(c_2 + s q z_+ + \frac{1}{2} s^3 \t^2 - s^2 c_0 \t) & \qquad \mbox{(H)},
  \end{aligned}\right.$$
  where (R) is the Rayleigh line, corresponding to \eqref{RH}(ii), (H) the Hugoniot curve, corresponding to \eqref{RH}(iii), and where
$$ c_0 := u_- + s \t_-, \qquad
  c_1 := p_- + s^2 \t_-, \qquad c_2:= (p_- u_- - s E_-) + \frac{1}{2} c_0^2 s$$
  depend on parameters $U_-$ and $s.$ The temperature and Lax constraints for both enstates are
  $$ \left\{\begin{aligned}
  \t_+ p_+ & < c \Gamma T_i < \t_- p_-  & \quad \mbox{(T)}_\pm,\\
  \t_+^{-1} p_+ & < (\Gamma +1)^{-1} s^2 < \t_-^{-1} p_- & \quad \mbox{(L)}_\pm. 
  \end{aligned}\right.$$
  
  We restrict to left endstates satisfying in the large $s$ regime
  \begin{equation} \label{reg1}   \t_- = O(1), \qquad p_- = 2 s^2 \Gamma^{-1} \t_- + \tilde p_-, \qquad u_- = s \tilde u_-,  \end{equation}
  with $\tilde u_- = O(1)$ and $\tilde p_- = O(1).$
  Under \eqref{reg1}, conditions $\mbox{(T)}_-$ and $\mbox{(L)}_-$ are satisfied as soon as $s$ is large enough. The Hugoniot curve takes the form 
  $$  \begin{aligned} p_{{\rm \tiny{H}}} = \left( \tilde u_- + \t_- - (1 + \Gamma^{-1}) \t \right)^{-1} & \Big(\frac{1}{2} s^3 (\t - (1 + 2 \Gamma^{-1}) \t_-)(\t - (1 - 2 \Gamma^{-1}) \t_- - 2 \tilde u_-) \\ & \qquad  + s q z_+\Big). \end{aligned}
   $$
   Assume that $\tilde u_-$ is such that
   \begin{equation}\label{reg-u} 
    \frac{\t_-}{1 + \Gamma^{-1}} < (1 - 2 \Gamma^{-1}) \t_- + 2 \tilde u_- < (1 + 2 \Gamma^{-1}) \t_-.
    \end{equation}
  For any such $\tilde u_-,$ any given $\t_-$ and any $q > 0,$ if $s$ is large enough then, for any $z_+ \in (0,1],$ the Hugoniot curve has two zeros $\un \t < \overline \t,$ with asymptotic expansions
  \begin{equation} \label{bart}
  \underline \t = (1 - 2 \Gamma^{-1}) \t_- + 2 \tilde u_- + O(s^{-2}).
  \end{equation}
  \begin{equation} \label{unt}
   \overline \t =  (1 + 2 \Gamma^{-1}) \t_- - s^{-2} \frac{\tilde p_- \tilde u_- + q z_+}{2 \Gamma^{-1} \t_- - \tilde u_-} + O(s^{-3}).
   \end{equation}
  If $s$ is large, by \eqref{reg-u}, $\t_0 < \un \t < \overline \t,$ where $\t_0 := c_0 s^{-1} (1 + \Gamma^{-1})^{-1}$ is the pole of (H).
   
  The Rayleigh line and the Hugoniot curve have at least one intersection point to the right of $\t_0$ if 
  $$ p_{\rm \tiny{R}}(\overline \t) < 0 < p_{\rm \tiny{R}}(\un \t).$$
  Under \eqref{reg-u}, the inequality $0 < p_{\rm \tiny{R}}(\un \t)$ holds, and $p_{\rm \tiny{R}}(\overline \t) < 0$ holds as well if in addition 
  \begin{equation} \label{6.9.1}
   \tilde p_- < - \frac{\tilde p_- \tilde u_- + q z_+}{2 \Gamma^{-1} \t_- - \tilde u_-}.
   \end{equation}

 Let $\t_+$ be an intersection point of (R) and (H) to the right of $\t_0.$ 
    Condition $\mbox{(T)}_+$ is satisfied if 
     \begin{equation} \label{alpha5}
      \t_+ = (1 + 2 \Gamma^{-1}) \t_- + s^{-2} \tilde \t_{+} + O(s^{-3}), 
      \end{equation}
    with 
    \begin{equation} \label{alpha6}
     (1 + 2 \Gamma^{-1}) \t_- (\tilde p_- -\tilde \t_+) < c \Gamma T_i.
     \end{equation}
     Condition $\mbox{(L)}_+$ is satisfied if 
    \begin{equation} \label{alpha4}
     (1 + 2 \Gamma^{-1}) \t_- < (1 + (1 + \Gamma)^{-1}) \t_+,
     \end{equation}
which holds under \eqref{alpha5}, if $s$ is large. 
   We plug the ansatz \eqref{alpha5} in the equation $p_{\rm \tiny{H}} = p_{\rm \tiny{R}},$ to find  
   \begin{equation} \label{tildet+}
   \tilde \t_+ = \frac{\Gamma \tilde p_-}{(1 + 2\Gamma^{-1})\t_-} \left((1 + \Gamma^{-1})(1 + 2 \Gamma^{-1}) - 1\right) + \frac{\Gamma q z_+}{(1 + 2\Gamma^{-1}) \t_-}.
     \end{equation}
  The intersection point $\t_+$ is an admissible right specific volume if $p_{\rm \tiny{H}}(\t_+) > 0$ and $p_{\rm \tiny{R}}(\t_+) > 0.$ These inequalities holds if
  \begin{equation} \label{more}
   \underline \t < \t_+ < (\a+ 1) \t_- + s^{-2} \tilde p_-.
   \end{equation}
  The inequalities \eqref{6.9.1}, \eqref{alpha6} and \eqref{more} are constraints on $\t_-, \tilde p_-,$ and $\tilde u_-.$ The lower bound on $\t_+$ in \eqref{more} is satisfied in the regime \eqref{reg1} if $s$ is large. If we let 
  $$\tilde p_- = \frac{- 2 \Gamma q z_+}{\t_-} + O(s^{-1}),$$
 then \eqref{6.9.1} holds. Finally, if $\t_-$ satisfies
    $$ 1  <     \frac{1}{4 \t_-} \left(3 + 2 \Gamma^{-1} - (1 + 2 \Gamma^{-1}) \t_-\right) < 1 + \frac{c T_i}{q z_+}.$$
   then the upper bound in \eqref{more} and \eqref{alpha6} hold as well. 
     
  The Rayleigh line (R), the Hugoniot curve (R) and the temperature (T) and Lax (L) constraints are pictured on Figure \ref{fig-RH}. The black dots represent the intersection points of (R) and (H). Note that (L) and (R) imply $\t_- < \t_+$ for a strong detonation, so that only the intersection point to the right to $\t_-$ is admissible. (The other intersection point corresponds to a deflagration, see for instance \cite{LyZ2}, Section 1.4.) 
\begin{figure}[t]
\begin{center}
\scalebox{.4}{\input{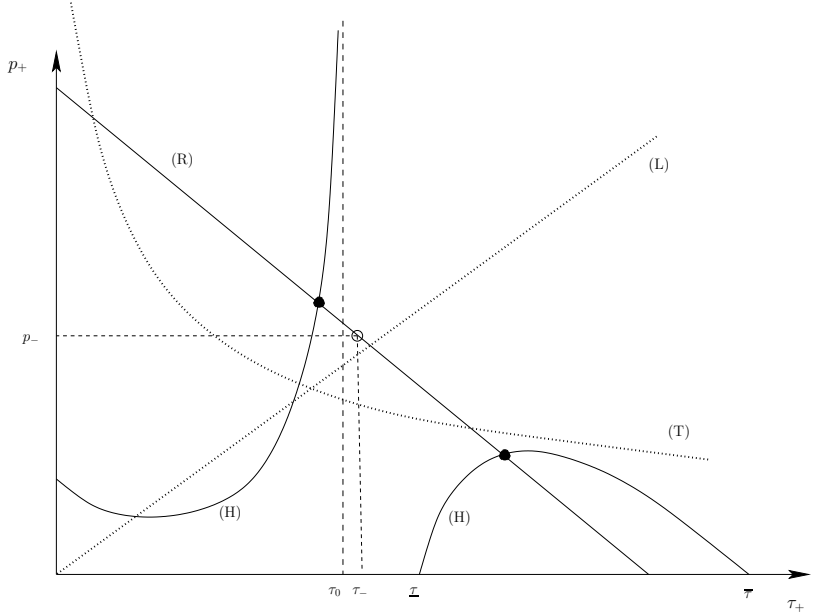}}
\caption{The Rankine-Hugoniot, Lax and temperature conditions.}
\label{fig-RH}
\end{center}
\end{figure}
 \end{proof}

 \begin{proof}[Proof of Corollary {\rm \ref{profdecay}.}] Rewrite \eqref{tw-ode} as $U' = {\mathfrak F}(\e,U).$ Let $U^\e_\pm$ be the endstates of a family of strong detonations. The linearized equations at $U^\e_\pm$ are governed by matrices
 \begin{equation} \label{lin-tw-ode}
 \d_U {\mathfrak F}(\e,\bar U^\e_\pm) = \left(\begin{array}{cc} a^f_\pm & * \\ 0 & a^r_\pm \end{array}\right).
  \end{equation}
 The block triangular structure is a consequence of Assumption \ref{idealgas}, \eqref{burnedstates}, \eqref{ignition-right}, and \eqref{ignition-right-prime}. Under Assumption \ref{idealgas}, the eigenvalues $\l$ of $a^f_\pm$
 $$ a^f_\pm := \left(\begin{array}{cc}  \nu \t^{-1} & 0 \\ \t^{-1}(\nu - \kappa c^{-1}) u & \kappa c^{-1} \t^{-1} \end{array}\right) \left(\begin{array}{cc} \d_u p - s - s^{-1} \d_\t p & \d_e p \\ u (\d_u p - s^{-1} \d_\t p) + p & u \d_e p - s \end{array}\right),$$
 satisfy
  \begin{equation} \label{e-af} \begin{aligned}
  \l^2 + \big(s \kappa c^{-1} \t_\pm^{-1} + s^{-1} \nu \t_\pm^{-1}(s^2 + (\d_\t p)_\pm)\big) \l + \kappa c^{-1} \nu \t^{-2}_\pm(s^2 - \sigma_\pm^2) = 0.\end{aligned}
  \end{equation}  
 The Lax condition \eqref{Lax} implies that the center subspace on both sides is trivial, that the eigenvalues of $a^f_-$ have opposite signs, and that the eigenvalues of $a^f_+$ are negative. The eigenvalues $\l$ of
 $$ a^r_\pm := \left(\begin{array}{cc} -s d^{-1} & k d^{-1} \phi(T_\pm)\\ 1 & 0\end{array}\right)$$
 satisfy
  $$ d \l^2 + s \l - k \phi(T_\pm) = 0.$$
 They are non zero and have distinct signs on the $-\infty$ side. On the $+\infty$ side, there is one negative eigenvalue, and a one-dimensional kernel. In particular, $U^\e_-$ is a hyperbolic rest point of the linearized traveling-wave ordinary differential equation, which implies \eqref{profdecayeq}(i) with $j= 0,$ by standard ODE estimates. However, the linearized traveling-wave equations at $U_+$ have a one-dimensional center subspace, which a priori precludes exponential decay \eqref{profdecayeq}. 
 
 From Lemma \ref{ex}, if $U_- \in {\cal O},$ then the system \eqref{tw-ode} has a line of equilibria that goes through $U^\e_+.$ Any center manifold of \eqref{tw-ode} at $U^\e_+$ contains all equilibria, so by
dimension count it must consist of equilibria.  Therefore, the $4$-dimensional stable center
manifold at $U^\e_+$ consists (again by dimension count) of the union of the stable
manifolds of all equilibria. Since solutions off of stable center manifold
do not stay for all time in small vicinity of center manifold, any
traveling-wave orbit must lie on the center-stable manifold, so lies
on the stable manifold of some equilibrium. Exponential decay, \eqref{profdecayeq}(ii), $j= 0,$ now follows
by the stable manifold theorem.

 To prove \eqref{profdecayeq} with $j= 1,$ consider now the traveling-wave ODE in $(U,\d_\e U).$ The rest points satisfy
  \begin{equation} \label{rest} {\mathfrak F}(\e,U) = 0, \qquad \d_\e {\mathfrak F}(\e,U) + \d_U {\mathfrak F}(\e,U) \d_\e U = 0.\end{equation}
  The kernel of $\d_U {\mathfrak F}(\e,U^\e_+)$ being one-dimensional, \eqref{rest} has a two-dimensional manifold of solutions. Let $(U^\e_+,V^\e_+)$ be such a rest point. The linearized equations at $(U^\e_+, V^\e_+)$ are governed by matrices
   $$ \left(\begin{array}{cc} \d_U {\mathfrak F}(\e,U^\e_+) & 0 \\ * & \d_U {\mathfrak F}(\e,U^\e_+)\end{array}\right),$$
   where the bottom left entry depends on second derivatives of ${\mathfrak F}.$ In particular, the linearized equations have a two-dimensional center subspace. We can thus argue as above that any center manifold consists entirely of equilibria, and that \eqref{profdecayeq}(ii) holds with $j=1.$ The proof of \eqref{profdecayeq}(i) with $j=1$ is similar. 
     \end{proof}

\section{Resolvent kernel and Green function bounds} \label{sec:resker}

 The linearized equations about a traveling wave $\bar U^\e$ solution of \eqref{rNS-symb} are 
  \begin{equation} \label{linearized}
   \d_t U = L(\e) U,
   \end{equation}
where $L(\e)$ is defined in \eqref{def-l}.
  The coefficients of $L(\e)$ are asymptotically constant at $\pm\infty.$ Let $L_\pm(\e)$ be the associated constant-coefficient, limiting operators: 
  $$L_\pm(\e) : = - A_\pm\d_x  + B_\pm \d_x^2 + G_\pm,$$
 with the notation $A_\pm := \d_U F(U^\e_\pm),$ $B_\pm := B(U^\e_\pm),$ $G_\pm := \d_U G(U^\e_\pm).$ 
 Let $L(\e)^*$ denote the dual operator of $L(\e).$ 
Its associated constant-coefficient, limiting operators are
 $ L_\pm(\e)^* = A_\pm^* \d_x + B^*_\pm \d_x^2 + G_\pm^*.$

\subsection{Laplace transform}

 Consider the Laplace transform of the linearized equations,
 \begin{equation} \label{laplace}
  (L(\e) - \l)  U = 0, \qquad \l \in \C, \quad x \in \R, \quad U(\e,x,\l) \in \C^4.
  \end{equation}
  Equation \eqref{laplace} can be cast as a first-order ordinary differential system in $\R^7,$
 \begin{equation} \label{ode}
  W' = \mA(\e,\l) W, \qquad \l \in \C, \quad x \in \R, \quad W(\e,x,\l) \in \C^7,
  \end{equation} 
  where the limits $\mA_\pm$ of $\mA$ at $\pm \infty$ are given by
   \begin{equation} \label{mA}
    \mA_\pm := \left(\begin{array}{ccc} s^{-1} \l & 0 & - s^{-1} J b_\pm^{-1} \\ 0 & 0 & b_\pm^{-1} \\ s^{-1} \l \d_\t f_{|\pm} & \l - \d_w g_{|\pm} & (\d_w f_{|\pm} - s^{-1} \d_\t f_{|\pm} J) b_\pm^{-1} \end{array}\right),
    \end{equation}
 where ${}|_{\pm}$ denotes evaluation at $U^\e_\pm,$ $b_\pm := b(U^\e_\pm),$ and $J$ is defined in \eqref{J}.

 Considered as an operator in $L^2(\R;\C^4),$ $L$ is closed, with domain $H^2$ dense in $L^2.$ Similarly, for all $\l,$ the operator $$\frac{d}{dx} - \mA(\l): \quad H^1(\R;\C^7) \subset L^2(\R;\C^7) \to L^2(\R;\C^7)$$ is closed and densely defined.

 The following straightforward Lemma gives a correspondence between \eqref{laplace} and \eqref{ode}.  
 \begin{lem} \label{corresp} Let $\l \in \C$ and $f = (f_1,f_2) \in L^2(\R;\C^1 \times \C^3).$ If the equation \begin{equation} \label{1lem} (L - \l) U = f
 \end{equation} has a solution $U =: (\t,w) \in H^2(\R;\C^1 \times \C^3),$ then $W := (\t, w, b w') \in H^1(\R;\C^7)$ satisfies 
 \begin{equation} \label{2lem} W' = \mA(\l) W + F,
 \end{equation} with $F = (f_1,0,f_2) \in L^2(\R;\C^7).$ Conversely, let $F = (f_1,0,f_2) \in L^2(\R;\C^7)$ and $\l \in \C.$ If $W = (w_1, w_2) \in H^1(\R;\C^7)$ satisfies \eqref{2lem}, then a solution in $H^2(\R;\C^4)$ to \eqref{1lem} with $f = (f_1,f_2)$ is given by $U = w_1.$ 
  \end{lem}
  
  Similarly, the dual eigenvalue equation
 \begin{equation} \label{laplace-dual}
  (L(\e)^* - \l)  \tilde U = 0, \qquad \l \in \C, \quad y \in \R, \quad \tilde U(\e,y,\l) \in \C^4,
  \end{equation}
  can be cast as 
\begin{equation} \label{ode-dual}
  \tilde W' = \tilde \mA(\e,\l) \tilde W, \qquad \l \in \C, \quad y \in \R, \quad \tilde W(\e,y,\l) \in \C^7,
  \end{equation} 
 where the limits $\tilde \mA_\pm$ of $\tilde \mA$ at $\pm \infty$ are given by 
 \begin{equation} \label{tildemA} \tilde \mA_\pm := \left(\begin{array}{ccc} - s^{-1} \l & 0 & s^{-1} \d_\t f_{|\pm}^{\tt tr} b^{\tt tr -1}_\pm \\ 0 & 0 & b^{\tt tr -1}_\pm \\ s^{-1} \l J^{\tt tr} & \l - \d_w g^{\tt tr}_{|\pm} & - (\d_w f_{|\pm}^{\tt tr} + s^{-1} J^{\tt tr} \d_\t f_{|\pm}^{\tt tr}) b^{\tt tr -1}_\pm \end{array}\right).
  \end{equation}
 A correspondence between \eqref{laplace-dual} and \eqref{ode-dual} holds, as in Lemma \ref{corresp}.

\subsubsection{The limiting, constant-coefficient equations}

 Associated with \eqref{laplace} and \eqref{laplace-dual} are the limiting, constant-coefficient eigenvalue equations
 \begin{equation} \label{e-eq-pm}
 (L_\pm(\e) - \l) U = 0, 
 \end{equation}
 and
 \begin{equation} \label{e-eq-pm-dual}
  (L_\pm(\e)^* - \l) \tilde U = 0.
 \end{equation}
  
  \begin{defi}[Normal modes] We call normal modes the solutions $(\l, U)$ of equations \eqref{e-eq-pm} and dual normal modes the solutions $(\l, \tilde U)$ of equations \eqref{e-eq-pm-dual}.
\end{defi}

 Associated with \eqref{ode} and \eqref{ode-dual} are the limiting, constant-coefficient differential equations
 \begin{equation} \label{limit-ode}
   W' = \mA_\pm(\e,\l) W,
   \end{equation}
  and
  \begin{equation} \label{dual-limit-ode}
  \tilde W' = \tilde \mA_\pm(\e,\l) \tilde W,
   \end{equation}
 where $\mA_\pm$ and $\tilde \mA_\pm$ are defined in \eqref{mA} and \eqref{tildemA}. 
  
  There is a correspondence between solutions of \eqref{e-eq-pm} and solutions of \eqref{limit-ode}:
     
  \begin{lem} \label{corresp2} If $(\l_0,U),$ $U =: (\t,w),$ is a normal mode, then $W := (\t, w, b w')$ solves \eqref{limit-ode} at $\l = \l_0.$ Conversely, if $W = (w_1, w_2) \in \C^4 \times \C^3$ solves \eqref{limit-ode} at $\l = \l_0,$ then $(\l_0, w_1)$ is a normal mode. In particular, 
   \begin{itemize}
   \item[{\rm (i)}] Eigenvalues $\mu$ of $\mA_\pm$ satisfy 
  \begin{equation} \label{n-modes}
  \det( - \mu A_\pm + \mu^2 B_\pm + G_\pm - \l) = 0,
  \end{equation}
and associated eigenvectors, satisfying $\mA_\pm(\l) W = \mu W,$ have the form $W = (U, w_2) \in \C^4 \times \C^3,$ with 
 \begin{equation} \label{vect-n-modes}
  U \in \ker( - \mu A_\pm + \mu^2 B_\pm + G_\pm - \l), \quad U =:(\t,w), \, w_2 := \mu b_\pm w,
 \end{equation}
 \item[{\rm (ii)}] Normal modes $(\l,U)$ satisfy
   \begin{equation} \label{12.5} 
    U = \sum_j e^{x \mu^\pm_j(\l)} U^\pm_j(x,\l),
   \end{equation}
   where the $\mu_j^\pm$ are eigenvalues of $\mA_\pm,$ and the $U_j^\pm$ are polynomials in $x.$
\end{itemize}
\end{lem}
 
 The correspondence between \eqref{e-eq-pm-dual} and \eqref{dual-limit-ode} is similar. In particular, eigenvalues $\tilde \mu$ of $\tilde \mA_\pm$ satisfy
\begin{equation} \label{d-n-modes} \det( \tilde \mu A^*_\pm + \tilde \mu^2 B^*_\pm + G^*_\pm - \l) = 0,
\end{equation}
associated eigenvectors, satisfying $\tilde \mA_\pm(\l) \tilde W = \tilde \mu \tilde W,$ have the form $\tilde W = (\tilde U, \tilde w_2) \in \C^4 \times \C^3,$ with 
 \begin{equation} \label{dual-vect-n-modes}
  \tilde U \in \ker( \tilde \mu A^*_\pm + \tilde \mu^2 B^*_\pm + G^*_\pm - \l), \quad \tilde U =: (\tilde \t, \tilde w), \quad \tilde w_2 := \mu b^{\tt tr}_\pm \tilde w,
 \end{equation}
  and dual normal modes satisfy
 \begin{equation} \label{dual-n-modes}
  \tilde U = \sum_j e^{y \tilde \mu^\pm_j(\l)} \tilde U^\pm_j(y,\l),
   \end{equation}
   where the $\tilde \mu_j^\pm$ are eigenvalues of $\tilde \mA_\pm,$ and the $\tilde U_j^\pm$ are polynomials in $y.$

  If $\tilde \mu(\l)$ is an eigenvalue of $\tilde \mA_\pm(\l),$ then $\overline{\tilde \mu(\l)} = - \mu(\bar \l),$ where $\mu(\bar \l)$ is some eigenvalue of $\mA_\pm(\bar \l).$ The matrices $A_\pm, B_\pm$ and $G_\pm$ having real coefficients, the complex conjugate of $\mu(\bar \l)$ is an eigenvalue of $\mA_\pm(\l).$ We can thus relate the solutions of \eqref{n-modes} and \eqref{d-n-modes} by
  $$ \tilde \mu(\l) = - \mu(\l).$$
 Note that $z_+^\e =1,$ $\phi(T_+^\e) = 0,$ and $\phi'(T_+^\e) = 0$ imply that the $v$ derivative of the coupling reaction term $k \phi(T) z$ vanishes when evaluated at $U^\e_\pm.$ In particular, in $(v,z)$ coordinates,
  $$ A_\pm = \left(\begin{array}{cc} \d_v f^\sharp_{|\pm} & \d_z f^\sharp_{|\pm} \\ 0 & -s \end{array}\right), B_\pm = \left(\begin{array}{cc} {b_1^\sharp}_{|\pm} & {b_2^\sharp}_{|\pm} \\ 0 & d \end{array}\right),  G_\pm = \left(\begin{array}{cc} 0 & 0 \\ 0 & - k \phi_{\pm} \end{array}\right),$$
  with the notation of Section \ref{sec:coord0}, ${}|_{\pm}$ denoting evaluation at $U^\e_\pm,$ and $\phi_\pm := \phi(T_\pm^\e),$ so that $\phi_+ = 0,$ while by \eqref{burned-temp} and Assumption \ref{ignition}, $\phi_- > 0.$ This triangular structure of the matrix $- \mu A_\pm + \mu^2 B_\pm + G_\pm$ allows a simple description of the solutions of \eqref{n-modes}. Indeed, \eqref{n-modes}, a polynomial, degree four equation in $\l,$ splits into the linear equation
 \begin{equation} \label{reactive-eq}  \mu s + \mu^2 d - k \phi_\pm - \l = 0,\end{equation}
 and the degree three equation
 \begin{equation} \label{fluid-eq}
  \det( - \mu \d_v f^\sharp_{|\pm} + \mu^2 {b_1^\sharp}_{|\pm} - \l) = 0.
  \end{equation}
  By inspection, \eqref{reactive-eq} is quadratic in $\mu,$ while \eqref{fluid-eq} is degree five in $\mu.$ Thus, the four solutions $\l(\mu)$ of \eqref{n-modes} correspond to seven eigenvalues $\mu(\l)$ of $\mA(\l).$

\subsubsection{Low-frequency behaviour of the normal modes} \label{sec:low-f}

 We describe here the behaviour of the normal modes in a small ball $B(0,r) := \{ \l \in \C, |\l| < r\}.$

\begin{defi}[Slow modes, fast modes] We call slow mode at $\pm\infty$ any family of normal modes $$\{ (\l, U(\l) \}_{\l \in B(0,r)}, \qquad  \mbox{for some $r > 0,$}$$ such that, in \eqref{12.5}, $\mu_j^\pm(0) = 0,$ for all $j.$  Normal modes which are not slow are called fast modes. We define similarly slow dual modes and fast dual modes, using \eqref{dual-n-modes}.  
\end{defi}

The solutions of \eqref{reactive-eq} are 
   \begin{equation} \label{r4}
   \mu^\pm_{4} =  \frac{1}{2d}(- s + (s^2 + 4 d(\l + k \phi_\pm))^{\frac{1}{2}}),
   \end{equation} 
 \begin{equation} \label{r5}
   \mu^\pm_{5} = - \frac{1}{2d}(s + (s^2 + 4 d(\l + k \phi_\pm))^{\frac{1}{2}});
   \end{equation} 
 they depend analytically on $\l$ (in the case of $\mu^4_+$ and $\mu_5^+,$ this is ensured by $s > 0,$ assumed in Definition \ref{def:strong}), and satisfy, for $\l$ in a neighborhood of the origin,
 \begin{equation} \label{DAS-r-slow-4+}
  \mu^+_{4} = s^{-1} \l  - s^{-3} d \l^2 + O(\l^3), \quad \mu^-_{4} > 0, \quad \mu^\pm_{5} < 0.
  \end{equation}
 Note that the inequality $\mu^-_{4} > 0$ is a consequence of $\phi_- > 0.$ By \eqref{dual-vect-n-modes}, the eigenvector of $\tilde \mA_+$ that is associated with $- \mu_{4}^+$ is
 \begin{equation} \label{dual-reactive-mode}
  L^+_{4} = \left(\begin{array}{c} \un \ell^{+}_4 \\ \mu_4^+ b_+^{\tt tr} \un \ell^+_4\end{array}\right) \in \C^4 \times \C^3, \quad \un \ell_4^+(0) = \ell_4^+,
  \end{equation}
  where
  \begin{equation} \label{l4}
  \ell^+_{4} := \left(\begin{array}{cccc} 0 & 0 & 0 & 1 \end{array}\right)^{\tt tr} 
  \end{equation}
  is the reactive left eigenvector of $A_+$ associated with the reactive eigenvalue of $A_+.$ We label $L^-_4,$ $L_5^\pm$ the eigenvectors of $\tilde \mA_\pm$ associated with $-\mu^-_4$ and $-\mu^\pm_5.$ By the block structure of $-\mu A_\pm + \mu^2 B_\pm + G_\pm,$ spectral separation of $\mu_4^-$ and $\mu_5^-$ (and of $\mu_4^+$ and $\mu_5^+$), the eigenvectors $L_4^\pm$ and $L_5^\pm$ are analytic in $\l,$ in a neighborhood of the origin (see for instance \cite{Kat}, II.1.4); in particular, 
  \begin{equation} \label{lm4} 
  \un \ell_4^+ = \ell_4^+ + O(\l), \quad \mu_4^+ b_+^{\tt tr} \un \ell^+_4 = O(\l).
  \end{equation} 
  
The solutions of \eqref{fluid-eq}, seen as an equation in $\l,$ are the eigenvalues of the matrix $ - \mu \d_v f^\sharp_{|\pm} + \mu^2 {b_1^\sharp}_{|\pm}.$ By \eqref{char} and the block structure of $A_\pm,$ we find that the spectrum of $\d_v f^\sharp_{|\pm}$ is 
$$\sigma(\d_v f^\sharp_{|\pm}) = \{ - s(\e) - \s_\pm, \quad -s(\e), \quad -s(\e) + \s_\pm \}.$$
 The eigenvalues of $\d_v f^\sharp_{|\pm}$ are distinct, hence, by Rouch\'e's theorem, the eigenvalues of $-\d_v f^\sharp_{|\pm} + \mu {b_1^\sharp}_{|\pm}$ are analytic in $\mu,$ for small $\mu,$ with expansions 
 \begin{equation} \label{fluid-modes0} \begin{aligned} 
  \l_1 & = s + \s_\pm + \b^\pm_{1} \mu + O(\mu^2), \\ 
  \l_2 & = s + \b^\pm_{2} \mu + O(\mu^2), \\ 
  \l_3 & = s - \s_\pm + \b^\pm_{3} \mu + O(\mu^2).
  \end{aligned}\end{equation}
 By (H3) (Section \ref{remarkson}), $\b^\pm_{j} > 0$ for all $j.$ Inversion of these expansions yields analytic functions $\mu^\pm_{j},$ called fluid modes, and defined in a neighborhood of the origin in $\C_\l:$
 \begin{equation} \label{fluid-modes} \begin{aligned}
 \mu^\pm_{1} & := (s + \s_\pm)^{-1} \l - (s + \s_\pm)^{-3} \b^\pm_{1} \l^2 + O(\l^3), \\
 \mu^\pm_{2} & := s^{-1} \l - s^{-3} \b^\pm_2 \l^2 + O(\l^3), \\
 \mu^\pm_{3} & := (s - \s_\pm)^{-1} \l - (s - \s_\pm)^{-3} \b^\pm_3 \l^2 + O(\l^3).
 \end{aligned}
 \end{equation}
 By \eqref{dual-vect-n-modes}, the eigenvectors of $\tilde \mA$ that are associated with these eigenvalues are
 \begin{equation} \label{dual-fluid-mode}
   L^\pm_{j}(\l) = \left(\begin{array}{c} \un \ell^{\pm}_{j} \\ \mu^\pm_j b_\pm^{\tt tr} \un \ell^\pm_j \end{array}\right) \in \C^4 \times \C^3, \quad \un \ell^{\pm}_j(0) = \ell_j^\pm, \quad 1 \leq j \leq 3,
   \end{equation}
   where the vectors $\ell^\pm_1,$ $\ell^\pm_2$ and $\ell^\pm_3$ are the left eigenvectors of $A_\pm$ associated with the fluid eigenvalues $-s - \s_\pm,$ $-s,$ and $-s+\s_\pm;$ they have the form
  \begin{equation} \label{lj}
   \ell^\pm_{j} := \left(\begin{array}{cccc} * & * & * & 0 \end{array}\right)^{\tt tr}, \qquad 1 \leq j \leq 3. \end{equation}
  The eigenvalues of $-\d_v f^\sharp_{|\pm} + \mu {b_1^\sharp}_{|\pm}$ being distinct, the associated eigenvectors are analytic as well, so that the $L^\pm_j,$ $1 \leq j \leq 3,$ are analytic in $\l;$ in particular, 
  \begin{equation} \label{lm} \un \ell_j^\pm = \ell_j^\pm + O(\l), \qquad \mu^\pm_j b_\pm^{\tt tr} \un \ell^\pm_j = O(\l). \end{equation}
  
Finally, the equation $\det (- \mu \d_v f^\sharp_{|\pm} + \mu^2 {b_1^\sharp}_{|\pm}) = 0$ has two non-zero solutions $\g^\pm_6, \g^\pm_7,$ corresponding to the remaining two (fast) modes, solutions of
 \begin{equation} \label{g67-eq}
  \kappa \t^{-2}_\pm c^{-1} s \nu \mu^2  + (\kappa c^{-1}(s^2 - \Gamma \t^{-2}_\pm e_\pm) + \nu s^2) \t^{-1}_\pm \mu + s(s^2 - \s_\pm^2) = 0.
  \end{equation}
 The Lax condition \eqref{Lax} implies that solutions of \eqref{g67-eq} are distinct and have small frequency expansions
  \begin{equation} \label{other-mode} \begin{aligned} \mu^\pm_{6} & = \g^\pm_{6} + O(\l), & \qquad \g_6^\pm < 0,  \\
   \mu^\pm_{7} & = \g^\pm_7 + O(\l), & \qquad \g_7^- > 0, \quad \g^+_7 < 0. 
   \end{aligned}
   \end{equation}
 We label $L^\pm_6$ and $L^\pm_7$ the eigenvectors of $\tilde \mA$ associated with $-\mu_6^\pm$ and $-\mu_7^\pm.$ Again, by spectral separation, $L^\pm_6$ and $L^\pm_7$ are analytic in $\l.$ 
  
 \begin{lem} \label{low-f} For some $r > 0,$ equations \eqref{dual-limit-ode} have analytic bases of solutions in $B(0,r),$
  \begin{equation} \label{def-v} \tilde {\cal B}^\pm := \{ \tilde V_{j}^\pm \}_{1 \leq j \leq 7}, \quad \tilde V_{j}^\pm := e^{- y \mu_{j}^\pm(\l)} L_{j}^\pm(\l),\end{equation}
  where the eigenvalues $\mu_j^\pm$ are given in \eqref{r4}, \eqref{r5}, \eqref{fluid-modes}, and \eqref{other-mode} and the eigenvectors associated with the slow modes are given in \eqref{dual-reactive-mode}, \eqref{lm4}, \eqref{dual-fluid-mode} and \eqref{lm}.
 \end{lem}
   
  \begin{proof} The above discussion describes analytic families $\mu_j^\pm,$ $L_j^\pm,$ such that the vectors $V_j^\pm$ defined in \eqref{def-v} are analytic solutions of \eqref{dual-limit-ode}. For $\l \neq 0,$ the eigenvalues $\mu_j^\pm$ are simple, so that the families $\tilde {\cal B}^\pm$ define bases of equations \eqref{dual-limit-ode}. By inspection of the expansions at $\l = 0,$ the families $\tilde {\cal B}^\pm$ define bases of equations \eqref{dual-limit-ode} at $\l = 0 $ as well. 
    \end{proof}
  
  The above low-frequency expansions of the eigenvalues show that
  \begin{itemize}
  \item[(i)] Equation $\tilde W' = \tilde \mA_-(\l) \tilde W$ has a 3-dimensional subspace of solutions associated with slow modes ($\mu^-_{j},$ $j=1,2,3$) and 4-dimensional subspace of solutions associated with fast modes ($\mu^-_{4}, \mu^-_{5},$ $\mu^-_{6}, \mu^-_{7}$);
  \item[(ii)] Equation $\tilde W' = \tilde \mA_+(\l) \tilde W$ has a 4-dimensional subspace of solutions associated with slow modes ($\mu^+_{j},$ $j=1,2,3,$ and $\mu_{4}^+$) and a 3-dimensional subspace of solutions associated with fast modes ($\mu^+_{5}, \mu^+_{6},$ $\mu^+_{7}$).
  \end{itemize}

\subsubsection{Description of the essential spectrum} \label{sec:ess}

  We adopt Henry's definition of the essential spectrum \cite{He}:
 
 \begin{defi}[Essential spectrum] \label{s_ess} Let ${\cal B}$ be a Banach space and $T: D(T) \subset {\cal B} \to {\cal B}$ a closed, densely defined operator. The essential spectrum of $T,$ denoted by $\s_{ess}(T),$ is defined as the complement of the set of all $\l$ such that $\l$ is either in the resolvent set of $T,$ or is an an eigenvalue with finite multiplicity that is isolated in the spectrum of $T.$
 \end{defi}
 
 By Lemma \ref{corresp2}, the matrix $\mA_\pm(\l)$ has a non trivial center subspace if and only if $\l \in {\cal C}_\pm,$
 $$ {\cal C}_\pm := \{ \l \in \C, \quad \det (- i\xi A_\pm  - \xi^2 B_\pm + G_\pm - \l) =0, \quad \mbox{for some $\xi \in \R$}\}.$$

The following Lemma can be found in \cite{He} (Theorem A.2, Chapter 5 of \cite{He}, based on Theorem 5.1, Chapter 1 of \cite{GK}): 

\begin{lem} \label{loc-s-ess} The connected component of $\C \setminus \big( {\cal C}_- \cup {\cal C}_+\big)$ containing real $+\infty$ is a connected component of the complement of the essential spectrum of $L(\e).$ 
\end{lem}

The reactive eigenvalues of $- i\xi A_\pm  - \xi^2 B_\pm + G_\pm$ are
 $$ \l = i \xi s - \xi^2 d - k \phi_\pm.$$
 For small $|\xi|,$ the fluid eigenvalues satisfy
 $$ \l = i \un \a \xi - \b \xi^2 + O(\xi^3), \qquad \un \a \in \R, \quad \b > 0,$$
 as described in Section \ref{sec:low-f}; for large $|\xi|,$ they satisfy
 \begin{align} \l & = - \xi^2 (\overline \a + O(\xi^{-1})) & \mbox{(parabolic eigenvalues)},\\
 \intertext{with $\overline \a \in \{ \nu \t^{-1}_\pm, \kappa c^{-1} \t^{-1}_\pm\},$ or}
  \l & = i s \xi + O(1) & \mbox{(hyperbolic eigenvalue)}.
 \end{align}
 This implies that the essential spectrum is confined to the shaded area in Figure \ref{fig-ess}, the boundary of which is the union of an arc of parabola and two half-lines. (The origin $\l = 0$ is an eigenvalue, associated with eigenfunction $(\bar U^\e)';$ the existence of bifurcation eigenvalues $\g(\e) \pm i \t(\e)$ is assumed in Theorem \ref{PH}, the proof of which is given in Section \ref{bifsection}.)
 
\begin{figure}[t]
\begin{center}
\scalebox{.4}{\input{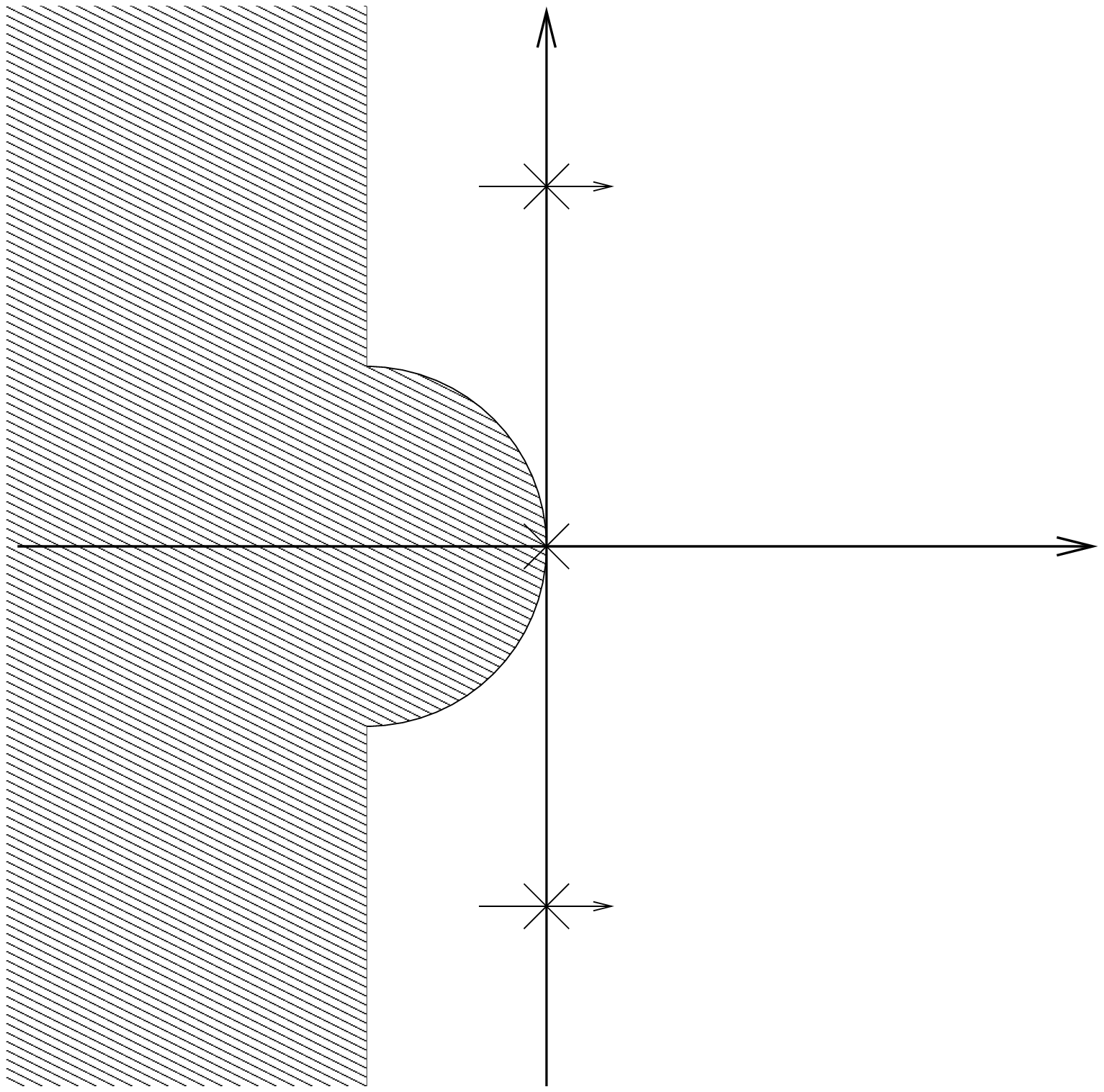}}
\caption{Spectrum of $L(\e).$}
\label{fig-ess}
\end{center}
\end{figure}

 \begin{rem} \label{ee} The essential spectrum, as given by Definition {\rm \ref{s_ess}}, is not stable under relatively compact perturbations (see {\rm \cite{EE}}, Chapter 4, Example 2.2); namely, a domain of the complement of the essential spectrum of a (closed, densely defined) operator $T$ is either a subset of the complement of the essential spectrum of $T + S,$ or is filled with point spectrum of $T + S,$ where $S$ is a relatively compact perturbation of $T.$
 \end{rem}
 
 \begin{rem} By the Fr\'echet-Kolmogorov theorem, $L$ is a relatively compact perturbation of $L_\pm.$ (This observation is the first step of the proof of Lemma {\rm \ref{loc-s-ess}}, see Henry {\rm \cite{He}.}) The pathology described in Remark {\rm \ref{ee}} does not occur in the right half-plane here, as we know by an energy estimate that if $\l$ is large and real, $\l \notin \s_p(L).$ 
 \end{rem}

\subsubsection{Gap Lemma and dual basis}

 Let $\L$ be the connected component of $\C \setminus \big( {\cal C}_- \cup {\cal C}_+\big)$ containing real $+\infty.$
 
\begin{defi}[Stable and unstable subspaces at $\pm \infty$] \label{stable-unstable-sp} Given $\l \in \L \cup B(0,r),$ $r$ as in Lemma {\rm \ref{low-f}}, denote by $S(\tilde \mA_\pm(\l))$ the stable subspace of $\tilde \mA_\pm(\l)$ (i.e., the subspace of generalized eigenvectors associated with eigenvalues with negative real parts) and by $U(\tilde \mA_\pm(\l))$ the unstable subspace of $\tilde \mA_\pm$ (i.e., the subspace of generalized eigenvectors associated with eigenvalues with positive real parts). We define similarly $S(\mA_\pm(\l))$ and $U(\mA_\pm(\l)).$
\end{defi}

 By definition of ${\cal C}_\pm,$ given $\l \in \L,$ the matrices $\mA_\pm(\l)$ do not have purely imaginary eigenvalues, so that $S(\mA_\pm(\l)) \oplus U(\mA_\pm(\l)) = \C^7,$ and $S(\tilde \mA_\pm(\l)) \oplus U(\tilde \mA_\pm(\l)) = \C^7,$ for all $\l \in \L.$ 
 
\begin{lem} \label{anal-bases} The vector spaces $S(\tilde \mA_\pm(\l))$ and $U(\tilde \mA_\pm(\l))$ have analytic bases in $\L.$ 
\end{lem}

\begin{proof} By simple-connectedness of $\L,$ the Lemma follows from a result of Kato (\cite{Kat}, II.4), that uses spectral separation in $\Lambda.$
 \end{proof}

\begin{cor} \label{continuation} Equations \eqref{dual-limit-ode} have analytic bases of solutions in $\L.$ 
\end{cor}

\begin{proof} Basis elements of the stable and unstable spaces defined in Definition \ref{stable-unstable-sp} are associated, through the flow of \eqref{dual-limit-ode}, with bases of solutions of \eqref{dual-limit-ode}. The matrices $\tilde \mA_\pm$ depending analytically on $\l,$ the flow of \eqref{dual-limit-ode} is analytic in $\l.$\end{proof}

 \begin{lem} \label{dcs} For $\l$ real and large, $\dim S(\mA_+(\l)) = \dim S(\mA_-(\l)) = 3.$ 
 \end{lem}
 
 \begin{proof} From Lemma \ref{corresp2}, $\mu$ is an eigenvalue of $\mA_\pm(\l)$ if and only if $\l$ is an eigenvalue of $ - \mu A_\pm + \mu^2 B_\pm + G_\pm.$ As in Section \ref{sec:ess}, for large $\mu,$ the eigenvalues of $ - \mu A_\pm + \mu^2 B_\pm + G_\pm$ are
 $s \mu + O(1)$ (hyperbolic mode) and $\nu \t_\pm^{-1} \mu^2 + O(\mu),$ $\kappa \t^{-1}_\pm c^{-1} \mu^2 + O(\mu),$ $d \mu^2 + O(\mu)$ (parabolic modes), $c^{-1}$ as in Assumption \ref{idealgas}. Inversion of these expansions gives three stable eigenvalues for both $\mA_-$ and $\mA_+.$ 
  \end{proof}
  
 \begin{rem} The above Lemma implies in particular that $\Lambda$ is a domain of consistent splitting, as defined in {\rm \cite{AGJ}.} (See also Section 3.1 of {\rm \cite{LyZ2}.})
 \end{rem}

 Given $\l \in \L,$ the flow of \eqref{dual-limit-ode} associates basis elements of $S(\tilde \mA_+(\l))$ with solutions of \eqref{dual-limit-ode} which are exponentially decaying as $t \to +\infty,$ and basis elements of $U(\tilde \mA_-(\l)$ with solutions which are exponentially decaying as $t \to -\infty.$ Similarly, the spaces $S(\tilde \mA_-(\l))$ and $U(\tilde \mA_+(\l)$ are associated with exponentially growing solutions, at $-\infty$ and $+\infty$ respectively.

\begin{defi}[Decaying and growing normal modes]  \label{extension} We call decaying dual normal mode at $\pm\infty$ any continuous family of dual normal modes $\{ \l, \tilde U(\l) \},$ $\l \in B(0,r),$ $r$ as in Lemma {\rm \ref{low-f}}, such that for all $\l \in \L \cap B(0,r),$ $\tilde U(\l)$ corresponds to a decaying solution of \eqref{dual-limit-ode} at $\pm \infty.$ Families of normal modes which are not decaying are growing. We define similarly decaying dual normal modes and growing dual normal modes.
\end{defi}

 By continuity of the eigenvalues and spectral separation in $\L,$ if for some $\l \in \L$ a continuous family of normal modes corresponds to a decaying (resp. growing) solution, then it corresponds for all $\l \in \L$ to a decaying (resp. growing) solution.

 By \eqref{Lax}, \eqref{DAS-r-slow-4+} and \eqref{fluid-modes}, $\mu_{1}^+,$ $\mu_{2}^+,$ $\mu^+_{3}$ and $\mu^+_{4}$ are growing (in the sense of Definition \ref{extension}) at $+\infty,$ while $\mu_{5}^+,$ $\mu_{6}^+$ and $\mu_{7}^+$ are decaying.
 
 Similarly, $\mu_{3}^-,$ $\mu_{5}^-$ and $\mu_{6}^-$ are growing, while $\mu_{1}^-,$ $\mu_{2}^-,$ $\mu_{4}^-$ and $\mu_7^-$ are decaying.

 The normal modes with which the characteristics of \eqref{rNS} are associated are pictured on Figures \ref{fig3} and \ref{fig4}. In particular, slow normal modes associated with incoming characteristics are growing.

\begin{figure}[t]
\begin{center}
\scalebox{.5}{\input{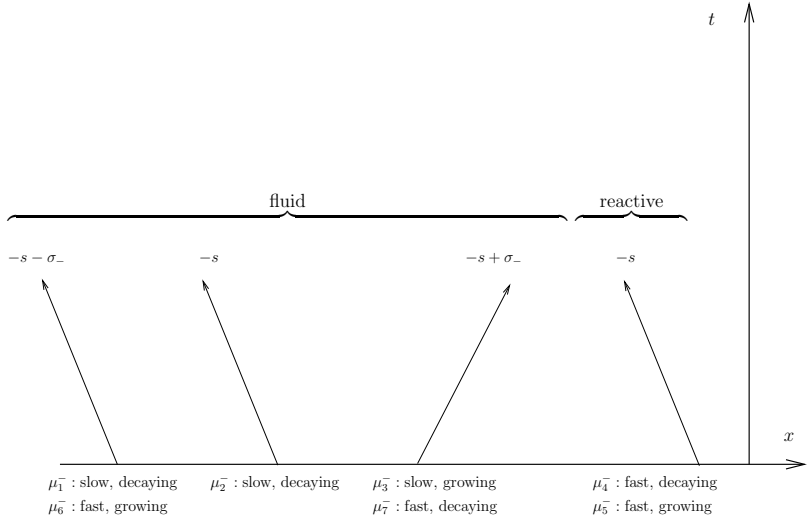}}
\caption{Normal modes on the $-\infty$ side.}
\label{fig4}
\end{center}
\end{figure}

\begin{figure}[t]
\begin{center}
\scalebox{.5}{\input{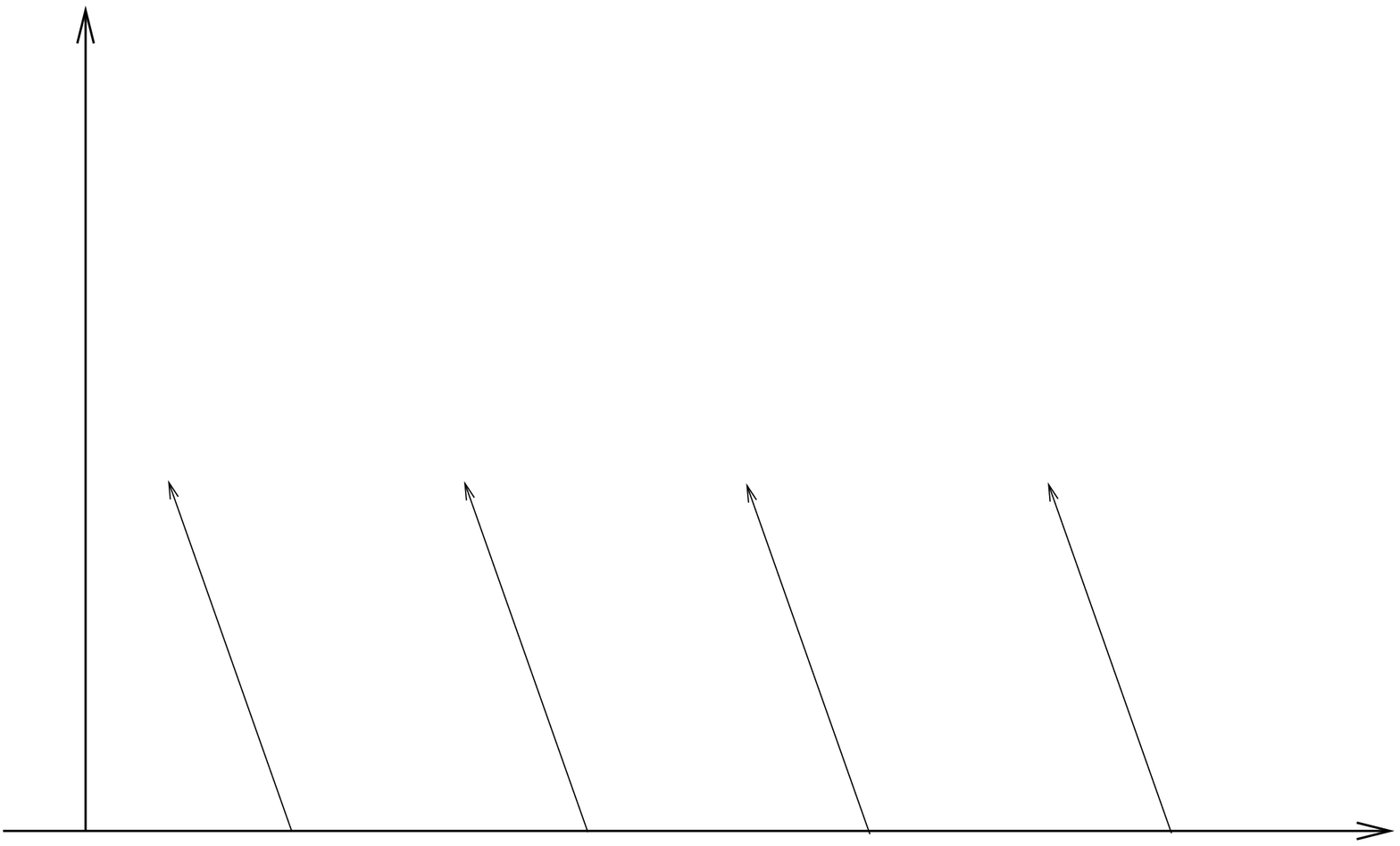}}
\caption{Normal modes on the $+\infty$ side.}
\label{fig3}
\end{center}
\end{figure}

\begin{defi}[Normal residuals] \label{norm-res} A map $(y,\l) \to \Theta^+(y,\l) \in \C^{7}$ defined on $[y_0,+\infty) \times B(0,r),$ for some $y_0 > 0,$ $r > 0,$ is said to belong to the class of normal residuals if it satisfies the estimates
 $$ |\Theta^+| \leq C , \qquad | \d_y \Theta^+| \leq C (|\l| + e^{- \theta |y|}).$$
  for some $\theta > 0$ and $C >0,$ uniformly in $y \geq y_0$ and $\l \in B(0,r).$  \end{defi} 

  We define similarly the class of normal residuals on $(-\infty,-y_0) \times B(0,r).$

\begin{lem}[Fast dual modes] \label{fast-modes} Equations \eqref{ode-dual} has solutions $$\tilde W^-_4, \tilde W_{5}^+, \tilde W_{6}^+, \tilde W_7^+ \quad \mbox{(growing)} \qquad \mbox{and} \quad \tilde W^-_5, \tilde W^-_6, \tilde W^-_7 \quad \mbox{(decaying)},$$
 which for $\l \in B(0,r),$ $r$ possibly smaller than in Lemma {\rm \ref{low-f},} satisfy   
     \begin{equation} \label{fastmodes}  \tilde W_j^\pm = e^{-y \mu_j^\pm(\l)} \big( L_j^\pm(0) +  e^{-\theta |y|} \tilde \Theta_{1j}^\pm + \l \tilde \Theta_{2j}^\pm\big), \qquad y \gtrless \pm y_0,\end{equation}
  for some $y_0 > 0$ independent of $\l,$ where the constant vectors $L_j^\pm(0)$ are defined in Section {\rm \ref{sec:low-f},} and $\tilde \Theta_{1j}^\pm,$ $\tilde \Theta_{2j}^\pm$ are normal residuals in the sense of Definition {\rm \ref{norm-res}.} 
  \end{lem}

 \begin{proof} With the description of the normal modes in Lemma \ref{low-f}, this is a direct application of the Gap Lemma (for instance in the form of Proposition 9.1 of \cite{MaZ3}).
 \end{proof}

\begin{lem}[Slow dual modes] \label{slow-modes} Equation \eqref{ode-dual} has solutions 
 $$\tilde W_{1}^-, \tilde W_{2}^- \quad \mbox{(growing)} \qquad \mbox{and} \quad \tilde W_{3}^-,\tilde W_{1}^+, \tilde W_2^+, \tilde W_3^+, \tilde W_4^+ \quad \mbox{(decaying)},$$
 which for $\l \in B(0,r),$ $r$ possibly smaller than in Lemma {\rm \ref{low-f}}, satisfy 
  \begin{equation} \label{gap-l+}
   \tilde W_{j}^\pm = e^{- y \mu_{j}^\pm(\l)} \big(L_j^{\pm}(0) + \l \tilde \Theta^{\pm}_j\big), \quad y \gtrless  \pm y_0,
  \end{equation}
 for some $y_0 > 0$ independent of $\l,$ where the constant vectors $L_j^\pm(0)$ are defined in Section {\rm \ref{sec:low-f},} and $\tilde \Theta_j^\pm$ are normal residuals. 
  \end{lem}
  
\begin{proof}
 The Conjugation Lemma (\cite{MeZ}; Lemma 3.1 of \cite{MaZ3}) implies that there exists a family of matrix-valued applications $\{\tilde {\bf \Theta}^+(\cdot,\l)\}_{\l \in B(0,r)},$ for some $r > 0$ possibly smaller than in Lemma \ref{low-f}, such that the matrix $\Id + \tilde {\bf \Theta}^+$ is invertible for all $\l$ and $y,$ the application $\tilde {\bf \Theta}^+$ is smooth in $y$ and analytic in $\l,$ with exponential bounds
  $$ | \d_\l^j \d_x^k \tilde {\bf \Theta}^+ | \leq C_{jk} e^{-\theta y}, \qquad \mbox{for some $\theta > 0,$ $C_{jk} > 0,$ for $y \geq y_0,$}$$
 for some $y_0 > 0,$ and such that any solution $\tilde W$ of \eqref{ode-dual} has the form 
 \begin{equation} \label{gap-l} \tilde W = (\Id + \tilde {\bf {\Theta}}^{+}) \tilde V^{+}, \qquad \mbox{for $y \geq y_0,$}
 \end{equation}
where $\tilde V^+$ is a dual normal mode, and, conversely, if $\tilde V^+$ is a dual normal mode, then $\tilde W$ defined by \eqref{gap-l} solves \eqref{ode-dual} on $y \geq y_0.$

 Equation \eqref{ode-dual} at $\l = 0$ has a four-dimensional subspace of constant solutions; let $\{ \tilde W^{0}_j\}_{1 \leq j \leq 4}$ be a generating family. The normal modes with which, through \eqref{gap-l}, the $\tilde W^{0}_j$ are associated are slow normal modes. Hence, by Lemma \ref{low-f}, there exist coordinates $c_{jk}$ such that
 $$
 \tilde W^{0}_j =  (\Id + \tilde {\bf \Theta}^{+}(\cdot,0)) \sum_{1 \leq k \leq 4} c_{jk} L_{k}^+(0), \qquad y \geq y_0,
 $$ 
 which implies in particular that the matrix
 $c := (c_{jk})_{1\leq j,k\leq 4}$ is invertible. Then, for $1 \leq j \leq 4,$
 $$ (\Id + \tilde {\bf \Theta}^{+}(\cdot,0))  L_{j}^+(0) = \sum_{1 \leq k \leq 4} (c^{-1})_{jk} \tilde W^{0}_k,$$
 in particular, $(\Id + \tilde {\bf \Theta}^{+}(\cdot,0)) L_{j}^+(0)$ is constant, hence, by exponential decay of $\tilde {\bf \Theta}^+,$ equal to $L_j^+(0).$ We can conclude that, for $1 \leq j \leq 4,$
  $$ \tilde W^+_j := (\Id + \tilde {\bf {\Theta}}^{+}) \tilde V_j^+$$
  (where $\tilde V_j^+$ is defined in Lemma \ref{low-f}) is a solution of \eqref{ode-dual} on $y \geq y_0,$ which can be put in the form \eqref{gap-l+}.
  
  The proof on the $-\infty$ side is based similarly on the decomposition of the fluid components of the $\tilde W^{0}_j$ onto the (fluid) dual slow modes $\tilde V_j^-,$ for $1 \leq j \leq 3.$ 
 \end{proof}

\subsubsection{Duality relation and forward basis}

 We use the duality relation, introduced in \cite{MaZ3},
  \begin{equation} \label{duality}  \tilde W^{\tt tr} {\cal S} W = 1
  \end{equation}
  that relates solutions $W$ of the forward equation \eqref{ode} with solutions $\tilde W$ of the adjoint equation \eqref{ode-dual} through the conjugation matrix in $(\t,w,bw')$ coordinates
  $$ {\cal S} := \left(\begin{array}{ccc} - A_{11} & -A_{12} & 0 \\ - A_{21} & -A_{22} & \mbox{Id}_{\C^3} \\ 0 & - \mbox{Id}_{\C^3} &  0 \end{array}\right),$$
  where $A$ is the convection matrix defined in \eqref{A}. Namely, $W$ is a solution of \eqref{ode} if and only if it satisfies \eqref{duality} for all solution $\tilde W$ of \eqref{ode-dual}, and conversely $\tilde W$ is a solution of \eqref{ode-dual} if and only if it satisfies \eqref{duality} for all solution $W$ of \eqref{ode}. (See Lemma 4.2, \cite{MaZ3}; note that the reactive term contains no derivative, hence does not play any role here.)
   
 Remark that there exist vectors $r_k^\pm$ such that
   \begin{equation} \label{a} \ell_j^\pm A_\pm r_k^\pm = - \delta_{jk}, \qquad 1 \leq j,k \leq 4.
   \end{equation}
Let $R_k^\pm$ be vectors of the form 
 \begin{equation} \label{r}
  R_k^\pm := \left(\begin{array}{c} r_k^\pm \\ * \end{array}\right) + e^{-\theta} \Theta_{1k}^\pm,
  \end{equation}
  where for $1 \leq k \leq 4,$ $r_k^\pm$ are given by \eqref{r}, and where $\Theta_{1k}^\pm$ are normal residuals. With the notation of Lemmas \ref{fast-modes} and \ref{slow-modes}, let 
  \begin{equation}
 \label{l} \bar L_j^\pm := \left\{ \begin{aligned} L_j^\pm(0) & \qquad \mbox{if $\mu_j^\pm$ is slow}, \\
   L_j^\pm(0) + e^{-\theta|y|} \tilde \Theta_{1j}^\pm & \qquad \mbox{if $\mu_j^\pm$ is fast.} \end{aligned} \right.
\end{equation}

\begin{lem}[Forward and dual basis] \label{fwd-dual} For some $r > 0$ and $y_0 > 0,$ 
  \begin{itemize}
   \item equation \eqref{ode} has analytic bases of solutions $\{ W^\pm_1, \dots, W^\pm_7\}_{\l \in \L \cup B(0,r)},$ for $y \gtrless \pm y_0;$ 
    \item equation \eqref{ode-dual} has analytic bases of solutions $\{\bar W^\pm_1, \dots, \bar W^\pm_7\}_{\l \in \L \cup B(0,r)},$ for $y \gtrless \pm  y_0,$
 \end{itemize}
  such that for $\l \in B(0,r),$ 
  \begin{eqnarray} W_j^\pm & = e^{x \mu_j^\pm(\l)}( R_j^\pm + \l \Theta_j^\pm), \qquad y \gtrless \pm y_0,\label{fwdbis} \\
  \label{dualbis} \bar W_j^\pm & = e^{-y \mu_j^\pm(\l)} (\bar L_j^\pm + \l \bar \Theta_j^\pm), \qquad y \gtrless \pm y_0,
  \end{eqnarray}
 where $R_j^\pm$ and $\bar L_j^\pm$ are defined in \eqref{r} and \eqref{l}, and $\Theta_j^\pm$ and $\bar \Theta_j^\pm$ are normal residuals; the fast forward modes $W_4^-$ and $W_7^+$ satisfy also
  \begin{equation} \label{u'} W_j^\pm(x,\l)  =  \left(\begin{array}{c} (\bar U^\e)'(x) \\ * \end{array}\right) + \l \un {\Theta}_j^\pm(x,\l), \qquad x \gtrless \pm y_0,\end{equation}
  where 
   $|{\un \Theta}_j^\pm| + |\d_x {\un \Theta}_j^\pm| \leq C e^{-\theta |x|},$
  for some $C, \theta > 0,$ uniformly in $\l \in B(0,r).$ \end{lem}

\begin{proof} Given a family $\{F_1,\dots,F_7\}$ of vectors in $\C^7,$ let $\mbox{col}(F_j)$ denote the $7 \times 7$ matrix $\mbox{col}(F_j) := \left(\begin{array}{ccc} F_1 & \dots & F_7\end{array}\right).$

 Let $y_0,$ $r,$ and $\tilde W_j^\pm$ as in Lemma \ref{fast-modes} and \ref{slow-modes}. For all $\l \in \L  \cup B(0,r),$ the families $\{ \tilde W_1^-,\dots,\tilde W_n^-\}$ and $\{ \tilde W_1^+,\dots,\tilde W_n^+\}$ are bases of solutions of \eqref{ode-dual}, on $y \leq - y_0$ and $y \geq y_0$ respectively. 
In particular, the $7 \times 7$ matrices
 $\tilde {\bf W}^{0\pm} := \mbox{col}(\tilde W_j^\pm)$ are invertible for all $\l \in B(0,r)$ and $y \gtrless  \pm y_0.$ Let
 \begin{equation} \label{def-fwd} {\bf W}^{0\pm} := ((\tilde {\bf W}^{0\pm})^{\tt tr} {\cal S})^{-1} =: \mbox{col}(W^{0\pm}_k).\end{equation}
 For the forward modes $W^{0\pm}_j$ defined in \eqref{def-fwd} to satisfy the low-frequency description  
 \begin{equation} \label{lf-fwd} W_j^{0\pm} = e^{x \mu_j^\pm(\l)}( R_j^{0\pm} + e^{-\theta|x|} \Theta_{1j}^{0\pm} + \l \Theta_{2j}^{0\pm}), \quad y \gtrless \pm y_0,\end{equation}
 where $R_j^{0\pm}$ are constant vectors and $\Theta_{\star j}^{0\pm}$ are normal residuals, it suffices, by \eqref{duality}, that the matrices ${\bf R}^{0\pm} := \mbox{col}(R_j^{0\pm})$ and ${\bf \Theta}_\star^{0\pm} := \mbox{col}(\Theta_{\star j}^{0\pm})$ satisfy
       \begin{eqnarray}
   \label{fwd1} {\bf L}^{\tt tr} {\cal S} {\bf R}^0 & = & \Id_{\C^7}, \\ \label{fwd2}
    ( {\bf L} + e^{-\theta|x|} \tilde {\bf \Theta}_{1} )^{\tt tr} {\cal S} {\bf \Theta}_{1}^0 & = & - \tilde {\bf \Theta}_{1}^{\tt tr} {\cal S} {\bf R}^0, \\ \label{fwd3}
      ({\bf L} + e^{-\theta|x|} \tilde {\bf \Theta}_{1}^0 + \l \tilde {\bf \Theta}_2)^{\tt tr}  {\cal S} {\bf \Theta}_2^0 & = & -\tilde {\bf \Theta}_2^{\tt tr} {\cal S} {\bf R}^0,
     \end{eqnarray}
  where ${\bf L}^\pm := \mbox{col}(L_j^\pm(0))$ and $\tilde {\bf \Theta}_\star^\pm := \mbox{col}(\bar \Theta_{\star j}^\pm)$ appear in the low-frequency description of the $\tilde W_j^\pm.$ In \eqref{fwd1}-\eqref{fwd3}, the $\pm$ exponents are omitted.  
 The matrices ${\bf L}^\pm$ being invertible, \eqref{fwd1} (with $+$ or $-$) has a unique solution, and, for $y_0$ large enough and $r$ small enough, equations \eqref{fwd2} and \eqref{fwd3} have unique solutions in the class of normal residuals.  Note that for $1 \leq j,k \leq 4,$ equation \eqref{fwd1} reduces to \eqref{a}, up to exponentially decaying terms, so that the vectors $R^{0\pm}_k$ have the form \eqref{r}.

 Remark now that $(\bar U^\e)'$ satisfies $L(\e) (\bar U^\e)' = 0,$ and decays at both $-\infty$ and $+\infty,$ hence $(\bar U^\e)'$ is associated with decaying fast normal modes; by Lemma \ref{fast-modes}, there exist constants $c_j^\pm,$ such that
  \begin{equation} \label{zero-E} \left(\begin{array}{c} (\bar U^\e)'(y) \\ * \end{array}\right) = {c_4}^- {W_4^{0-}}_{|\l = 0} = \sum_{5 \leq j \leq 7}  c_j^+ {W_j^{0+}}_{|\l = 0}.\end{equation}
 We may assume, without loss of generality, that $ c_7^+ \neq 0.$ Let now
  \begin{eqnarray} \nonumber
   {\bf W}^- & := & \big(\begin{array}{ccccccc} W_1^{0-} & W_2^{0-} & W_3^{0-} & c_4^- W_4^{0-} & W_5^{0-} & W_6^{0-} & W_7^{0-}\end{array}\big), \\ \nonumber 
   {\bf W}^+ & := & \big(\begin{array}{ccccccc} W_1^{0+} & W_2^{0+} & W_3^{0+} & W_4^{0+} & W_5^{0+} & W_6^{0+} & \sum_{j=5}^7  c_j^+ W_j^{0+} \end{array}\big),
  \end{eqnarray}
  and ${\bf W}^\pm =: \mbox{col}(W_j^\pm).$ These forward modes satisfy \eqref{fwdbis} and \eqref{u'}. Let finally
  $ \bar {\bf W}^{\pm {\tt tr}} := ({\cal S} {\bf W}^\pm)^{-1} =: \mbox{col}(\bar W_j^\pm),$ so that, in particular, the slow modes of $\tilde {\bf W}^{0\pm}$ and $\bar {\bf W}^\pm$ coincide. We can prove as above that the low-frequency description \eqref{lf-fwd} of the forward modes carries over to the dual modes through the duality relation, so that \eqref{dualbis} is satisfied.
\end{proof}

\subsubsection{The resolvent kernel} \label{sec:resker2}

 Let 
  $$L^2(\Omega,{\cal D}'(\R)) := \{ \phi \in {\cal D}'(\Omega \times \R), \quad \mbox{for all $\varphi \in {\cal D}(\R),$ $\langle \phi, \varphi \rangle \in L^2(\Omega)$}\}$$
   A linear continuous operator $T: L^2(\R) \to L^2(\R)$ operates on $L^2(\R,{\cal D}'(\R)),$ by
  $ \langle T \phi, \varphi \rangle := T \langle \phi, \varphi \rangle.$
   Let $\t_{(\cdot)} \delta \in L^2(\R, {\cal D}'(\R))$ be defined by $\langle \t_{x} \delta, \varphi \rangle = \varphi(x),$ for all $x \in \R.$

\begin{defi}[Resolvent kernel] Given $\l$ in the resolvent set of $L(\e),$ define the resolvent kernel ${\cal G}_\l$ of $L(\e)$ as an element of $L^2(\R_x,{\cal D}'(\R_y))$ by  
 $$ {\cal G}_\l := (L(\e) - \l)^{-1} \t_{(\cdot)} \delta.$$
  \end{defi}

 Given $y \in \R,$ let $s_y = \sgn(y),$ and 
 $$ \tilde {\bf D}(y) := \{ j, \quad \mbox{$\tilde \mu_j^{s_y}$ slow and decaying} \},$$
 so that 
 $$\tilde {\bf D}(y) = \{3\}, \quad \mbox{if $y < 0,$} \qquad \tilde {\bf D}(y) = \{1,2,3,4\}, \quad \mbox{if $y > 0.$}$$   
 Given $x,y \in \R,$ let ${\bf D}(x,y)$ be the set of all $(j,k)$ such that for all $x,y,$ for $\Re \l > 0$ and $|\l|$ small enough, $\Re(\mu_j^{s_x} x - \mu_k^{s_y} y) < 0,$ that is, 
 $$ \begin{aligned} {\bf D}(x,y) & = \{ (j,k), \quad \mbox{$\mu_j^{s_x}$ and $\tilde \mu_k^{s_y}$ slow and decaying} \} \\
  & \quad \bigcup\, \{ (j,j), \quad s_x = s_y, \quad |y| < |x|, \quad \mbox{$\mu_j^{s_x}$ slow and decaying}\} \\
  & \quad \bigcup\, \{ (j,j), \quad s_x = s_y, \quad |x| < |y|, \quad \mbox{$\tilde \mu_j^{s_y}$ slow and decaying}\},
   \end{aligned}
   $$
  so that  
 $$
  {\bf D}(x,y) := \left\{\begin{array}{lc} \{ (1,1), (2,2), (1,3), (2,3)\}, & \mbox{$x \leq y \leq 0$}, \\ \{ (1,3), (2,3), (3,3) \}, & \mbox{$y \leq x \leq 0,$} \\ \emptyset, & \mbox{$y \leq 0 \leq x,$} \\ \{ (j,k), \quad 1 \leq j \leq 4, 1 \leq k \leq 2\}, & \mbox{$x \leq 0 \leq y$}, \\ \{ (j,j), \quad 1 \leq j \leq 4\}, & \mbox{$0 \leq x \leq y,$} \\ \emptyset,  & \mbox{$0 \leq y \leq x.$} \end{array}\right.
 $$
 
 Define now the excited term
 $$ {\cal E}_\l(x,y) := \lambda^{-1} (\bar U^\e)'(x) \sum_{j \in \tilde {\bf D}(y)} [c^{0}_{j,s_y}] \ell_{j}^{s_y\tt tr} e^{-y \mu^{s_y}_{j}(\l)},$$
 and the scattered term 
  $$ {\cal S}_\l(x,y) := \sum_{(j,k) \in {\bf D}(x,y)} [c^{j,s_x}_{k,s_y}] r_{j}^{s_x} \ell_{k}^{s_y\tt tr} e^{x \mu_j^{s_x}(\l) - y \mu_k^{s_y}(\l)},$$
  where the vectors $\ell_j^\pm$ are defined in \eqref{l4} and \eqref{lj}, the vectors $r_j^\pm$ are defined in \eqref{a}, and the transmission coefficients $[c^0_{k,\pm}]$ and $[c^{j,\pm}_{k,\pm}]$ are constants.   
   
\begin{prop} \label{G_l} Under \eqref{D}, for $\l \in B(0,r),$ the radius $r$ being possibly smaller than in Lemma {\rm \ref{low-f},} there exist transmission coefficients $[c^0_{k,\pm}]$ and $[c^{j,\pm}_{k,\pm}]$ such that the resolvent kernel decomposes as 
$$ {\cal G}_\lambda = {\cal E}_\lambda + {\cal S}_\lambda + {\cal R}_\l,$$
 where ${\cal R}_\l$ satisfies
$$\begin{aligned} | \d_x^\a \d_y^{\a'} {\cal R}_\l| \leq C e^{-\theta|x-y|} & + C \l^{\a'} e^{-\theta|x|} \sum_{j \in \tilde {\bf D}(y)} e^{-y \mu_j^{s_y}} \\ & + C \big( \l^{1 + \min(\a,\a')} + \l^\a e^{-\theta|x|}\big) \sum_{(j,k) \in {\bf D}(x,y)}  e^{x \mu_k^{s_x} - y \mu_j^{s_y}},
\end{aligned}$$
for $\a \in \{0,1,2,\},$ $\a' \in \{0,1\},$ for some $C, \theta > 0,$ uniformly in $x,y$ and $\l \in B(0,r).$
\end{prop}

\begin{proof} The duality relation \eqref{duality} allows to apply Proposition 4.6 of \cite{MaZ3} (and its Corollary 4.7), which describes ${\cal G}_\l$ as sums of pairings of forward and dual modes, for $\l$ in the intersection of $\L$ and the resolvent set of $L.$ By Lemma \ref{fwd-dual}, ${\cal G}_\l$ extends as a meromorphic map on $B(0,r).$  
 
 The excited term ${\cal E}_\l$ comprises the pole terms, corresponding to pairings of a fast, decaying forward mode associated with the derivative of the background wave with a slow, decaying dual mode, i.e. $W^+_7/\bar W^-_3$ for $y \leq 0$ and $W^-_4/\bar W^+_j$ for $y \geq 0,$ $1 \leq j \leq 4.$  
  
 The next-to-leading order term is the scattered term ${\cal S}_\l.$ It corresponds to pairings of a slow forward mode with a slow dual mode. For $y \leq 0,$ the scattered term comprises only fluid modes. For $y \leq 0 \leq x$ and for $0 \leq y \leq x,$ the scattered term vanishes, as there are no outgoing modes to the right of the shock (see Figures \ref{fig1} and \ref{fig4}).
  
 By the Evans function condition \eqref{D} and Lemma 6.11 of \cite{MaZ3}, the residual ${\cal R}_\l$ does not contain any pole term; it comprises: 
  \begin{itemize}
  \item[(a)] the contribution of the normal residuals to the fast forward/slow dual pairings involving the derivative of the background profile,
  \item[(b)] the fast forward/slow dual pairings not involving the derivative of the background profile,
  \item[(c)] the contribution of the normal residuals to the slow forward/slow dual pairings, and
  \item[(d)] the slow forward/fast dual pairings.
  \end{itemize}
  
 Term (a) is bounded by the first two terms in the upper bound for ${\cal R}_\l.$ Term (b) is smaller than term (a) by a $O(\l)$ factor. Term (c) is bounded by the third term in the upper bound. By the Lax condition \eqref{Lax}, the Evans function condition \eqref{D} and Lemma 6.11 of \cite{MaZ3}, term (d) is also bounded by the third term.
 \end{proof}

\subsubsection{The Evans function} \label{sec:Evans}

 By Lemma \ref{dcs}, for all $\l \in \L,$ the dimensions of $U(\mA_-(\l))$ and $S(\mA_+(\l)),$ the vector spaces associated with decaying solutions of \eqref{ode} at $-\infty$ and $+\infty,$ add up to the full dimension of the ambient space: 
  $$ \dim U(\mA_-(\l)) + \dim S(\mA_+(\l)) = 7. $$
   
\begin{defi}[Evans function] On $\L \cup B(0,r),$ define the Evans function as 
 $$ D(\e,\l) := \det (W_1^-, W_2^-, W^-_4, W^-_7, W_5^+, W_6^+, W_7^+)_{|x=0}.$$
\end{defi}

 The Evans function $D$ satisfies Proposition \ref{Lyprop}; it has a zero at $\l = 0,$ as reflected in equality \eqref{zero-E}.

\subsection{Inverse Laplace transform}\label{sec:IVT}

 Similarly as in Section \ref{sec:resker2} (or Section 2 of \cite{MaZ3}), define the Green function of $L(\e)$ as
 \begin{equation} \label{greenfn} {\cal G} := e^{t L(\e)} \t_{(\cdot)} \delta,
 \end{equation}
 where $\{ e^{t L(\e)}\}_{t \geq 0}$ is the semi-group generated by $L(\e).$ That is, the kernel of the integral operator $e^{t_0 L(\e)}$ is the Green function ${\cal G}$ evaluated at $t = t_0.$
 
 Assuming \eqref{D}, the inverse Laplace transform representation of the semi-group by the resolvent operator (see for instance \cite{Pa} Theorem 7.7; \cite{Z3} Proposition 6.24) yields 
 \be \label{inverseLT2}
{\cal G}(\e, x,t;y) = \frac{1}{2\pi i} \text{\rm P.V.}\int_{\eta_0-i\infty}^{\eta_0 + i\infty}
e^{\lambda t}{\cal G}_\lambda (\e,x,y) \,d \lambda,
\ee
for $\eta_0 > 0$ sufficiently large. 

\subsubsection{Pointwise Green function bounds} \label{sec:pointwise}

 Introduce the notations 
  $$\mbox{errfn}(y) := \int_{-\infty}^y e^{-z^2} \, dz,$$ and let, for $y < 0:$
  \begin{equation} \label{E-} {\bf e} :=  [c^0_{3,-}] \ell_{3}^{-\tt tr} \left( {\rm errfn}\left(\frac{y + a_3^- t}{\sqrt{4 \b_j^- t}}\right) - {\rm errfn}\left(\frac{y - a_3^- t}{\sqrt{4 \b_j^- t}}\right) \right),\end{equation}
 for $y > 0:$
 \begin{equation} \label{E+} {\bf e} :=  \sum_{1 \leq j \leq 4} [c^0_{j,+}] \ell_{j}^{+\tt tr} \left({\rm errfn}\left(\frac{y + a_j^+ t}{\sqrt{4 \b_j^+ t}}\right)  - {\rm errfn}\left(\frac{y - a_j^+ t}{\sqrt{4 \b_j^+ t}}\right)\right),\end{equation}
 and
  \begin{equation} \label{def-E}
   {\cal E}(\e,x,t;y) := (\bar U^\e)'(x) {\bf e}(\e,t;y).
   \end{equation}
 In \eqref{E-}-\eqref{E+} and below, the $\{a_j^\pm\}_{1\leq j \leq 4}$ are the characteristic speeds, i.e.~the limits at $\pm \infty$ of the eigenvalues of $\d_U F(\e, \bar U^\e),$ ordered as in \eqref{char}, the $\b_j^\pm,$ $1 \leq j \leq 3,$ are the positive diffusion rates that were introduced in \eqref{fluid-modes0} (and which depend on $\e,$ as do the characteristic speeds), and $\b_4^+ := d,$ the species diffusion coefficient.
 
Let for $y < 0:$
\begin{equation} \label{S-}\begin{aligned}
 {\cal S} :=  \chi_{\{t \geq 1\}} \sum_{1 \leq j \leq 2} r_j^- \ell_j^{-\tt tr} & (4 \pi \b_j^- t)^{-\frac{1}{2}} e^{-(x-y - a_j^-t)^2/4 \b_j^- t}
 \\ \quad + \quad \chi_{\{ t \geq 1\}} \frac{e^{-x}}{e^{-x} + e^x} & \Big( r_3^- \ell_3^{-\tt tr} (4 \pi \b_3^- t)^{-\frac{1}{2}} e^{-(x-y +(s - \s_-)t)^2/4 \b_3^- t} \\
   & \quad + \sum_{1 \leq j \leq 2} [c^{j,-}_{3,-}] r_j^- \ell_3^{-\tt tr} (4 \pi \b^{j,-}_{3,-} t)^{-\frac{1}{2}} e^{-(x - z^{j,-}_{3,-})^2/4 \b^{j,-}_{3,-} t} \Big),
     \end{aligned}
\end{equation}
and for $y > 0:$ 
\begin{equation} \label{S+}
 \begin{aligned}
  {\cal S}  & := \chi_{\{t \geq 1\}} \frac{e^{x}}{e^{-x} + e^x} \sum_{1 \leq j \leq 4} r_j^+ \ell_j^{+\tt tr} (4 \pi \b_j^+ t)^{-\frac{1}{2}} e^{-(x - y - a_j^+t)^2/4 \b_j^+ t} \\
    & \quad  +  \quad \chi_{\{t \geq 1\}} \frac{e^{-x}}{e^{-x} + e^x} \sum_{1 \leq j \leq 4} \sum_{1 \leq k \leq 2} [c^{k,-}_{j,+}] r_k^- \ell_j^{+\tt tr} (4 \pi \b^{k,-}_{j,+} t)^{-\frac{1}{2}}  e^{-(x - z^{k,-}_{j,+})^2/4 \b^{k,-}_{j,+} t},   \end{aligned}
\end{equation} 
   where the indicator function $\chi_{\{t \geq 1\}}$ is identically equal to 1 for $t \geq 1$ and 0 otherwise, and
 $$ z^{k,\pm}_{j,\pm} := a_j^\pm (t - |y| |a_k^\pm|^{-1}), \quad \b^{j,\pm}_{k,\pm} := \frac{|x|}{|a_j^\pm| t} \b_j^\pm + \frac{|y|}{ |a_k^\pm|t} \left(\frac{a_j^\pm}{a_k^\pm}\right)^2 \b_k^\pm.$$

 Let 
 \begin{equation} \label{def-h} {\cal H} := h(\e,t,x,y) \t_{x + st} \delta, \qquad h \ell_4^+ \equiv 0,\end{equation}
 where the notation $\t_{(\cdot)}\delta$ was introduced at the beginning of Section \ref{sec:resker2}. 
 
 Let finally ${\cal S}_0$ be the scattered term defined in \eqref{S-}-\eqref{S+} in which $[c^{j,\pm}_{k,\pm}] = 1$ for all $j,k.$

\begin{prop}\label{greenbounds}
Under \eqref{D}, there exists transmission coefficients $[{c}^0_{j,\pm}]$ and $[{c}^{j,\pm}_{k,\pm}],$ satisfying 
 \begin{equation} \label{trans-coeff} \left\{\begin{aligned}{} [c^0_{4+}] & = 0, & \\
   r_k^\star & = [c^0_{k,\star}] (v^\e_+ - v^\e_-) + [c^{1,-}_{k,\star}] r_1^- + [c^{2,-}_{k,\star}] r_2^-,\\
   [c^{1,-}_{4,+}] & = [c^{2,-}_{4,+}] = 0,
   \end{aligned} \right. \end{equation}
 where $1 \leq k\leq3$ if $\star = +$ and $1 \leq k \leq 2$ if $\star = -,$ $U^\e_\pm =:(v^\e_\pm,z^\e_\pm),$ such that the Green function ${\cal G}(\e,x,t;y)$ defined in \eqref{greenfn} may be decomposed as a sum of hyperbolic, excited, scattered, and residual terms, as follows:
 \begin{equation} \label{dec-G} {\cal G}= {\cal H} + {\cal E} + {\cal S} + {\cal R},\end{equation}
 where ${\cal H},$ ${\cal E}$ and ${\cal S}$ are defined in \eqref{E-}-\eqref{def-h}, 
with the estimates
\begin{equation} \label{bds-R} \begin{aligned}
  |\d_t^k \d_x^\a \d_y^{\a'} h| & \leq C \, e^{-\theta t},  \\ 
  |\d_x^\a \d_y^{\a'} {\cal R}| & \le  C \, e^{-\theta(|x-y| + t)} \\ &  + C\left( t^{-\frac{1}{2}(1 + \a + \a')} (t+1)^{-\frac{1}{2}} + e^{-\theta t}\right) e^{-(x-y)^2/Mt}  \\ & + C \left((t^{- \frac{1}{2}} + e^{-\theta |x|}) t^{-\frac{1}{2}(\a + \a')} + \a' t^{-\frac{1}{2}}e^{-\theta |y|} \right) |{\cal S}_0| . 
\end{aligned} \end{equation}
 uniformly in $\e,$ for $k \in \{0,1\},$ $\a \in \{0,1,2\},$ $\a' \in \{0,1\},$ for some $\theta, C, M > 0.$ 
\end{prop}

\begin{proof} We only check \eqref{trans-coeff}, as decomposition \eqref{dec-G} and bounds \eqref{bds-R} are easily deduced from Proposition 7.1 of \cite{MaZ3} and Proposition 7.3 of \cite{LRTZ}. (See also Proposition 3.7 of \cite{TZ2}, especially equations (3.30)-(3.33) and (3.38).)

 The description of the residue of ${\cal G}_\l$ at $\l  = 0$ for $y < 0$ and $y > 0$ implies
    $$[c^0_{3,-}] \ell_{3}^- = [c^0_{4,+}] \ell_{4}^+ + \sum_{1 \leq j \leq 3} [c^0_{j,+}] \ell_{j}^+,$$
    corresponding to equation (1.34) in \cite{MaZ3}. The (reactive) left eigenvector vector $\ell_{4}^+$ being orthogonal to the (fluid) left eigenspace $\mbox{span}\{\ell_{j}^\pm\}_{1 \leq j \leq 3}$ (see \eqref{l4} and \eqref{lj}), this implies 
 \eqref{trans-coeff}(i).
 
Given $U_0 \in L^1,$ the estimates for ${\cal H}$ and ${\cal R}$ imply
 $$ \lim_{t \to +\infty} \int_{\R^2} ({\cal H} + {\cal R}) U_0 \, dy \, dx = 0.$$
Hence, by conservation of mass in the fluid variables, \eqref{cons}, for all $U_0 \in L^1,$
 \begin{equation} \label{cons-mass}
   \lim_{t \to +\infty} \int_{\R^2}  \pi ({\cal E} + {\cal S}) U_0 \, dy \, dx = \int_\R \pi U_0 \, dy.
 \end{equation}
where $\pi: \C^4_{v,z} \to \C^3_{v}$ is defined by $\pi (v,z) := v.$ (Equation \eqref{cons-mass} corresponds to (1.33) and (7.60) in \cite{MaZ3}.) Taking $U_0 \in \mbox{span}\{\ell_j^\pm\}_{1 \leq j \leq 3},$ we find \eqref{trans-coeff}(ii), and taking $U_0$ parallel to $\ell_4^+,$ we find \eqref{trans-coeff}(iii).
\end{proof}

\begin{rem} The terms ${\cal E}$ and ${\cal S}$ correspond to the low-frequency part of 
representation of ${\cal G}$ by inverse Laplace transform of the resolvent
kernel ${\cal G}_\lambda$, while the term ${\cal H}$ corresponds to the high-frequency part.  
As observed in {\rm \cite{MaZ3},} for low frequencies, 
the resolvent kernel in the case of real (physical) viscosity obeys
essentially the same description as in the artificial (Laplacian)
viscosity case, hence the estimates on ${\cal E}$ and ${\cal S}$ follow by the 
analysis in {\rm \cite{LRTZ}} of the corresponding artificial viscosity system,
specialized to the case of strong detonations (more general waves were
treated in {\rm \cite{LRTZ}}).
The estimate of the terms ${\cal H}$ and ${\cal R}$ follows exactly as for the nonreactive
case treated in {\rm \cite{MaZ3,Z2}.}
\end{rem}

\begin{rem} 
 Bound \eqref{bds-R}{\rm (ii)} is implied by bounds (7.1)-(7.4) of Proposition 7.1 of {\rm \cite{MaZ3}} and bounds (3.30), (3.32) and (3.38) of Proposition 3.7 of {\rm \cite{TZ2}}.
 
 Here the contribution of the hyperbolic, delta-function terms to the upper bounds for the spatial derivatives of ${\cal R}$ is absorbed in ${\cal H},$ and the short-time, $t \leq |a^\pm_k||y|,$ contributions of the scattered terms are absorbed in the generic parabolic residual term $e^{-\theta t} e^{-(x-y)^2/Mt}.$
 \end{rem}
 
\begin{cor} \label{e-r} The excited terms ${\cal E}_\l$ and ${\cal E}$ contain only fluid terms: ${\cal E}_\l \, \ell_4^{+} \equiv 0$ and ${\cal E} \, \ell_4^{+} \equiv 0.$
\end{cor}

\begin{proof} The equality ${\cal E} \, \ell_4^{+} \equiv 0$ follows from \eqref{trans-coeff}(i). The resolvent kernel ${\cal G}_\l$ is the Laplace transform of the Green function ${\cal G},$ so that the coefficients $[c^0_{j,\pm}]$ in Propositions \ref{G_l} and \ref{greenbounds} must agree. Hence, \eqref{trans-coeff}(i) implies also ${\cal E}_\l \, \ell_4^{+} \equiv 0.$
\end{proof}

\begin{cor} \label{s-r} For all $\eta > 0,$ for some $C, M > 0,$ some $\theta_1(\eta,s) > 0,$ 
 the following bounds hold, for $\a \in \{0,1,2\}:$
 \begin{equation} \label{r-s-r} \begin{aligned} | e^{- \eta y^+} \d_x^\a {\cal S} \ell_4^{+}| & \leq C e^{-\theta_1 t} e^{-\eta|x-y|/2},\\
 | e^{- \eta y^+} \d_x^\a {\cal R} \ell_4^{+}| & \leq C e^{-\theta_1 t} (e^{-\eta|x-y|/2} + e^{-(x-y)^2/Mt}).
 \end{aligned}\end{equation}
 \end{cor}

\begin{proof} By \eqref{trans-coeff}{\rm (iii)}, the contribution of the reactive modes to $S$ is 
 \begin{equation} \label{con-r-s} \chi_{\{ y > 0\}} \chi_{\{t \geq 1\}} \frac{e^x}{e^x + e^{-x}} r_4^+ \ell_4^{+\tt tr} (4 \pi d t )^{-1/2} e^{-(x - y  + s t)^2/4dt}.
 \end{equation}
 Given $0 \leq x \leq y,$ we can bound $e^{-\eta y} e^{-(x - y + st)^2/4dt},$ for $|x - y| \leq \frac{1}{2} st,$ by
  $$ e^{-\eta y} e^{- (st/2)2/4t}\le   e^{-\eta y}e^{- s^2 t/16} \leq  e^{-\eta |y-x|}e^{- s^2 t/16},$$
  and, for $|x-y| > \frac{1}{2} st,$ by
  $$ e^{-\eta y}\le e^{-\eta |y-x|/2} e^{-\eta y/2}                \le e^{-\eta|y-x|/2} e^{-\eta st/4},$$
  and this implies \eqref{r-s-r}(i). 
  To prove \eqref{r-s-r}(ii), we note that the contribution of the parabolic terms $t^{-\frac{1}{2}} (t+1)^{-\frac{1}{2}} e^{-(x-y)^2/Mt}$ and ${\cal S}_0$ to ${\cal R} \ell_4^+$ comes from Riemann saddle-point estimates of the sole scattered terms ${\cal S}_\l \ell_4^+$ (see the proof of Proposition 7.1 in \cite{MaZ3} for more details). Hence \eqref{r-s-r}(i) implies \eqref{r-s-r}(ii).
 \end{proof}
 
 \begin{rem} \label{rem-greenbounds} The proof of Proposition 7.1 of {\rm \cite{MaZ3}} shows that Proposition {\rm \ref{greenbounds}} 
 applies more generally to linear operators of the form \eqref{def-l} that satisfy \eqref{D} and the conditions {\rm (A1)-(A2),} {\rm (H1)-(H4)} of Section {\rm \ref{remarkson}.}
\end{rem}

\subsubsection{Convolution bounds}\label{conbds}

{}From the pointwise bounds of Proposition \ref{greenbounds} and Remarks \ref{e-r} and \ref{s-r}, we
obtain by standard convolution bounds the following $L^p\to L^q$
estimates, exactly as described in \cite{MaZ1,MaZ2,MaZ3,MaZ4,Z2}
for the viscous shock case.

\begin{cor}\label{Lpbds}
Under \eqref{D}, for all $t\ge 1$, some $C >0$, any $\eta>0$, for any $1\le q\le p,$ $1 \leq p \leq +\infty,$ and $f\in L^q \cap W^{1,p},$
\begin{eqnarray}
  \left|\int_\R ({\cal S} + {\cal R})(\cdot,t;y)f(y)
   \,dy\right|_{L^p}
  & \le & C t^{-\frac{1}{2}(\frac{1}{q}-\frac{1}{p})} |f|_{L^q},
 \label{tGbounds} \nonumber
 \\
 \left|\int_\R \d_y ({\cal S} + {\cal R})(\cdot,t;y)f(y)
   \,dy\right|_{L^p}
  & \le & C t^{-\frac{1}{2}(\frac{1}{q} - \frac{1}{p})-\frac{1}{2}} |f|_{L^q}
  + Ce^{-\eta t} |f|_{L^p},
 \label{tGybounds} \nonumber 
 \\
  \left|\int_\R 
({\cal S} + {\cal R}) (\cdot,t;y) \ell_4^+ f(y) e^{-\theta y^+} \,dy\right|_{L^p}
 & \le & C t^{-\frac{1}{2}(\frac{1}{q} - \frac{1}{p})-\frac{1}{2}} |f|_{L^q} + Ce^{-\eta t} |f|_{L^p},
 \label{tGqbounds} \nonumber \\
 \left|\int_\R {\cal H}(\cdot,t;y)f(y)dy\right|_{L^p}
  & \le & Ce^{-\eta t} |f|_{L^p},
 \label{Hbounds} \nonumber
 \end{eqnarray}
where $y^+ := \max(y,0).$ Likewise, for all $x$ and all $t \geq 0,$ 
\begin{eqnarray}
  |\d_y {\bf e} (\cdot, t)|_{L^p} + |\d_t {\bf e}(\cdot, t)|_{L^p} 
  & \le&  C t^{-\frac{1}{2}(1-\frac{1}{p})},
 \label{e1} \nonumber \\
  |\d_t \d_y {\bf e}(\cdot, t)|_{L^p} & \le & C t^{-\frac{1}{2}(1-\frac{1}{p})-\frac{1}{2}}.
 \label{e2} \nonumber
\end{eqnarray}
\end{cor}

\section{Stability: Proof of Theorem \ref{stab}}\label{stabsection}

 We often omit to indicate dependence on $\e$ in the proof below. All the estimates are uniform in $\e.$

\subsection{Linearized stability criterion}\label{linstab}

\begin{proof}[Proof of Theorem {\rm \ref{stab}:} linear case]
Sufficiency of \eqref{D} for linearized orbital stability follows
immediately by the bounds of Corollary \ref{Lpbds}, 
exactly as in the viscous shock case, setting
$$
\delta(t):= \int_\R {\bf e}(x,t;y)U_0(y) \, dy
$$
so that
$$
U-\delta(t)\bar U'= \int_\R ({\cal H} + {\cal S} + {\cal R})(x,t;y)U_0(y) \, dy;
$$
see \cite{ZH, MaZ3, Z2} for further details.
Necessity follows from more general spectral considerations not requiring
the detailed bounds of Proposition \ref{greenbounds};
see the discussion of effective spectrum in \cite{ZH, MaZ3, Z2}.
The argument goes again exactly as in the viscous shock case.
\end{proof}

\subsection{Auxiliary energy estimate}\label{auxiliary}

Consider $\tilde U$ the solution of \eqref{rNS-symb} issued from $\tilde U_0,$ and let 
 \begin{equation} \label{perturb0}
  U(x,t):= \tilde U(x+\delta(t),t)-\bar U(x).
  \end{equation}
Then, the following auxiliary energy estimate holds. 

\begin{lem} [Proposition 4.15, \cite{Z2}] \label{aux}
Under the hypotheses of Theorem {\rm \ref{stab}}, assume that $\tilde U_0\in H^3$, and
suppose that, for $0\le t\le T$, the suprema of $ |\dot\delta|$
and the $H^{3}$ norm of $U$
each remain bounded by a sufficiently small constant. Then, for all
$0\le t\le T$, for some $\theta > 0,$ 
$$ |U(t)|_{H^3}^2\le C e^{-\theta
t} |U(0)|^2_{H^3} +C\int_0^t e^{-\theta (t-s)}(|U|_{L^2}^2+
|\dot\delta|^2 )(s)\, ds.
$$
\end{lem}

\subsection{Nonlinear stability} \label{stabproof}

\begin{proof}[Proof of Theorem {\rm \ref{stab}:} nonlinear case]

Let $U$ be the perturbation variable associated with solution $\tilde U$ as in \eqref{perturb0}; by a Taylor expansion, $U$ solves the perturbation equation
$$\d_t U-LU= \d_x \mathit{Q_f}(U,\d_x U) + \mathit{Q}_r(U)+ \dot \delta (t)(\bar U' + \d_x U),$$
where the linear operator $L$ is defined in \eqref{def-l}, and
 \begin{equation} \label{Qbound} | \mathit{Q_f} |\leq C |U|(|U| +  |\d_x U|),\end{equation}
 where $C$ depends on $\| U \|_{L^\infty}$ and $\| \bar U\|_{W^{1,\infty}}.$  

\begin{lem}\label{r-term} Under Assumptions {\rm \ref{idealgas}}, {\rm \ref{ignition}}, if the temperature $T$ associated with solution $U$ satisfies
 $\| T\|_{L^\infty} < T_i - T_+$
 (by Lemma {\rm \ref{ex}}, $0 < T_i - T_+$), then the nonlinear reactive term $\mathit{Q}_r$ has the form
\begin{equation} 
 \label{reactive-NL-term}
 \mathit{Q}_r(U)=\ell_4^+ e^{-\eta_0 x^+}q_r(U),
 \end{equation}
 where $x^+ := \max(x,0),$ $\eta_0 > 0$ is as in Corollary {\rm \ref{profdecay},} and $q_r(U) = q_r(w,z)$ is a scalar such that
\begin{equation} \label{qbound} |q_r(U)| \leq C|U^|2,\end{equation}
where $C$ depends on $\|U\|_{L^\infty}$ and $\|\bar U\|_{L^\infty}.$
\end{lem}

\begin{proof} 
We use the specific form
$-k \phi(T)z \ell_4^{+}$
of the reactive source in \eqref{rNS}, 
together with Taylor expansion
$$
\begin{aligned}
(\phi(\bar T+T)(\bar z + z)-
(\phi(\bar T)\bar z -
&(\phi'(\bar T)T\bar z + \phi(\bar T)z)\\&=
\phi'(\bar T) T  z + \phi''(\bar T+ \beta T)T^2\bar z,\\
\end{aligned}
$$
for some $0< \beta <1$, and the fact that $\phi'(\bar T + T)\le Ce^{-\eta_0 x^+}$
for $|T| < T_i - T_+,$ for $\eta > 0$ as in Corollary \ref{profdecay}, by $\phi(T_+) = 0$ together with
the property that $\phi'(T)\equiv 0$ for $T\le T_i$
and exponential convergence of $\bar U(x)$ to $U_+$ as $x\to +\infty$.
\end{proof}

Recalling the standard fact that $\bar U'$ is a stationary
solution of the linearized equations \eqref{linearized},
$L\bar U'=0$, or
$$
\int_\R {\cal G}(x,t;y)\bar U'(y)dy=e^{t L}\bar U'(x)
=\bar U'(x),
$$
we have by Duhamel's principle:
$$
\begin{array}{l}
  \displaystyle{
  U(x,t)= \delta (t)\bar U'(x) + \int_\R {\cal G}(x,t;y)U_0(y)\,dy} \\
  \displaystyle{\qquad
  +\int^t_0 \int_\R {\cal G}(x,t-s;y) \ell_4^{+} e^{-\eta y^+}
  q_r(U) (y,s)\,dy\,ds}\\
  \displaystyle{\qquad
  -\int^t_0 \int_\R \d_y {\cal G}(x,t-s;y)
  (\mathit{Q_f}(U,\d_x U)+\dot \delta U ) (y,s)\,dy\,ds.}
\end{array}
$$
Defining
\begin{equation}
 \begin{array}{l}
  \displaystyle{
  \delta (t)=-\int_\R {\bf e}(y,t) U_0(y)\,dy }\\
  \displaystyle{\qquad
  -\int^t_0 \int_\R {\bf e}(y,t-s) \ell_4^{+} e^{-\eta y^+}
  q_r(U) (y,s)\,dy\,ds}\\
  \displaystyle{\qquad
  +\int^t_0\int_{\R} \d_y {\bf e}(y,t-s)(\mathit{Q_f}(U,\d_x U)+
  \dot \delta\, U)(y,s) dy ds, }
  \end{array}
 \label{predelta}
\end{equation}
following \cite{Z3, MaZ1, MaZ2, MaZ4},
and recalling Proposition \ref{greenbounds},
we obtain finally the {\it reduced equations}:
\begin{equation}
\begin{array}{l}
 \displaystyle{
  U(x,t)=\int_\R ({\cal H} + {\cal S} + {\cal R})(x,t;y)U_0(y)\,dy }\\
  \displaystyle{\qquad
  +\int^t_0 \int_\R {\cal H}(x,t-s;y) \big(\d_y (\mathit{Q_f}(U,\d_x U)+ \dot \delta U) + \ell_4^{+}
e^{-\eta y^+}
  q_r(U) \big)\,dy\,ds}\\
   \displaystyle{\qquad
  +\int_0^t \int_\R ({\cal S} + {\cal R})(x,t-s;y) \ell_4^{+} e^{- \eta y^+} q_r(U) \, dy \,ds} \\
  \displaystyle{\qquad
  -\int^t_0 \int_\R \d_y({\cal S} + {\cal R})(x,t-s;y)
(\mathit{Q_f}(U, \d_x U)+ \dot \delta U) dy \, ds, }
\end{array}
\label{reduced}
\end{equation}
and, differentiating (\ref{predelta}) with respect to $t$,
and recalling Corollary \ref{e-r}:
\begin{equation}
 \begin{array}{l}
 \displaystyle{
  \dot \delta (t)=-\int_\R \d_t {\bf e}(y,t) U_0(y)\,dy }\\
 \displaystyle{\qquad
  +\int^t_0\int_{\R} \d_y \d_t {\bf e}(y,t-s)(\mathit{Q_f}(U, \d_x U)+
  \dot \delta U)(y,s)\,dy\,ds. }
 \end{array}
\label{deltadot}
\end{equation}\begin{equation}
 \begin{array}{l}
  \displaystyle{
  \delta (t)=-\int_\R {\bf e}(y,t) U_0(y)\,dy }\\
  \displaystyle{\qquad
  +\int^t_0\int_{\R} \d_y {\bf e}(y,t-s)(\mathit{Q_f}(U, \d_x U)+
  \dot \delta\, U)(y,s) \, dy \, ds, }
  \end{array}
 \label{delta}
\end{equation}
\medskip

Define
$$
\aligned
\zeta(t)
&:= \sup_{0\le s \le t,\, 2\le p\le \infty}
\Big( \, |U(\cdot, s)|_{L^p}(1+s)^{\frac{1}{2}(1-\frac{1}{p})}
+ |\dot \delta (s)|(1+s)^{\frac{1}{2}} 
+ |\delta (s)|\Big).
\endaligned
$$
We shall establish:

{\it Claim. There exists $c_0 > 0,$ such that, for all $t\ge 0$ for which a solution exists with
$\zeta$ uniformly bounded by some fixed, sufficiently small constant,
there holds}
$$
\zeta(t) \leq c_0(|U_0|_{L^1 \cap H^3} + \zeta(t)^2).
$$
\medskip
{}{}From this result, it follows by continuous induction that,
provided 
$$
|U_0|_{L^1\cap H^3} < \frac{1}{4} c_0^2,  
$$
there holds 
\begin{equation} \label{bd}
\zeta(t) \leq 2 c_0 |U_0|_{L^1\cap H^3}
\end{equation}
for all $t\geq 0$ such that $\zeta$ remains small.
For, by standard short-time theory/local well-posedness in $H^3$, 
and the standard principle of continuation, 
there exists a solution  $U\in H^3$
on the open time-interval for which $|U|_{H^3}$ remains bounded,
and on this interval $\zeta$ is well-defined and continuous. 
Now, let $[0,T)$ be the maximal interval on which $|U|_{H^3}$
remains strictly bounded by some fixed, sufficiently small constant $\delta>0$.
By Lemma \ref{aux}, we have
$$
\aligned
|U(t)|_{H^3}^2&\le C |U(0)|^2_{H^3}e^{-\theta t}
+C\int_0^t e^{-\theta (t-\tau )}(|U|_{L^2}^2+ |\dot\delta|^2)(\tau)\, d\tau \\
&\le C'\big(|U(0)|^2_{H^3}+ \zeta(t)^2\big) (1+t)^{-\frac{1}{2}},
\endaligned
$$
for some $C, C', \theta > 0,$ and so the solution continues so long as $\zeta$ remains small,
with bound \eqref{bd}, at once yielding existence and the claimed 
sharp $L^p\cap H^3$ bounds, $2\le p\le \infty$.

\medskip

{\it Proof of Claim.}
We must show that each of the quantities 
$$
|U|_{L^p}(1+s)^{\frac{1}{2}(1-\frac{1}{p})}, \quad |\dot \delta|(1+s)^{\frac{1}{2}}, \quad \mbox{and }\quad|\delta|$$
is separately bounded by
$
C(|U_0|_{L^1\cap H^3} + \zeta(t)^2),$
for some $C>0$, all $0\le s\le t$, so long as $\zeta$ remains
sufficiently small.
By \eqref{reduced}--\eqref{deltadot} and the triangle inequality, we have
\begin{eqnarray}
 |U|_{L^p} & \leq & {\rm {\rm I_a}} + {\rm {\rm I_b}} + {\rm I_c} + {\rm I_d}, \nonumber \\
 |\dot \delta(t)| & \leq & {\rm I{\rm I_a}} + {\rm I{\rm I_b}}, \nonumber \\
 |\delta(t)| & \leq & {\rm II{\rm I_a}} + {\rm II{\rm I_b}},\nonumber
 \end{eqnarray}
where ${\rm I_a}$ is the $L^p$ norm of the first integral term in the right-hand side of \eqref{reduced}, ${\rm I_b}$ the second term, etc., and similarly ${\rm II_a}$ is the modulus of the first term in the right-hand side of \eqref{deltadot}, etc. 

We estimate each term in turn, following the approach of \cite{MaZ1, MaZ4}.

 The linear term ${\rm I_a}$  satisfies bound
$$
{\rm I_a} \le  C|U_0|_{L^1\cap L^p} (1+t)^{-\frac{1}{2}(1-\frac{1}{p})},
$$ 
by Proposition \ref{greenbounds} and Corollary \ref{Lpbds}. 

Likewise, applying the bounds of Corollary \ref{Lpbds}, we have
$$
\aligned
{\rm I_b}
&\le
C\zeta(t)^2
\int_0^t e^{-\eta (t-s)}(1+s)^{-\frac{1}{2}}ds\\
&\le
C\zeta(t)^2 (1+t)^{-\frac{1}{2}},\\
\endaligned
$$
and (taking $q=2$ in the second estimate of Corollary \ref{Lpbds})
$$
\aligned
{\rm I_c}+ {\rm I_d}
&\le
C\int_{0}^t e^{-\eta(t-s)}
(|U|_{L^\infty}+|\d_x U|_{L^\infty} + |\dot \delta|)|U|_{L^p}(s) ds\\
&\quad+
C\int_0^{t} (t-s)^{-\frac{3}{4}+\frac{1}{2p}}
(|U|_{L^\infty}+|\dot \delta|)|U|_{H^1}(s) ds\\
&\le
C\zeta(t)^2
\int_{0}^t e^{-\eta(t-s)}
(1+s)^{-\frac{1}{2}(1-\frac{1}{p})-\frac{1}{2}}ds\\
&\quad +
C\zeta(t)^2
\int_0^{t} (t-s)^{-\frac{3}{4}+\frac{1}{2p}}
(1+s)^{-\frac{3}{4}}ds\\
&\le
C\zeta(t)^2 (1+t)^{-\frac{1}{2}(1-\frac{1}{p})},\\
\endaligned
$$
$$
\aligned
{\rm II_a}
&\le&
| \d_t {\bf e}(t)|_{L^\infty} |U_0|_{L^1}
&\le&
C|U_0|_{L^1} (1+t)^{-\frac{1}{2}}
\endaligned
$$
and 
$$
\aligned
{\rm II_b}
&\le
\int^t_0
| \d_y \d_t {\bf e}(t-s)|_{L^2}
(|U|_{L^\infty}+|\dot \delta|)|U|_{H^1}(s) ds\\
&\le
C\zeta(t)^2 
\int^t_0
(t-s)^{-\frac{3}{4}} (1+s)^{-\frac{3}{4}} ds\\
&\le
C\zeta(t)^2 (1+t)^{-\frac{1}{2}},
\endaligned
$$
while
$$
\aligned
{\rm III_a}
&\le&
| {\bf e}(t)|_{L^\infty_y} |U_0|_{L^1}
&\le&
C|U_0|_{L^1} 
\endaligned
$$
and 
$$
\aligned
{\rm III_b}
&\le
\int^t_0
| \d_y {\bf e}(t-s)|_{L^2} 
(|U|_{L^\infty}+|\dot \delta|)|U|_{H^1}(s) ds\\
&\le
C\zeta(t)^2 
\int^t_0
(t-s)^{-\frac{1}{4}} (1+s)^{-\frac{3}{4}} ds\\
&\le
C\zeta(t)^2. 
\endaligned
$$

This completes the proof of the claim, establishing \eqref{nonest}
for $p\ge 2$.  The remaining bounds $1\le p<2$ then
follow by a bootstrap argument as described in \cite{Z2}; we omit
the details.
\end{proof}

\section{Bifurcation: Proof of Theorem \ref{PH}}\label{bifsection}

 Given two Banach spaces $X$ and $Y,$ we denote by 
${\cal L}(X,Y)$ the space of linear continuous applications from $X$ to $Y,$
and let ${\cal L}(X) := {\cal L}(X,X).$ We use \eqref{wS} to denote weighted Sobolev spaces and norms. Let $x^+ := \max(0,x).$ Given a constant $\eta > 0$ and a weight function $\o > 0,$
 define subspaces of ${\cal S}'(\R)$ by
   $$  L^1_{\eta^+}  := \{ f, \, e^{\eta (\cdot)^+} f\in L^1\}, 
  \quad  L^1_{\o} := \{ f, \, \o f\in L^1 \}, 
   \quad L^1_{\o,\eta^+} := \{ f, \, \o f \in L^1_{\eta^+}\}.
   $$

\begin{defi} \label{B-spaces} Given a constant $\eta > 0$ and a weight function $\o$ satisfying \eqref{ass-o}, define the Banach spaces 
 $${\cal B}_1, {\cal B}_2, X_1, X_2 \subset {\cal D}'(\R;\C^3_v \times \C_z)$$
  by 
$$ \begin{aligned} {\cal B}_1 & :=  & H^1, \quad {\cal B}_2 & :=  H^1 \cap  (\d_x L^1 \times L^1_{\eta^+}), \\
\nonumber X_1 & := & H^2_{\o}, \quad X_2 & :=  H^2_{\o} \cap (\d_x L^1_\o \times L^1_{\o,\eta^+}). 
\end{aligned}
$$
with norms
$$\begin{aligned} \|(v,z)\|_{{\cal B}_1} & :=\|(v,z)\|_{H^1}, \\
                 \|(\d_x v, z)\|_{{\cal B}_2} & := \| (\d_x v, z) \|_{H^1} +\|v\|_{L^1} + \|e^{\eta (\cdot)^+} z\|_{L^1},\\ \|(v,z)\|_{X_1} & :=
\|(v,z) \|_{H^2_\o}, 
\\ \|(\d_x v, z)\|_{X_2} & := \| (\d_x v, z) \|_{H^2_\o} + \| \o v \|_{L^1} + \| \o e^{\eta (\cdot)^+} z\|_{L^1}.
\end{aligned}$$
\end{defi}

 In particular, $X_2 \hookrightarrow X_1 \hookrightarrow \CalB_1,$ with $\|\cdot\|_{{\cal B}_1}\le \|\cdot\|_{X_1} \le \| \cdot \|_{X_2},$ and $X_2 \hookrightarrow \CalB_2 \hookrightarrow {\cal B}_1,$
 with $\| \cdot \|_{{\cal B}_1} \le \|\cdot\|_{{\cal B}_2}\le \|\cdot\|_{X_2},$ and the unit ball in $X_1$ is closed in $\CalB_1$.

\subsection{The perturbation equations} \label{sec:perturb}

If $\tilde U^\e$ solves \eqref{rNS-symb} with initial datum $\tilde U^\e_{|t=0} = \bar U^\e + U_0^\e,$ then the perturbation variable $U(\e,x,t):=\tilde U^\eps (x,t) - \bar U^\eps(x)$
satisfies
\be\label{nonlin} \left\{ \begin{aligned}
\d_t U-L(\eps) U & = \d_x \mathit{Q_f}(\e, U, \d_x U) + \mathit{Q_r}(\e,U),
\\
U(\e,x,0) & =U_0^\e(x). \end{aligned} \right.
\ee
The nonlinear term $\mathit{Q_f}$ satisfies \eqref{Qbound}, while $\mathit{Q_r}$ satisfies Lemma \ref{r-term}.

\subsection{Coordinatization} \label{sec:coord} 
 Let $\varphi^\e_\pm$ be the eigenfunctions of $L(\e)$ associated with the bifurcation eigenvalues $\g(\e) \pm i \t(\e),$ and let $\tilde \varphi^\e_\pm$ be the corresponding left eigenfunctions. We know from Section \ref{sec:ess} that $(\g \pm i \t)(\e) \in \C \setminus {\cal C}_- \cup {\cal C}_+,$ hence $\varphi^\e_\pm$ decay exponentially at both $-\infty$ and $+\infty,$ in particular, if in \eqref{ass-o} $\theta_0$ is small enough, then $\varphi^\e_\pm \in H^2_\o.$ Let $\Pi$ be the $L^2$-projection onto $\mbox{span}(\varphi^\e_\pm)$ parallel to $\mbox{span}(\tilde \varphi^\e_\pm)^\perp.$ Decomposing
$$ %
U= u_{11} \Re \varphi_+^\eps +  u_{12} \Im \varphi_+^\eps + u_2, \qquad U^\e_0 = a_{1} \Re \varphi_+^\e + a_2 \Im \varphi_+^\e + b,$$
where %
$u_{11} \Re \varphi_+^\eps +  u_{12} \Im \varphi_+^\eps$ and $a_1 \Re \varphi_+^\e + a_2 \Im \varphi_+^\e$ belong to $\mbox{span}(\phi^\e_\pm)$
(so that, in particular, $u_{1j}$ and $a_j$ are real),
and coordinatizing as $(u_1, u_2)$, $u_1 :=(u_{11}, u_{12}) \in \R^2,$ we obtain after a brief calculation that $U$ solves \eqref{nonlin} if and only if its coordinates solve the system 
\be \label{PHeq2} \left\{ \begin{aligned} 
\d_t u_1&= \bp \gamma(\eps) & \tau(\eps)\\
-\tau(\eps) & \gamma(\eps)\ep u_1  + \Pi N(\e, u_1, u_2),\\
\d_t u_2&= (1 - \Pi) L(\eps) u_2 + (1 - \Pi) N(\e,u_1, u_2),\\
u_{1|t=0} & = a, \\
u_{2|t=0} & = b,
\end{aligned}
\right.
\ee
where $$N(\e,u_1,u_2) := (\d_x \mathit{Q_f} + \mathit{Q_r})(\e, \bar U^\e, U).$$

 Given $T_0 > 0,$ there exist $\zeta_0 > 0$ and $C > 0,$ such that, if $|a| + \| b \|_{H^2_\o} < \zeta_0,$ the initial value problem \eqref{PHeq2}  possesses a unique solution 
$(u_1, u_2)(a,b,\e)\in C^0([0,T_0], \RR^2\times H^2_\o)$
 satisfying
\begin{equation}
\begin{aligned}\label{truncshort}
C^{-1}|a|- C\|b\|_{H^2_\o}^2&\le |u_1(t)|\le C(|a|+ \|b\|_{H^2_\o}^2),\\
\|u_2(t)\|_{H^2_\o}& \le C(\|b\|_{H^2_\o}+|a|^2),\\
\|\d_{(a,b)}(u_1, u_2)(t)\|_{{\cal L}(\R^2\times H^1,H^1)}& \le C.
\end{aligned}
\end{equation}
(For more details on the initial value problem \eqref{PHeq2} and estimate \eqref{truncshort}, see \cite{TZ2}, Proposition 4.2.)

\subsection{Poincar\'e return map} \label{sec:poinc}

We express the period map
$(a,b,\eps)\to \hat b := u_2(a,b,\eps, T)$
as a discrete dynamical system 
\begin{equation} 
\label{PHdyn}
\hat  b=  {\bf S}(\e,T) b+ \tilde N(a,b,\eps, T),
\end{equation}
where
$${\bf S}(\eps, T):=
e^{T (1 - \Pi)L(\eps)}$$
is the linearized
solution operator in 
$v$
and
$$%
\tilde N(a,b,\eps, T) := \int_0^{T} {\bf S}(\e,T-s)
(1 - \Pi) N(\e,u_1, u_2)(s)ds\\
$$
the 
difference between nonlinear and linear solution operators. 

Evidently, periodic solutions of (\ref{PHeq2}) with period $T$ 
are fixed
points of the period map (equilibria of (\ref{PHdyn})) or,
equivalently, zeros of the displacement map

$$%
\Delta (a,b,\eps, T) :=  ({\bf S}(\e,T) - \Id) b+ N_2(a,b,\eps, T).
$$

\subsection{Lyapunov-Schmidt reduction}

We now carry out a nonstandard Lyapunov--Schmidt reduction
following the ``inverse temporal dynamics'' framework of \cite{TZ2}, tailored for the
situation
that $\Id - {\bf S}(\e,T)$ is not uniformly invertible, or,
equivalently, the spectrum of $(1 - \Pi)L(\e)$ is not bounded away
from $\{j\pi/T\}_{j\in \ZZ}.$
In the present situation, $(1 - \Pi)L(\e)$ has both a $1$-dimensional
kernel (a consequence of (H4), see Section \ref{remarkson}) and essential spectra accumulating at $\lambda=0$, and no
other purely imaginary spectra, so that
$\Id - {\bf S}(\e,T)$ inherits the same properties;
see \cite{TZ2} for further discusssion.

Our goal, and the central point of the analysis, is to solve $\Delta(a,b,\e,T)=0$
for $b$ as
a function of $(a,\eps,T)$, eliminating the transverse variable
and reducing to a standard planar bifurcation problem in the oscillatory
variable $a$.
A ``forward'' temporal dynamics technique would be to rewrite
$\Delta=0$ as a fixed point map
\be\label{fwd0}
b =  {\bf S}(\e,T) b+ \tilde N(a,b,\eps, T),
\ee
then to substitute for $T$ an arbitrarily large integer multiple $jT$.
In the strictly stable case $\Re \sigma((1 - \Pi) L) \le - \theta<0$, $\|{\bf S}(\e,j T)\|_{{\cal L}(X_1)}<\frac{1}{2}$ for $j$ sufficiently large.
Noting that $\tilde N$ is quadratic in its dependency, we would have
therefore contractivity of (\ref{fwd0}) with respect to $b$,
yielding the desired reduction.
However, in the absence of a spectral gap between $\sigma((1 - \Pi) L)$
and the imaginary axis, $\|{\bf S}(\e,jT)\|_{{\cal L}(X_1)}$ does not
decay, and may be always greater than unity; thus, this naive
approach does not succeed.

The key idea in \cite{TZ2} is to rewrite $\Delta=0$ instead in ``backward'' form
\be\label{backwd}
b =  \left(\Id - {\bf S}(\e,T)\right)^{-1} \tilde N(a,b,\eps, T),
\ee
then show that $\left(\Id-{\bf S}(\e,T)\right)^{-1}$ is well-defined and bounded on $\Range \, \tilde N$,
thus obtaining contractivity by quadratic dependence of $\tilde N$.
Since the right inverse $\left(\Id-{\bf S}(\e,T)\right)^{-1}\tilde N$ is formally given by $\sum_{j=0}^\infty {\bf S}(\e,jT) \tilde N$
this amounts to establishing convergence:
a stability/cancellation estimate.
Quite similar estimates appear in the nonlinear stability theory,
where the interaction of linearized evolution ${\bf S}$ and nonlinear
source $\tilde N$ are likewise crucial for decay.
The formulation (\ref{backwd}) can be viewed also as a ``by-hand''
version of the usual proof of the standard Implicit Function Theorem \cite{TZ2}.

 \begin{lem}\label{returnsetup}
Under the assumptions of Theorem {\rm \ref{PH}}, if the constant $\eta$ in Definition {\rm \ref{B-spaces}} satisfies $\eta < \eta_0,$ where $\eta_0$ was introduced in Corollary {\rm \ref{profdecay},} then 
$$\tilde N: \quad (a,b,\eps,T) \in \RR^2\times X_1\times \RR^2 \to \tilde N(a,b,\eps,T)\in
X_2,$$
is quadratic order, and $C^1$ as a map from $\RR^2\times {\cal B}_1\times \RR^2$ to ${\cal B}_2$ for
$\|b\|_{X_1}$
uniformly bounded, with
\begin{equation} \label{Nbound}
\begin{aligned}\|\tilde N\|_{X_2} & \le C(|a|+\|b\|_{X_1})^2, \\
\|\partial_{(a,b)}
\tilde N\|_{{\cal L}(\R^2 \times \CalB_1,\CalB_2)}
& \le C(|a|+\|b\|_{H^2}).\\ 
\|\partial_{(\eps, T)} \tilde N\|_{{\cal L}(\R^2, {\cal B}_2)}
& \le C(|a|+\|b\|_{H^2})^2,
 \end{aligned}
\end{equation}
 \end{lem}

 \begin{proof} We use the variational bounds of \cite{TZ3} (see Propositions 5 and 6, \cite{TZ3}) and Lemma \ref{r-term}. Note that, in \eqref{ass-o}, only $\o^{-1} \in L^\infty$ and \eqref{ass-o}(ii) were used at this point. 
 \end{proof} 

\subsubsection{Pointwise cancellation estimate}\label{cancellation}

 We now develop the key cancellation estimates, adapting the pointwise semigroup
methods of \cite{ZH,MaZ3,Z2} to the present case. 

Our starting point is the inverse Laplace transform representation \eqref{inverseLT2}. 
Deforming the contour using analyticity of ${\cal G}_\lambda$
across oscillatory eigenvalues $\lambda_\pm(\eps)$
we obtain ${\cal G}=\tilde {\cal G} + {\cal O}$, where 
$$
\CalO(x,t;y):= e^{\l_+(\e) t}\varphi_+^\e(x) \tilde \varphi_+^{\e}(y)^{\tt tr}
+
e^{ \l_-(\e) t}\varphi_-^\e(x) \tilde \varphi_-^{\e}(y)^{\tt tr}
$$
is the sum of the residues
of the integrand at $\lambda_\pm$ (the right- and left-eigenfunctions $\varphi^\e_\pm$ and $\tilde \varphi^\e_\pm$ are defined in Section \ref{sec:coord}). The Green function $\tilde {\cal G}$ is the kernel of the integral operator ${\bf S}(\e,t)$ defined in Section \ref{sec:poinc}. Note that, under the assumptions of Theorem \ref{PH}, the Evans function associated with $(1 - \Pi) L(\e)$ satisfies \eqref{D}, so that, by Remark \ref{rem-greenbounds}, Proposition \ref{greenbounds} applies to $\tilde {\cal G}.$

For $\nu, \nu_0 > 0,$ let $\Gamma$ be the counterclockwise arc of circle $\partial B(0,r)$ ($r$ as in Proposition \ref{G_l}) connecting
$-\nu-i \nu_0$ and $-\nu+i \nu_0.$ If $\nu$ and $\nu_0 >0$ are sufficiently small, then $\Gamma$ is entirely contained in the resolvent set of $(1  - \Pi) L(\e),$ and 
$\tilde {\cal G}$ can be decomposed as ${\cal G}_{\rm I} + {\cal G}_{{\rm II}},
$
with
\ba \label{I-II}
 {\cal G}_{{\rm I}}(\e,x,t;y) & := \frac{1}{2\pi i} \int_\Gamma e^{\lambda t} {\cal G}_\lambda (\e,x,y) \,d\lambda\\
 {\cal G}_{{\rm II}}(\e,x,t;y) & := \frac{1}{2\pi i} \text{\rm P.V.}
\Big(
\int_{-\nu-i\infty}^{-\nu -i \nu_0}
+
\int_{-\nu+ i\nu_0}^{-\nu + i\infty}
\Big)
e^{\lambda t} {\cal G}_\lambda (\e,x,y) \, d\lambda.
\ea
 
Let ${\bf S}_{{\rm I}}$ and ${\bf S}_{{\rm II}}$ denote the integral operators with respective kernels ${\cal G}_{{\rm I}}$ and ${\cal G}_{{\rm II}},$ so that ${\bf S} = {\bf S}_{\rm I} + {\bf S}_{\rm II},$ 
 and let
$\Omega := (-\bar \e_0, \bar \e_0) \times (0,+\infty),$
for some $\bar \e_0 > 0.$ 

\begin{rem} \label{decayG_l} The contour $\Gamma$ being contained in the resolvent set of $L,$ the elementary bound holds:
 $$ | \d_y {\cal G}_\l| \leq C e^{-\theta_\nu |x-y|}, \qquad \l \in \Gamma,$$
 for some $\theta_\nu > 0$ depending on $\nu.$ See for instance Proposition 4.4, {\rm \cite{MaZ3}.}
 \end{rem} 

Our treatment of the high-frequency term follows \cite{TZ3}: 

\begin{lem} \label{hf-lem} Under the assumptions of Theorem {\rm \ref{PH}}, the sequence of operators with kernel
 $ \sum_{n=0}^N {\cal G}_{\rm II}(\e, nT)$
 is absolutely convergent in ${\cal L}(H^1),$ uniformly in $(\e,T) \in \Omega.$
 \end{lem}
 
 \begin{proof} Starting from the description of the resolvent kernel given in Proposition \ref{G_l}, we find by the same inverse Laplace transform estimates that give terms ${\cal H}$ and ${\cal R}$ in Proposition \ref{greenbounds}, that the high-frequency resolvent kernel ${\cal G}_{\rm II},$ defined in \eqref{I-II}, may be  expressed as
 \begin{equation}\label{g2} {\cal G}_{\rm II} = C e^{-\theta(|x-y| + t)} + h \t_{x + st} \delta,\end{equation}
 where $C$ and its space-time derivatives are bounded, $\theta > 0,$ and $h \t_{x + st} \delta$ is a generic hyperbolic term; in particular $h$ has the form \eqref{def-h} and satisfies \eqref{bds-R}(i). The Lemma follows. 
 \end{proof}
 
Next we turn to the low-frequency component of $\tilde {\cal G}.$ Its fluid terms are handled as in \cite{TZ3}: 

\begin{lem} \label{lf-lem} Under the assumptions of Theorem {\rm \ref{PH}}, the sequence of operators with kernel $\sum_{n=0}^N {\cal G}_{\rm I}(\e,nT)$ converges in ${\cal L}(\d_x L^1, H^1),$ uniformly with respect to $(\e,T) \in \O.$
   \end{lem}
   
\begin{proof} We argue as in the proof of Proposition 3 of \cite{TZ3}. Let $f \in L^1.$ By \eqref{I-II}, $\sum_{n=0}^{N-1} \d_y {\cal G}_{\rm I} f$ decomposes into ${\rm I} - {\rm II}_N,$ where 
 $$ \begin{aligned} {\rm I} & = \frac{1}{2 i \pi} \int_\R \int_\Gamma \frac{1}{1 - e^{\l T}} \d_y {\cal G}_\l f \, d\l \, dy, \\
   {\rm II}_N & = \frac{1}{2i \pi} \int_\R \int_\Gamma \frac{e^{NT \l}}{1 - e^{\l T}} \d_y {\cal G}_\l f\, d\l \, dy.
   \end{aligned}
   $$
 For small $\nu$ and $\l \in \Gamma,$ $(1 - e^{\l T})^{-1} = \l^{-1} T^{-1} (1 + O(\l)).$ 
 
 The boundary term ${\rm I}$ is independent of $N$ and is seen to belong to $H^1$ by Remark \ref{decayG_l}.

 By \eqref{E-}-\eqref{S+}, $\l^{-1} \d_y {\cal E}_\l$ and $\l^{-1} \d_y {\cal S}_\l$ have the same form as ${\cal E}_\l$ and ${\cal S}_\l.$ By Proposition \ref{G_l}, $\l^{-1} \d_y {\cal R}_\l$ behaves like the sum of ${\cal R}_\l$ and a pole term of form $\l^{-1} e^{-\theta |x-y|}.$ Hence, by the same Riemann saddle-point estimates used to bound ${\cal G}$ in Proposition \ref{greenbounds}, we find that
  \begin{equation} \label{decI-II} \int_\Gamma \frac{e^{\l N T}}{1 - e^{\l T}} \d_y {\cal G}_\l \, d\l = \big({\cal E} + {\cal S} + {\cal R}\big)(\e,NT),
  \end{equation}
  up to a constant (independent of $N$) term of the form $C e^{-\theta|x-y|},$ where the space-time derivatives of $C$ are uniformly bounded. This constant term satisfies the same bound as term ${\rm I}.$ 
  
  In \eqref{decI-II}, ${\cal E},$ ${\cal S},$ ${\cal R}$ denote generic excited, scattered and residual terms of form \eqref{def-E}, \eqref{S-}-\eqref{S+} and \eqref{bds-R}(ii).
By dominated convergence, 
 $$ \mbox{$H^1$-}\underset{N \to \infty}{\lim} \int_\R {\cal E}(\e,N T) f(y) \, dy$$
  exists and is equal to a sum of terms of the form 
  \begin{equation} \label{contrib-e} C(\e,T) (\bar U^\e)' \int_\R f(y) \, dy.
  \end{equation} Besides, by \eqref{S-}-\eqref{S+} and \eqref{bds-R},
 \begin{equation} \label{contrib-s} \left\| \int_\R ({\cal S} + {\cal R})(\e, NT) f(y) \, dy \right\|_{H^1} \leq C (NT)^{-frac{1}{4}} \| f \|_{L^1}.
 \end{equation} This proves convergence in $H^1$ of the sequence ${\rm II}_N.$  
\end{proof}

 We examine finally the contribution to the series $\sum_n {\bf S}(\e, nT)$ of the new (not present in \cite{TZ3}), reactive terms. 

\begin{lem} \label{r-lem} Under the assumptions of Theorem {\rm \ref{PH}}, 
 the sequence of operators with kernels $\sum_{n=0}^N {\cal G}_{\rm I}(\e,nT) \ell_4^{+}$ is absolutely convergent in ${\cal L}(L^1_{\eta^+}, H^1),$ uniformly with respect to $(\e,T) \in \O.$
  \end{lem}

\begin{proof} Let $f \in L^1.$ By \eqref{I-II}, Proposition \ref{greenbounds} and Corollary \ref{e-r}, the low-frequency kernel ${\cal G}_{\rm I}$ satisfies  
\begin{equation} \label{contrib-r}  \int_\R e^{- \eta y^+} {\cal G}_{\rm I}(\e,t) \ell_4^+ f(y) \, dy = \int_\R e^{-\eta y^+} \big({\cal S} + {\cal R}\big)(\e,t) \ell_4^+ f(y) \, dy,\end{equation}
and, by Corollary \ref{s-r},
$$ \left\| \int_\R e^{-\eta y^+} ({\cal S} + {\cal R})(\e,t) \ell_4^{+} f(y) \, dy \right\|_{H^1} \leq C (1 + t^{\frac{1}{4}}) e^{-\theta_1 t} \|f\|_{L^1},$$
 and the upper bound defines for $t = NT$ an absolutely converging series in $H^1.$
 \end{proof}

From Lemmas \ref{hf-lem}, \ref{lf-lem} and \ref{r-lem} and the fact that ${\bf S}(\e,T) \in {\cal L}({\cal B}_1),$ for all $(\e,T) \in \Omega,$ we can conclude that, under the assumptions of Theorem {\rm \ref{PH}}, the operator $\Id -{\bf S}(\e,T)$ has a right inverse 
$$
\left(\Id- {\bf S}(\e,T)\right)^{-1}:= \sum_{n \geq 0} {\bf S}(\e, n T): \qquad {\cal B}_2 \to {\cal B}_1,$$
that belongs to ${\cal L}({\cal B}_2, {\cal B}_1),$ locally uniformly in $(\e,T) \in \Omega.$ 

 We will need the following regularity result for the right inverse: 

\begin{lem} \label{der-lem} Under the assumptions of Theorem {\rm \ref{PH}}, the operator $(\Id- {\bf S}(\e,T))^{-1}$ is $C^1$ in $(\e, T) \in \O,$ with respect to the 
${\cal L}({\cal B}_2, {\cal B}_1)$ norm on ${\cal B}_2.$
\end{lem}

\begin{proof} Note that, by \eqref{inverseLT2}, $\d_t {\cal G}$ has kernel $\l {\cal G}_\l;$ in particular, the small $\l$ (low-frequency) estimates of the proofs of Lemmas \ref{lf-lem} and \ref{r-lem} imply the convergence of the sequence $\sum_{n=0}^N \d_T {\bf S}_{\rm I}(\e,nT)$ in ${\cal L}({\cal B}_2,{\cal B}_1).$ The contribution of $\d_T {\cal G}_{\rm II}(\e, n T)$ is handled as in Lemma \ref{hf-lem}, by \eqref{g2} and \eqref{bds-R}(i) with $k = 1.$

 Bounds for $\e$-derivatives are handled as in \cite{TZ3}, using either the variational equation $(L - \l) \d_\e {\cal G}_\l = - (\d_\e L) {\cal G}_\l,$ or the $\d_\e {\cal G}$ bounds of Proposition 3.11 from \cite{TZ2}.
\end{proof}

 Note that the $\e$-derivative bound \eqref{Nbound}(iii) is stated on a proper subspace of ${\cal B}_1,$ namely $X_1.$ In this respect, the following Lemma, asserting boundedness of the right inverse on $X_2 \hookrightarrow {\cal B}_2,$ in ${\cal L}(X_2,X_1)$ norm, is key to the reduction procedure of the following Section. (See Remark \ref{rem-fin}.)
 
\begin{lem} \label{crux-cor} Under the assumptions of Theorem {\rm \ref{PH}}, $(\Id- {\bf S}(\e,T))^{-1}$ belongs to ${\cal L}(X_2,X_1),$ for all $(\e,T) \in \Omega.$
\end{lem}

\begin{proof} The convolution bound 
\begin{equation} \label{cst-bd}
  \left\| \o^\frac{1}{2} \int_\R e^{-\theta |x-y|} f(y) \, dy \right\|_{L^2} \leq C \min\left(\| f \|_{L^2_\o}, \| f \|_{L^1_\o}\right),
  \end{equation}
  where $C$ depends on $\| \o^\frac{1}{2} e^{-\theta |\cdot|}\|_{L^1 \cap L^2},$ holds by \eqref{ass-o}(i) and \eqref{ass-o}(iii). It implies that the contributions of ${\cal G}_{\rm II},$ of ${\rm I}$ and of the constant pole term in ${\rm II}_N$ (see the proofs of Lemmas \ref{hf-lem} and \ref{lf-lem}) are all bounded in ${\cal L}(X_1).$
 
 The scattered and residual terms in ${\rm II}$ contribute nothing to the limit, by \eqref{contrib-s}.   

 We use again Corollary \ref{s-r} to handle the contribution of the reactive term. In \eqref{r-s-r}, there are two terms in the upper bound for $({\cal S} + {\cal R}) \ell_4^+.$ The first term is handled by \eqref{cst-bd}, and the second by  \begin{equation} \label{last?}
    e^{-\theta t} \left\| \o^\frac{1}{2} \int_\R e^{-|x-y|^2/Mt} f(y) \, dy \right\|_{L^2} \leq e^{-\theta t}\left\| \o^\frac{1}{2} e^{-|x-y|^2/Mt} \right\|_{L^2} \| f \|_{L^1_\o},
    \end{equation}
   noting that 
  \begin{equation} \label{o-gr} \begin{aligned} e^{-\theta t} \| \o^\frac{1}{2} e^{-|x-y|^2/Mt} \|_{L^2} & \leq C e^{-\theta_1 t} \left( \| \o^\frac{1}{2} e^{-|\cdot|^2/Mt} \|_{L^2_{|x| < t}} + \| \o^\frac{1}{2} e^{-|\cdot|^2/Mt} \|_{L^2_{t < |x|}} \right) \\
  & \leq C e^{-\theta_1 t} \left( t^\frac{3}{4} \o(t)^\frac{1}{2} + \| \o^\frac{1}{2} e^{-|\cdot|/M} \|_{L^2} \right), \end{aligned} \end{equation}  
  and the upper bound in \eqref{o-gr} defines for $t = NT$ an absolutely converging series, if $\theta_0$ in \eqref{ass-o} is small enough. We used in \eqref{o-gr} the growth assumption on $\o.$
  \end{proof} 
  
\begin{rem} \label{rem-o} The above proof shows that, in \eqref{ass-o}, we need in particular $\theta_0 < \frac{1}{2} \eta_0,$ where $\eta_0$ is the decay rate of the background profile (see Corollary {\rm \ref{profdecay}}). In addition, we need $\theta_0 < \min(\theta_\nu,\theta, \frac{2}{M}),$ where $\theta$ is the rate of decay in the upper bounds of Proposition {\rm \ref{greenbounds}} and $M$ is the diffusion rate in the upper bound of \eqref{r-s-r}{\rm (ii)}, and $\theta_0$ to be small enough so that $\varphi^\e_\pm \in H^2_\o.$
\end{rem}

\subsubsection{Reduction}

\begin{cor} \label{g+} Under the Assumptions of Theorem {\rm \ref{PH},} the equation
\begin{equation} \label{b} \Delta(a,b,\eps,T)=0
\qquad (a,b, \eps, T) \in \R^2 \times X_1 \times \RR^2,\end{equation}
where $\Delta$ is defined in Section {\rm \ref{sec:poinc}}, is equivalent to
$$b = \left(\Id - {\bf S}(\e,T)\right)^{-1} \tilde N(a,b,\e, T) +\omega$$
 for
$$ \omega \in \kernel(\Id -{\bf S}(\e,T)) \cap X_1.$$
\end{cor}

\begin{proof} A simple consequence of the definition in the above Section of the right inverse $\left(\Id - {\bf S}(\e,T)\right)^{-1}.$ For more details, see the proof of Lemma 2.8 in \cite{TZ2}.
\end{proof}

 Note that, as a consequence of (H4), Section \ref{remarkson}, the kernel of $\Id-{\bf S}(\e,T)$ is of dimension one, for all $\e,T,$ generated by $(\bar U^\e)'.$

\begin{cor}\label{boundedmap}
Under the Assumptions of Theorem {\rm \ref{PH},} the map
$$
\CalT(a,b,\eps,T,\alpha)
:=\left(\Id-{\bf S}(\e,T)\right)^{-1} \tilde N(a,b,\eps,T) + \alpha (\bar U^\e)',
$$
is bounded from $\RR^2\times X_1\times \RR^{2+1}$ to $X_1,$ and $C^1$ from $\RR^2\times {\cal B}_1 \times \RR^{2+1}$ to $X_1,$ for
$|\alpha|$ bounded and $|a|+\|b\|_{X_1}+|\eps|$ sufficiently small, with
\ba  \nonumber 
\|\CalT\|_{X_1}&\le C(|a|+ \|b\|_{X_1}^2),\\
\|\partial_{(a,b)} {\cal T}\|_{{\cal L}(\R^2 \times {\cal B}_1, {\cal B}_1)}&
\le C(|a|+ \|b\|_{X_1}),\\
\|\partial_{T} {\cal T}\|_{{\cal B}_1}&
\le C(|a|^2+ \|b\|_{X_1}^2),\\
\|\partial_{\eps} \CalT\|_{{\cal B}_1}&
\le C(|a|^2+ \|b\|_{X_1}^2 +|\alpha|),\\
\|\partial_{\alpha} \CalT\|_{{\cal B}_1}&
\le C(|a|^2+ \|b\|_{X_1}^2 +1 ).\\
\ea
\end{cor}

\begin{proof}
Follows from Lemma \ref{returnsetup}, the results of Section \ref{cancellation}, and the above remark on the kernel of $\Id-{\bf S}(\e,T).$ 
\end{proof}

\begin{rem} \label{rem-fin} Without Lemma {\rm \ref{crux-cor}} but with Lemmas {\rm \ref{hf-lem}} to {\rm \ref{der-lem}}, we could see ${\cal T}$ has a $C^1$ map from ${\cal B}_1$ to ${\cal B}_1$ (for fixed $a,\e,T,\a$), with quadratic bound $\| {\cal T}\|_{{\cal B}_1} \leq C \| b \|_{{\cal B}_1}^2,$ and derivative bound $\| \d_b {\cal T}\|_{{\cal L}({\cal B}_1} \leq C \| b \|_{{\cal B}_1}.$ This would be sufficient to prove existence of a fixed point $b \in {\cal B}_1$ as in the following Proposition, but not to prove regularity of the fixed point with respect to $\e,$ precisely because the $\e$-derivative bound \eqref{Nbound}{\rm (iii)} requires more regularity in $b,$ by $\d_\e e^{t L(\e)} = e^{t L(\e)} \d_\e L$ (see also Remark 3.12, {\rm \cite{TZ2}}). Without regularity of the fixed point of ${\cal T},$ we could not use the standard implicit function theorem in Section {\rm \ref{sec:fin}.} These issues were discussed in detail in Section 2.3 of {\rm \cite{TZ2}}. 
 \end{rem}

\begin{prop}\label{LSprop} Under the assumptions of Theorem {\rm \ref{PH}},
there exists a function $\b (a,\eps, T, \alpha)$, bounded
from $\RR^{4+1}$ to $X_1$ and $C^1$ from $\RR^{4+1}$ to ${\cal B}_1$,
such that
$$\Delta(a, \b(a,\eps,T,\alpha), \eps, T)\equiv 0,$$
\ba\label{betabds}
\|\beta\|_{X_1} +
\|\partial_{(\eps, T)}\beta\|_{{\cal L}(\R^2, {\cal B}_1)}&\le C(|a|^2 + |\alpha|),\\
\|\partial_{a}\beta\|_{{\cal L}(\R^2, {\cal B}_1)}&\le C|a|,\\
\|\partial_{\a}\beta\|_{{\cal L}(\R, {\cal B}_1)}&\le C,\\
\ea
for $|(a,\eps,\alpha)|$ sufficiently small.
Moreover, for $|(a,\eps)|$, $\|b\|_{X_1}$ sufficiently small,
all solutions of {\rm (\ref{b})} lie on the $1$-parameter manifold
$$\{b=\beta(a,\eps,T,\alpha), \quad \a \in \R\}.$$
\end{prop}

\begin{proof} A consequence of the Banach fixed point theorem applied to map ${\cal T}.$ For more details, see the proof of Proposition 2.9 in \cite{TZ2}. 
\end{proof}

\subsubsection{Bifurcation} \label{sec:fin}

The bifurcation analysis is straightforward
now that we have reduced to a finite-dimensional problem,
the only tricky point being to deal with the $1$-fold multiplicity of solutions
(parametrized by $\alpha$).

Define to this end
$$\tilde \beta(a,\eps,T,\hat \alpha):=
\beta(a,\eps,T,|a|\hat \alpha),$$
with $\hat \alpha$ restricted to a ball in $\R^1$,
noting, by (\ref{betabds}), that
$$
\|\tilde\beta\|_{X_1}, \,
\|\partial_{(a,\eps, T, \hat \a)}\tilde\beta\|_{{\cal L}(\R^4, {\cal B}_1)} \le C|a|,
$$
with $\tilde \beta$ Lipshitz in $(a,\eps,T, \hat \a)$ and $C^1$ away from $a=0$. %

Solutions $(u_1,u_2)$ of (\ref{PHeq2}) originating at 
 \begin{equation} \label{ci-beta} 
  (a,b)=(a,\tilde \b(a,\e,T,\hat a)),
  \end{equation}
by (\ref{truncshort}), remain for $0\le t\le T$ in a cone
$$\CalC:=\{(u_1,u_2):\, |u_2|\le C_1 |u_1|\},$$
$C_1>0$. 
Indeed, \eqref{truncshort} implies the bound 
$$\|u_2(t)\|_{X_1}\le C|u_1(t)|, \quad \mbox{for $\|b\|_{X_1}\le C_1|a|$,}$$
for $C_1 > 0$ small enough, for all $0\le t\le T.$

Likewise, any periodic solution of (\ref{PHeq2}) originating
in $\CalC$, since it necessarily satisfyies $\Delta=0$, must originate from data
$(a,b)$ of the form \eqref{ci-beta}.

Defining $b\equiv \tilde \beta(a,\eps,T,\hat \a)$, and recalling
invariance of $\CalC$ under flow (\ref{PHeq2}), we may view
$v(t)$ as a multiple
\begin{equation} \label{multiple}
u_2(x,t)=c(a,\eps,T, \hat \a, x, t) u_1(t)\end{equation}
of $u_1(t)$, where $c$ is bounded, Lipschitz in all arguments,
and $C^1$ away from $a=0$.
Substituting into (\ref{PHeq2})(i), we obtain a planar ODE 
$$
\d_t u_1 = \bp \gamma(\eps) & \tau(\eps)\\
-\tau(\eps) & \gamma(\eps)\ep u_1 + M(u_1,\eps, T, t, \hat \a, a)
$$
in approximate Hopf normal form, with nonlinearity $M:= \Pi N$ now
nonautonomous and depending on the additional parameters $(T, \hat \a,a)$, but, by \eqref{Qbound} and \eqref{qbound}, still
satisfying the key bounds
\ba\label{tildeNbds}
|M|, \,
|\partial_{\eps, T, \hat \a} M|&\le C|u_1|^2;
\quad
|\partial_{a, w} M|&\le C|u_1|
\ea
along with planar bifurcation criterion \eqref{nondeg}.
{}From (\ref{tildeNbds}), we find that $M$ is $C^1$ in all
arguments, also at $a=0$.
By standard arguments (see, e.g., \cite{HK,TZ1}), we thus obtain a classical Hopf
bifurcation in the variable $u_1$
with regularity $C^1$, yielding existence and uniqueness up to
time-translates of a $1$-parameter family of solutions originating in $\CalC$,
indexed by $r$ and $\delta$
with $r:=a_1$ and (without loss of generality) $a_2 \equiv 0$. Bound \eqref{expdecaybd} is a consequence of \eqref{truncshort}(i) and \eqref{multiple}. 

Finally, in order to establish uniqueness up to spatial translates, we observe, first, that, by dimensional considerations,
the one-parameter family constructed must agree with the one-parameter
family of spatial translates, and second, we argue as in \cite{TZ2}
that any periodic solution has a spatial translate originating
in $\CalC$, yielding uniqueness up to translation among all solutions
and not only those originating in $\CalC$; %
see Proposition 2.20 and Corollary 2.21 of \cite{TZ2} for further details.
\section{Nonlinear instability: Proof of Theorem \ref{instab}}\label{instabproof}

We describe a nonlinear instability result in general setting.
Consider
\begin{equation}\label{evolution}
\d_t U = LU + \d_x N(U) + R(U),
\end{equation}
well-posed in $H^s$, where 
$$
L= \d_x (B \d_x U) + \d_x (AU) + GU, 
$$
and $|N(U)|, |R(U)|\le C|U|^2$ for $|U|\le C$.
Suppose that $L$ has a conjugate pair of simple unstable eigenvalues
$\lambda_\pm=\gamma \pm i\tau$, $\gamma >0$, and the rest of the
spectrum is neutrally stable, without loss of generality
$e^{(1 - \Pi)L t}\le Ct$, where $\Pi$ is
the projection onto the eigenspace associated with $\lambda_\pm$.

Coordinatizing similarly as in Section \ref{bifsection} by
$$
 U(x,t)=u_{11} \varphi_1(x) + u_{12} \varphi_2(x) + u_2(x,t),
$$
where $\varphi_j=O(e^{-\theta |x|})$ are eigenfunctions of $L$,
denote $r(t):=|u_1|(t)$.
Then, so long as $|U|_{H^s}\le C\epsilon$, we have existence (by
variation of constants, standard continuation) of solutions
of \eqref{evolution} in $H^s$, with estimates
\begin{equation}\label{setup}
\begin{aligned}
r'=\gamma r + O(\epsilon)|U|,\\
u_2'= (1 - \Pi) L u_2 + O(\epsilon)|U|\\
\end{aligned}
\end{equation}
in $L^2$.

We shall argue by contradiction.  That is, using \eqref{setup},
we shall show, for $C>0$ fixed, $\epsilon>0$ sufficiently small, and
$|u_2(0)|_{H^1}\le Cr(0)$, that eventually $r(t)\ge \epsilon$, no matter
how small is $r(0)$, or equivalently $|U|_{H^s}(0)$.
This, of course entails nonlinear instability.

Define $\alpha(x,t):= u_2(x,t)/r(t)$.
Then,
$$
\begin{aligned}
\alpha'&= \frac{ru_2'-u_2 r'}{r^2}=
\frac{u_2'}{r} - \frac{u_2}{r}\frac{r'}{r},\\
\end{aligned}
$$
yielding after some rearrangement the equation
\begin{equation}\label{alpha}
\alpha'=((1 - \Pi) L-\gamma)\alpha + O(\epsilon \big(e^{-\theta|x|}+ |\alpha|
+|\alpha|^2)\big).
\end{equation}
{}From \eqref{alpha} and standard variation of constants/contraction mapping
argument, we find that $|\alpha(t)|_{H^1}$ remains less than or equal to 
$C|\alpha(t_0)|_{H^1}$ for $t-t_0$ small.

By variation of constants and the semigroup bound
$|e^{((1 - \Pi) L-\gamma)t}|_{H^1 \to H^1}\le Ce^{-\gamma t}$ (note: $\gamma$
is scalar so commutes with $(1 - \Pi)L$), we obtain
$$
\delta(t)\le C(|\alpha(0)|_{H^1} + \epsilon (1+\delta(t)^2) ),
$$
for $\delta(t):=\sup_{0\le \tau \le t} |\alpha(\tau)|_{H^1}$.
So long as $\delta$ remains less than or equal to unity and
$C\epsilon \le \frac{1}{2}$, this yields
$$
\delta(t)\le 2C(|\alpha(0)|_{H^1} + \epsilon ),
$$
and thus 
$$
\delta(t)\le  2C|\alpha(0)|_{H^1} + \frac{1}{2}.
$$

Substituting into the radial equation, we obtain
$$
r'\ge (\gamma -(1+\delta)\epsilon)r,
$$
yielding exponential growth for $\epsilon$ sufficiently small.
In particular, $r\ge C\epsilon$ for some time, and thus $|U|_{H^1}\ge \epsilon$,
a contradiction.
We may conclude, therefore, that $|U|_{L^2}$ eventually grows larger than any
$\epsilon$, no matter how small the initial size $r(0)$, and thus we may
conclude {\it instability} of the trivial solution $U\equiv 0$.
Taking now \eqref{evolution} to be the perturbation equations about
a strong detonation profile, we obtain the result of nonlinear instability
of the background profile $\bar U$.

\br\label{scal}
In the easier case of a single, real eigenvalue, the scalar, $w$ equation,
would play the role of the radial equation here. This case is subsumed in
our analysis as well.
\er


\begin{thebibliography}{GMWZ2}

{\footnotesize 

\bibitem [AT]{AT} G. Abouseif and T.Y. Toong,
{\it Theory of unstable one-dimensional detonations,}
Combust. Flame 45 (1982) 67--94.

\bibitem [AGJ] {AGJ} J. Alexander, R. Gardner and C.K.R.T. Jones,
{\it A topological invariant arising in the analysis of
traveling waves}, J. Reine Angew. Math. 410 (1990) 167--212.

\bibitem [AlT] {ALT} R. L. Alpert and T.Y. Toong,
{\it Periodicity in exothermic hypersonic flows
about blunt projectiles,} Acta Astron. 17 (1972) 538--560.

\bibitem[BHRZ]{BHRZ} B. Barker, J. Humpherys, , K. Rudd, and K. Zumbrun,
%
{\it Stability of viscous shocks in isentropic gas dynamics},
 Comm. Math. Phys.  281  (2008),  no. 1, 231--249.
%

\bibitem[Ba]{Ba} G.K. Batchelor,
{\it An introduction to fluid dynamics,}
Second paperback edition. Cambridge Mathematical
Library. Cambridge University Press, Cambridge (1999) xviii+615
pp. ISBN: 0-521-66396-2.

\bibitem [BeSZ]{BeSZ} M. Beck, B. Sandstede, and K. Zumbrun,
{\it Nonlinear stability of time-periodic shocks,}
Archive for Rational Mechanics and Analysis 196 (2010) 1011-1076.


\bibitem[BM]{BM} A. Bourlioux and A. Majda, 
\emph{Theoretical and numerical structure of unstable 
detonations.} Proc. R. Soc. Lond. A (1995) 350, 29-68.

\bibitem[BMR]{BMR} A. Bourlioux, A. Majda, and V. Roytburd,
\emph{Theoretical and numerical structure for unstable one-dimensional
detonations.} SIAM J. Appl. Math. 51 (1991) 303--343.


\bibitem[BDG]{BDG} 
T.J. Bridges, G. Derks, and G. Gottwald,
\textit{Stability and instability of solitary waves of the
fifth-order KdV equation: a numerical framework.}
Phys. D  172  (2002),  no. 1-4, 190--216.

\bibitem[Br1]{Br1} L. Brin,
{\it Numerical testing of the stability of viscous shock waves}.
Doctoral thesis, Indiana University (1998).

\bibitem[Br2]{Br2} L. Q. Brin,
{\it Numerical testing of the stability of viscous shock waves.}
Math. Comp. 70 (2001) 235, 1071--1088.

\bibitem[BrZ]{BrZ} L. Brin and K. Zumbrun,
{\it Analytically varying eigenvectors and the stability of viscous
shock waves}. Seventh Workshop on Partial Differential Equations, Part I (Rio de Janeiro, 2001).
Mat. Contemp. 22 (2002), 19--32.

\bibitem [B]{B} J.D. Buckmaster, {\it An introduction to combustion theory,}
3--46,
in {\it The mathematics of combustion,} Frontiers in App. Math.
(1985) SIAM, Philadelphia ISBN: 0-89871-053-7.

\bibitem[BN]{BN} J. Buckmaster and J. Neves,
{\it One-dimensional detonation stability: the spectrum for infinite 
activation energy,} 
Phys. Fluids 31 (1988) no. 12, 3572--3576.



\bibitem [C]{C} J. Carr,
{\it Applications of centre manifold theory.} Applied Mathematical Sciences, 35. Springer-Verlag, New York-Berlin, 1981. vi+142 pp. ISBN: 0-387-90577-4.

%
%
%
%
%

\bibitem[Ch]{Ch} G. Q. Chen, {\it Global solutions to the compressible Navier-Stokes equations for a reacting mixture.}
SIAM J. Math. Anal. 23 (1992), no. 3, 609--634. 

\bibitem[CF]{CF} R. Courant and K.O. Friedrichs,
{\it Supersonic flow and shock waves,}
Springer--Verlag, New York (1976) xvi+464 pp.

\bibitem[EE]{EE} D.E. Edmunds and W.D. Evans, {\it Spectral theory and differential operators}. Oxford University Press (1987).

\bibitem [Er1]{Er1} J. J. Erpenbeck,
{\it Stability of steady-state equilibrium detonations,}
Phys. Fluids 5 (1962),
%
604--614.

\bibitem [Er2]{Er2} J. J. Erpenbeck,
{\it Stability of idealized one-reaction detonations,}
Phys. Fluids 7 (1964).

\bibitem [Er3]{Er3} J. J. Erpenbeck,
{\it Detonation stability for disturbances of small transverse wave length},
Phys. Fluids 9 (1966) 1293--1306.


\bibitem [Er4]{Er4} J. J. Erpenbeck,
{\it Nonlinear theory of unstable
one--dimensional detonations,}
Phys. Fluids 10 (1967) No. 2, 274--289.

\bibitem[F1] {F1} W. Fickett,
{\it Stability of the square wave detonation in a model
system.}  Physica 16D (1985) 358--370.

\bibitem [F2]{F2} W. Fickett, {\it Detonation in miniature,} 133--182,
in {\it The mathematics of combustion,} Frontiers in App. Math.
(1985) SIAM, Philadelphia ISBN: 0-89871-053-7.

\bibitem [FD] {FD} W. Fickett and W.C. Davis,
{\it Detonation,} University of California Press, Berkeley, CA (1979):
reissued as {\it Detonation: Theory and experiment,}
Dover Press, Mineola, New York (2000), ISBN 0-486-41456-6.

\bibitem [FW] {FW} Fickett and Wood,
{\it Flow calculations for pulsating one-dimensional
detonations.} Phys. Fluids 9 (1966) 903--916.

%
%
%
%
%
%
%
%
%
%


\bibitem[G]{G} R. Gardner,
{\it On the detonation of a combustible gas,}
Trans. Amer. Math. Soc.  277  (1983), no. 2, 431--468.


\bibitem[GK]{GK} I. Gohberg and M.G. Krein, {\it Introduction to the theory of linear nonselfadjoint operators.} American Mathematical Society. Translations of mathematical monographs volume 18 (1969). 


\bibitem[GZ]{GZ}  R. Gardner and K. Zumbrun,
{\it The Gap Lemma and geometric criteria for instability
of viscous shock profiles}, Comm. Pure Appl. Math.  51  (1998),  no. 7, 797--855. 

\bibitem[GS1]{GS1} I. Gasser and P. Szmolyan,
{\it A geometric singular perturbation analysis of detonation
and deflagration waves,} SIAM J. Math. Anal. 24 (1993) 968--986.

\bibitem[GS2]{GS2} I. Gasser and P. Szmolyan,
{\it Detonation and deflagration waves with multistep reaction schemes,}
SIAM J. Appl.  Math. 55 (1995) 175--191.

%
%
%

\bibitem [HK]{HK} J. Hale and H. Ko\c{c}ak,
{\it Dynamics and bifurcations,}
Texts in Applied Mathematics, 3. Springer-Verlag, New York, 1991.
xiv+568 pp. ISBN: 0-387-97141-6.


\bibitem[He]{He} D. Henry,
{\it Geometric theory of semilinear parabolic equations}.
Lecture Notes in Mathematics volume 840, Springer--Verlag, Berlin (1981),
iv + 348 pp.

\bibitem[HZ]{HZ} P. Howard and K. Zumbrun,
\emph{Stability of undercompressive viscous shock waves}, in
press, J. Differential Equations  225  (2006),  no. 1, 308--360.

\bibitem[HLZ]{HLZ} J. Humpherys, O. Lafitte, and K. Zumbrun,
{\it Stability of viscous shock profiles in the high Mach number  
limit,} Comm. Math. Phys.  293  (2010),  no. 1, 1--36.

\bibitem[HLyZ]{HLyZ} J. Humpherys, G. Lyng, and K. Zumbrun,
{\it Spectral stability of ideal-gas shock layers,}
Arch. Ration. Mech. Anal.  194  (2009),  no. 3, 1029--1079.

\bibitem[HuZ1]{HuZ1} J. Humpherys and K. Zumbrun,
{\it An efficient shooting algorithm for Evans function calculations 
in large systems,}  Phys. D  220  (2006),  no. 2, 116--126.

\bibitem[HuZ2]{HuZ2} J. Humpherys and K. Zumbrun,
{\it Spectral stability of small amplitude
shock profiles for dissipative symmetric hyperbolic--parabolic systems.}
Z. Angew. Math. Phys. 53 (2002) 20--34.


\bibitem[JLW]{JLW} H.K. Jenssen, G. Lyng, and M. Williams,
{\it Equivalence of low-frequency stability conditions for 
multidimensional detonations in three models of combustion,}
Indiana Univ. Math. J.  54  (2005),  no. 1, 1--64. 


%
%
%
%
%
%
%
%

%
%
%
%
%
%
%
%
%
%
%
%
%
%
%
%
%
%
%
%

\bibitem[Kat]{Kat} T. Kato,
{\it Perturbation theory for linear operators}.
Springer--Verlag, Berlin Heidelberg (1985).

\bibitem[KS]{KS} A.R. Kasimov and D.S. Stewart,
{\it Spinning instability of gaseous detonations.}
J. Fluid Mech. 466 (2002), 179--203.

%
%
%
%
%
%

\bibitem[LS]{LS} H. I. Lee and D. S. Stewart,
{\it Calculation of linear detonation instability:
one-dimensional instability of plane detonation,}
J. Fluid Mech. 216 (1990) 102--132.


\bibitem[LyZ1]{LyZ1} G. Lyng and K. Zumbrun,
{\it A stability index for detonation waves in Majda's model 
for reacting flow.}  Phys. D  194  (2004),  no. 1-2, 1--29.


\bibitem[LyZ2]{LyZ2} G. Lyng and K. Zumbrun,
{\it One-dimensional stability of viscous strong detonation waves.}
Arch. Ration. Mech. Anal.  173  (2004),  no. 2, 213--277.

\bibitem[LRTZ]{LRTZ} G. Lyng, M. Raoofi, B. Texier, and K. Zumbrun,
\emph{Pointwise Green Function Bounds and stability of combustion waves},
Journal of Differential Equations 233 (2007), no. 2, 654-698.


\bibitem[MM]{MM} J.~E.~Marsden, M.~McCracken, {\it The Hopf bifurcation and its applications}, Applied Mathematical Sciences 19, Springer. 


\bibitem[MaZ1]{MaZ1} C. Mascia and K. Zumbrun,
{\it Pointwise Green's function bounds and stability of relaxation shocks}.
Indiana Univ. Math. J. 51 (2002), no. 4, 773--904.

\bibitem[MaZ2]{MaZ2} C. Mascia and K. Zumbrun,
\emph{Stability of small-amplitude shock profiles of symmetric
hyperbolic-parabolic systems,}
Comm. Pure Appl. Math.  57  (2004),  no. 7, 841--876.

\bibitem[MaZ3]{MaZ3} C. Mascia and K. Zumbrun,
\emph{Pointwise Green function bounds for shock profiles of
systems with real viscosity,}
Arch. Ration. Mech. Anal.  169  (2003),  no. 3, 177--263.

\bibitem[MaZ4]{MaZ4} C. Mascia and K. Zumbrun,
\emph{Stability of large-amplitude viscous shock profiles
of hyperbolic-parabolic systems,}
Arch. Ration. Mech. Anal.  172  (2004),  no. 1, 93--131.

\bibitem[MaZ5]{MaZ5} C. Mascia and K. Zumbrun,
\textit{Stability of large-amplitude shock profiles
of general relaxation systems,}
SIAM J. Math. Anal.  37  (2005),  no. 3, 889--913.

\bibitem[MeZ]{MeZ} G. M\'etivier and K. Zumbrun, {\it Large viscous boundary layers for noncharacteristic nonlinear hyperbolic problems},  Mem. Amer. Math. Soc.  175  (2005),  no. 826, vi+107 pp.

\bibitem[Pa]{Pa} A. Pazy, {\it Semigroups of linear operators and applications to partial differential equations.}
Applied Mathematical Sciences, 44. Springer-Verlag, New York, 1983. viii+279 pp.
%
%
%
%

\bibitem [MT]{MT} U.B. McVey and T.Y. Toong,
{\it Mechanism of instabilities in
exothermic blunt-body flows,} Combus. Sci. Tech. 3 (1971) 63--76.


\bibitem[RZ]{RZ} R. Raoofi and K. Zumbrun,
{\it Stability of undercompressive viscous shock profiles of
hyperbolic--parabolic systems,}
Journal of Differential Equations
Volume 246, Issue 4, 15 February 2009, Pages 1539-1567.

\bibitem[SS]{SS} B. Sandstede and A. Scheel,
{\it Hopf bifurcation from viscous shock waves,} 
SIAM J. Math. Anal. 39 (2008), no. 6, 2033--2052. 
%
%

\bibitem[ShK]{ShK} S. Shizuta and Y. Kawashima, {\it On the normal form of the symmetric hyperbolic-parabolic systems associated with the conservation laws.}
Tohoku Math. J. (2) 40 (1988), no. 3, 449--464. 

\bibitem[S1]{S1} M. Short,
{\it An Asymptotic Derivation of the Linear Stability of
the Square-Wave Detonation using the Newtonian limit,}
Proc. R. Soc. Lond. A (1996) 452, 2203-2224. 

\bibitem[S2]{S2} M. Short,
{\it Multidimensional linear stability of a detonation
wave at high activation energy,}
Siam J. Appl. Math. 57 (1997), No. 2, 307--326.

\bibitem[TT]{TT} D. Tan and A. Tesei,
{\it Nonlinear stability of strong detonation waves in
gas dynamical combustion,}  Nonlinearity, 10(1997), pp. 355-376.

%
%
%
%

\bibitem[TZ1]{TZ1} B. Texier and K. Zumbrun, {\it Relative
Poincar\'e--Hopf bifurcation and galloping instability
of traveling waves,}  Methods Appl. Anal.  12  (2005),  no. 4, 349--380. 

\bibitem[TZ2]{TZ2} B. Texier and K. Zumbrun, {\it 
Galloping instability of viscous shock waves.} Physica D 237 (2008), no. 10-12, 1553--1601. 

\bibitem[TZ3]{TZ3} B. Texier and K. Zumbrun, {\it 
Hopf bifurcation of viscous shock waves in
compressible gas-dynamics and MHD.} 
Arch. Ration. Mech. Anal. 190 (2008), no. 1, 107--140.

\bibitem [VT]{VT} A. Vanderbauwhede and G. Iooss,
{\it Center manifold theory in infinite dimensions,}
Dynamics reported: expositions in dynamical systems,  
125--163, Dynam. Report. Expositions Dynam. Systems (N.S.), 
1, Springer, Berlin, 1992. 
 
%
%
%
%
%
%

\bibitem[Z1]{ZKochel} K. Zumbrun, {\it Multidimensional stability of
planar viscous shock waves}.
Advances in the theory of shock waves, 307--516,
Progr. Nonlinear Differential Equations Appl., 47, Birkh\"{a}user Boston,
Boston, MA, 2001.

\bibitem[Z2]{Z2} K. Zumbrun, {\it Stability of large-amplitude shock
waves of compressible Navier--Stokes equations.} Handbook of mathematical fluid dynamics. Vol. III, 311--533, North-Holland, Amsterdam, 2004. 

\bibitem[Z3]{Z3} K. Zumbrun, 
{\it Planar stability criteria for 
viscous shock waves of systems with real viscosity,}
Hyperbolic systems of balance laws,  229--326,
Lecture Notes in Math., 1911, Springer, Berlin, 2007. 

\bibitem[Z4]{Z4} K. Zumbrun, {\it Stability of viscous detonations
in the ZND limit,} to appear, Arch. Ration. Mech. Anal.

\bibitem[ZH]{ZH} K. Zumbrun and P. Howard,
\textit{Pointwise semigroup methods and stability of viscous shock waves}.
Indiana Mathematics Journal V47 (1998), 741--871;
Errata, Indiana Univ. Math. J.  51  (2002),  no. 4, 1017--1021.

\bibitem[ZS]{ZS} K. Zumbrun and D. Serre,
{\it Viscous and inviscid stability of multidimensional 
planar shock fronts,} Indiana Univ. Math. J. 48 (1999) 937--992.

}
\end{thebibliography}
\end{document}